\documentclass[a4paper,12pt]{amsart}

\usepackage{amsmath,amsxtra,amssymb,latexsym, amscd,amsthm,dsfont,subcaption,cases}
\usepackage[mathscr]{eucal}
\usepackage{mathrsfs,bbm,enumerate,calc,yhmath,mathtools,faktor}
\usepackage{stackengine}
\usepackage{pict2e}
\usepackage{tikz}
\usepackage{graphicx}
\usepackage{color}
\usepackage[foot]{amsaddr}

\usepackage[a4paper]{geometry}
\usepackage{hyperref}
\hypersetup{colorlinks=true, breaklinks=true, urlcolor= blue, linkcolor= black, citecolor=teal,
  pdftitle={Friedman-Ramanujan functions}, 
  pdfauthor={Nalini Anantharaman and Laura Monk}, 
  pdfsubject={}}

\usepackage[capitalize]{cleveref}
\crefname{equation}{equation}{equations}
\Crefname{equation}{Equation}{Equations}
\crefname{figure}{Figure}{Figures}
\Crefname{figure}{Figure}{Figures}

\numberwithin{equation}{section}

\title[Friedman--Ramanujan functions in random hyperbolic
geometry I]{Friedman--Ramanujan functions \protect\\ in random hyperbolic geometry
  \protect\\ and application to spectral gaps I}

\author{Nalini Anantharaman\textsuperscript{1} and Laura Monk\textsuperscript{2}}

\address[1]{Coll\`ege de France, 11 place Marcelin Berthelot, 75005 Paris / IRMA, 7 rue Ren\'e Descartes, 67084 Strasbourg Cedex, France} 

\address[2]{School of Mathematics, University of Bristol, Bristol BS8 1UG, U.K.}

\email{nalini.anantharaman@college-de-france.fr}
\email{laura.monk@bristol.ac.uk}

\makeatletter
\@namedef{subjclassname@2020}{\textup{2020} Mathematics Subject Classification}
\makeatother



\date{\today}

\theoremstyle{plain}
\newtheorem{thm}{Theorem}[section]
\newtheorem{prp}[thm]{Proposition}
\newtheorem{cor}[thm]{Corollary}
\newtheorem{lem}[thm]{Lemma}

\newtheorem*{namedthm}{\namedthmname}
\newcounter{namedthm}

\makeatletter
\newenvironment{named}[1]
  {\def\namedthmname{#1}%
   \refstepcounter{namedthm}%
   \namedthm\def\@currentlabel{#1}}
  {\endnamedthm}
\makeatother

\theoremstyle{definition}
\newtheorem{defa}[thm]{Definition}

\newtheorem{rem}[thm]{Remark}

\newtheorem{exa}[thm]{Example}
\newtheorem{nota}[thm]{Notation}

\renewcommand{\d}{\, \mathrm{d}}

\makeatletter
\newcommand*{\ov}[1]{%
  $\m@th\overline{\mbox{#1}}$%
}
\newcommand*{\ovA}[1]{%
  $\m@th\overline{\mbox{#1}\raisebox{3mm}{}}$%
}
\newcommand*{\ovB}[1]{%
  $\m@th\overline{\mbox{#1\rule{0pt}{3mm}}}$%
}
\newcommand*{\ovC}[1]{%
  $\m@th\overline{\mbox{#1\strut}}$%
}
\newcommand*{\ovD}[1]{%
  $\m@th\overline{\mbox{#1\vphantom{\"A}}}$%
}
\newcommand*{\ovE}[1]{%
  $\m@th\overline{\raisebox{0pt}[1.2\height]{#1}}$%
}
\newcommand*{\ovF}[1]{%
  $\m@th\overline{\raisebox{0pt}[\dimexpr\height+0.3mm\relax]{#1}}$%
}
\newcommand*{\ovG}[1]{%
  $\m@th\overline{\raisebox{0pt}[\dimexpr\height+1mm\relax]{#1\vphantom{A}}}$%
}
\makeatother

\newcommand\runderset[2][\sim]{\mathrel{\ensurestackMath{%
  \stackengine{-.2pt}{\scriptscriptstyle#2}{#1}{O}{c}{F}{F}{S}}}}

\newcommand{\Z}{\mathbb{Z}}

\newcommand{\R}{\mathbb{R}}
\newcommand{\C}{\mathbb{C}}

\newcommand{\D}{\mathcal{D}}

\DeclareMathOperator{\dist}{dist}

\DeclareMathOperator{\tr}{tr}

\DeclareMathOperator{\id}{id}

\DeclareMathOperator\argch{argch}
\DeclareMathOperator\argcosh{argch}

\DeclareSymbolFont{extraup}{U}{zavm}{m}{n}
\DeclareMathSymbol{\varheart}{\mathalpha}{extraup}{86}
\DeclareMathSymbol{\vardiamond}{\mathalpha}{extraup}{87}

\newcommand{\nwc}{\newcommand}

\nwc{\mf}{\mathbf} 
\nwc{\blds}{\boldsymbol} 
\nwc{\ml}{\mathcal} 

\nwc{\lam}{\lambda}
\nwc{\del}{\delta}
\nwc{\Del}{\Delta}
\nwc{\Lam}{\Lambda}
\nwc{\elll}{\ell}

\nwc{\IA}{\mathbb{A}} 
\nwc{\IB}{\mathbb{B}} 
\nwc{\IC}{\mathbb{C}} 
\nwc{\ID}{\mathbb{D}} 
\nwc{\IE}{\mathbb{E}} 
\nwc{\IF}{\mathbb{F}} 
\nwc{\IG}{\mathbb{G}} 
\nwc{\IH}{\mathbb{H}} 
\nwc{\IN}{\mathbb{N}} 
\nwc{\IP}{\mathbb{P}} 
\nwc{\IQ}{\mathbb{Q}} 
\nwc{\IR}{\mathbb{R}} 
\nwc{\IS}{\mathbb{S}} 
\nwc{\IT}{\mathbb{T}} 
\nwc{\IZ}{\mathbb{Z}} 

\def\bbleft{{\mathchoice {[\mskip-3mu {[}} {[\mskip-3mu {[}}{[\mskip-4mu {[}}{[\mskip-5mu {[}}}}
\def\bbright{{\mathchoice {]\mskip-3mu {]}} {]\mskip-3mu {]}}{]\mskip-4mu {]}}{]\mskip-5mu {]}}}}
\nwc{\setK}{\bbleft 1,K \bbright}
\nwc{\setN}{\bbleft 1,\cN \bbright}
 
\nwc{\va}{{\bf a}}
\nwc{\vb}{{\bf b}}
\nwc{\vc}{{\bf c}}
\nwc{\vd}{{\bf d}}
\nwc{\ve}{{\bf e}}
\nwc{\vf}{{\bf f}}
\nwc{\vg}{{\bf g}}
\nwc{\vh}{{\bf h}}
\nwc{\vi}{{\bf i}}
\nwc{\vI}{{\bf I}}
\nwc{\vj}{{\bf j}}
\nwc{\vk}{{\bf k}}
\nwc{\vl}{{\bf l}}
\nwc{\vm}{{\bf m}}
\nwc{\vM}{{\bf M}}
\nwc{\vn}{{\bf n}}
\nwc{\vo}{{\it o}}
\nwc{\vp}{{\bf p}}
\nwc{\vq}{{\bf q}}
\nwc{\vr}{{\bf r}}
\nwc{\vs}{{\bf s}}
\nwc{\vt}{{\bf t}}
\nwc{\vu}{{\bf u}}
\nwc{\vv}{{\bf v}}
\nwc{\vw}{{\bf w}}
\nwc{\vx}{{\bf x}}
\nwc{\vy}{{\bf y}}
\nwc{\vz}{{\bf z}}
\nwc{\bal}{\blds{\alpha}}
\nwc{\bep}{\blds{\epsilon}}
\nwc{\barbep}{\overline{\blds{\epsilon}}}
\nwc{\bnu}{\blds{\nu}}
\nwc{\bmu}{\blds{\mu}}
\nwc{\bet}{\blds{\eta}}

\nwc{\bk}{\blds{k}}
\nwc{\bm}{\blds{m}}
\nwc{\bM}{\blds{M}}
\nwc{\bp}{\blds{p}}
\nwc{\bq}{\blds{q}}
\nwc{\bn}{\blds{n}}
\nwc{\bv}{\blds{v}}
\nwc{\bw}{\blds{w}}
\nwc{\bx}{\blds{x}}
\nwc{\bxi}{\blds{\xi}}
\nwc{\by}{\blds{y}}
\nwc{\bz}{\blds{z}}

\nwc{\cA}{\ml{A}}
\nwc{\cB}{\ml{B}}
\nwc{\cC}{\ml{C}}
\nwc{\cD}{\ml{D}}
\nwc{\cE}{\ml{E}}
\nwc{\cF}{\ml{F}}
\nwc{\cG}{\ml{G}}
\nwc{\cH}{\ml{H}}
\nwc{\cI}{\ml{I}}
\nwc{\cJ}{\ml{J}}
\nwc{\cK}{\ml{K}}
\nwc{\cL}{\ml{L}}
\nwc{\cM}{\ml{M}}
\nwc{\cN}{\ml{N}}
\nwc{\cO}{\ml{O}}
\nwc{\cP}{\ml{P}}
\nwc{\cQ}{\ml{Q}}
\nwc{\cR}{\ml{R}}
\nwc{\cS}{\ml{S}}
\nwc{\cT}{\ml{T}}
\nwc{\cU}{\ml{U}}
\nwc{\cV}{\ml{V}}
\nwc{\cW}{\ml{W}}
\nwc{\cX}{\ml{X}}
\nwc{\cY}{\ml{Y}}
\nwc{\cZ}{\ml{Z}}

\nwc{\fA}{\mathfrak{a}}
\nwc{\fB}{\mathfrak{b}}
\nwc{\fC}{\mathfrak{c}}
\nwc{\fD}{\mathfrak{d}}
\nwc{\fE}{\mathfrak{e}}
\nwc{\fF}{\mathfrak{f}}
\nwc{\fG}{\mathfrak{g}}
\nwc{\fH}{\mathfrak{h}}
\nwc{\fI}{\mathfrak{i}}
\nwc{\fJ}{\mathfrak{j}}
\nwc{\fK}{\mathfrak{k}}
\nwc{\fL}{\mathfrak{l}}
\nwc{\fM}{\mathfrak{m}}
\nwc{\fN}{\mathfrak{n}}
\nwc{\fO}{\mathfrak{o}}
\nwc{\fP}{\mathfrak{p}}
\nwc{\fQ}{\mathfrak{q}}
\nwc{\fR}{\mathfrak{r}}
\nwc{\fS}{\mathfrak{s}}
\nwc{\fT}{\mathfrak{t}}
\nwc{\fU}{\mathfrak{u}}
\nwc{\fV}{\mathfrak{v}}
\nwc{\fW}{\mathfrak{w}}
\nwc{\fX}{\mathfrak{x}}
\nwc{\fY}{\mathfrak{y}}
\nwc{\fZ}{\mathfrak{z}}

\nwc{\tA}{\widetilde{A}}
\nwc{\tB}{\widetilde{B}}
\nwc{\tE}{E^{\vareps}}
\nwc{\tk}{\tilde k}
\nwc{\tN}{\tilde N}
\nwc{\tP}{\widetilde{P}}
\nwc{\tQ}{\widetilde{Q}}
\nwc{\tR}{\widetilde{R}}
\nwc{\tV}{\widetilde{V}}
\nwc{\tW}{\widetilde{W}}
\nwc{\ty}{\tilde y}
\nwc{\teta}{\tilde \eta}
\nwc{\tdelta}{\tilde \delta}
\nwc{\tlambda}{\tilde \lambda}
\nwc{\ttheta}{\tilde \theta}
\nwc{\tvartheta}{\tilde \vartheta}
\nwc{\tPhi}{\widetilde \Phi}
\nwc{\tpsi}{\tilde \psi}
\nwc{\tmu}{\tilde \mu}

\nwc{\To}{\longrightarrow} 

\nwc{\ad}{\rm ad}
\nwc{\eps}{\epsilon}
\nwc{\ep}{\epsilon}
\nwc{\vareps}{\varepsilon}

\def\ep{\epsilon}
\def\tr{{\rm tr}}

\def\sq2{\sqrt{2}}

\def\t2{{\mathbb T}^2}
\def\s2{{\mathbb S}^2}

\def\R{\mathbb{R}}

\def\Z{\mathbb{Z}}
\def\C{\mathbb{C}}
\def\O{\mathcal{O}}

\nwc{\lap}{\bigtriangleup}
\nwc{\rest}{\restriction}
\nwc{\Diff}{\operatorname{Diff}}
\nwc{\diam}{\operatorname{diam}}
\nwc{\Res}{\operatorname{Res}}
\nwc{\Spec}{\operatorname{Spec}}
\nwc{\Vol}{\operatorname{Vol}}
\nwc{\Op}{\operatorname{Op}}
\nwc{\supp}{\operatorname{supp}}
\nwc{\Span}{\operatorname{span}}

\nwc{\dia}{\varepsilon}
\nwc{\cut}{f}
\nwc{\qm}{u_\hbar}

\def\hto0{\xrightarrow{\hbar\to 0}}

\def\rto0{\xrightarrow{r\to 0}}

\providecommand{\abs}[1]{\lvert#1\rvert}
\providecommand{\norm}[1]{\lVert#1\rVert}

\nwc{\la}{\langle}
\nwc{\ra}{\rangle}
\nwc{\lp}{\left(}
\nwc{\rp}{\right)}

\nwc{\bequ}{\begin{equation}}
\nwc{\be}{\begin{equation}}
\nwc{\ben}{\begin{equation*}}
\nwc{\bea}{\begin{eqnarray}}
\nwc{\bean}{\begin{eqnarray*}}
\nwc{\bit}{\begin{itemize}}
\nwc{\bver}{\begin{verbatim}}

\nwc{\eequ}{\end{equation}}
\nwc{\ee}{\end{equation}}
\nwc{\een}{\end{equation*}}
\nwc{\eea}{\end{eqnarray}}
\nwc{\eean}{\end{eqnarray*}}
\nwc{\eit}{\end{itemize}}
\nwc{\ever}{\end{verbatim}}

\makeatletter
\newlength{\temp@wc@width}
\newlength{\temp@wc@height}
\newcommand{\widecheck}[1]{%
  \setlength{\temp@wc@width}{\widthof{$#1$}}%
  \setlength{\temp@wc@height}{\heightof{$#1$}}%
  #1\hspace{-\temp@wc@width}%
  \raisebox{\temp@wc@height+2pt}[\heightof{$\widehat{#1}$}]%
     {\rotatebox[origin=c]{180}{\vbox to 0pt{\hbox{$\widehat{\hphantom{#1}}$}}}}%
}
\makeatother

\DeclareMathOperator{\grad}{grad}

\newcommand{\mcg}{\mathrm{MCG}}

\newcommand{\Volwp}[1][g]{\mathrm{Vol}_{#1}^{\mathrm{\scriptsize{WP}}}}
\newcommand{\Pwpo}{\mathbb{P}_g^{\mathrm{\scriptsize{WP}}}}
\newcommand{\Ewpo}[1][g]{\mathbb{E}_{#1}^{\mathrm{\scriptsize{WP}}}}
\newcommand{\Pwp}[1]{\Pwpo \left( #1 \right)}
\newcommand{\Ewp}[2][g]{\mathbb{E}_{#1}^{\mathrm{\scriptsize{WP}}} \Bigg[ #2 \Bigg]}

\DeclarePairedDelimiter{\paren}{(}{)}

\DeclarePairedDelimiter{\abso}{|}{|}
\DeclarePairedDelimiter{\brac}{[}{]}
\DeclarePairedDelimiter{\floor}{\lfloor}{\rfloor}

\DeclarePairedDelimiter{\norminf}{\|}{\|_{\infty}}
\let\div\relax
\newcommand{\div}[1]{\paren*{\frac{#1}{2}}}

\renewcommand{\O}[2][ ]{\mathcal{O}_{#1} \left( #2 \right)}
\newcommand{\Ow}[2][ ]{\mathcal{O}_{#1}^w \left( #2 \right)}

\newcommand{\1}[1]{\mathds{1}_{#1}}

\DeclareMathOperator{\area}{Area}
\newcommand{\ntor}{N_{a}}
\DeclareMathOperator{\arcsinh}{arcsinh}

\newcommand{\av}[2][\mathrm{all}]{\langle #2 \rangle_g^{{#1}}}

\newcommand{\avb}[2][\mathrm{all}]{\left\langle #2 \right\rangle_g^{{#1}}}

\newcommand{\FR}{\mathcal{F}}
\newcommand{\FRrem}{\mathcal{R}}
\newcommand{\FRw}{\mathcal{F}_{w}}

\newcommand{\eqc}[1]{[ #1 ]_{\mathrm{loc}}}
\newcommand{\eq}{\, \raisebox{-1mm}{$\runderset{\mathrm{loc}}$} \,}

\newcommand{\tei}{\, \raisebox{-1mm}{$\runderset{\mathrm{Teich}}$} \,}
\newcommand{\eqmcg}{\, \raisebox{-1mm}{$\runderset{\mathrm{MCG}}$} \,}

\newcommand{\orb}{\mathrm{Orb}}
\newcommand{\Stab}{\mathrm{Stab}}

\newcommand{\x}{\mathbf{x}}

\DeclareMathOperator{\li}{li}

\makeatletter
\newcommand{\smallbullet}{} 
\DeclareRobustCommand\smallbullet{%
  \mathord{\mathpalette\smallbullet@{0.5}}%
}
\newcommand{\smallbullet@}[2]{%
  \, \vcenter{\hbox{\scalebox{#2}{$\m@th#1\bullet$}}} \,%
}
\makeatother

\newcommand{\Sf}{\mathbf{S}}
\newcommand{\type}{\mathbf{T}}
\newcommand{\simple}{\mathbf{s}}
\newcommand{\pop}{\mathbf{P}}
\newcommand{\ord}{N}
\newcommand{\rN}{\mathrm{N}}
\newcommand{\rK}{\mathrm{K}}
\newcommand{\curve}{\mathbf{c}}
\newcommand{\geod}{\mathcal{G}}
\newcommand{\cc}{\mathfrak{n}}

\newcommand{\chic}{{\chi_+}}

\begin{document}

\maketitle

\begin{abstract}
  In this series of articles, we analyze the level-sets of length functions on the moduli space of
  compact hyperbolic surfaces of fixed genus.  This work
  ultimately culminates in a proof that typical hyperbolic surfaces have an optimal spectral gap.

  In this first article, we introduce new \emph{volume functions} $V_g^{\mathbf{T}}(\ell)$, counting the
  expected number of closed geodesics of type $\mathbf{T}$ and length $\ell$ on a random hyperbolic
  surface of genus~$g$.  So far, this function has only been considered in the case where the type is
  \emph{simple}, in which case it can be expressed as a combination of Weil--Petersson volumes
  polynomials, as proven by Mirzakhani. We provide an integral expression for
  $V_g^{\mathbf{T}}(\ell)$ for any prescribed type $\mathbf{T}$, which we use to prove that
  $V_g^{\mathbf{T}}(\ell)$ admits a full asymptotic expansion in powers of $1/g$.

  We then claim that the coefficients in this expansion, as a function of the length variable
  $\ell$, belong to a newly-introduced class of functions called ``Friedman--Ramanujan
  functions''. We relate this claim to the study of the spectral gap of the Laplace--Beltrami
  operator, and prove it when $\mathbf{T}$ fills a surface of Euler characteristic $0$ or $-1$,
  providing a method to explicitly compute all coefficients in the second-order expansion.  We
  conclude by displaying how the presence of tangles (which is an event of vanishing probability)
  prevents the sum over all types to satisfy the Friedman--Ramanujan property at the second order.
\end{abstract}

\tableofcontents

\section{Introduction}
\label{sec:introduction}

The aim of this article is to develop new geometric tools for the study of \emph{random hyperbolic
  surfaces}, and in particular \emph{non-simple closed geodesics} on them.  The objects we
introduce, compute and study are vast generalisations of objects studied by Mirzakhani in the case
of simple closed geodesics \cite{mirzakhani2007,mirzakhani2013}, which have lead to several
breakthroughs in the field of random hyperbolic geometry. However, counting non-simple closed
geodesics presents several completely new challenges, from the understanding of complicated
level-set integrals to the proliferation of tangles. The tools developed in this article can be used
in a broad class of geometric and analysis problems, opening the door to counting problems which
would have otherwise been out of reach.

A \emph{compact hyperbolic surface} $X$ is a connected, oriented, compact surface, without boundary,
equipped with a Riemannian metric of constant curvature $-1$. Its topology is therefore entirely
determined by its \emph{genus} $g \geq 2$. We will be particularly interested in surfaces of large
genus. This large-genus limit can be viewed as a large-scale limit, because the area of a compact
hyperbolic surface of genus $g$ is $4 \pi(g-1)$ by the Gauss--Bonnet formula.

\subsection{Random hyperbolic surfaces}
\label{sec:rand-hyperb-surf-1}

Several different models of random hyperbolic surfaces exist
\cite{brooks2004,guth2011,mirzakhani2013,magee2022}. In this article, we will focus solely on the
\emph{Weil--Petersson model}, that consists in equipping the \emph{moduli space}
\begin{equation*}
  \mathcal{M}_g := \faktor{\{ \text{compact hyperbolic surfaces of genus } g\}}{\text{isometry}}
\end{equation*}
with the probability measure $\Pwpo$ obtained by renormalization of the measure induced by the
Weil--Petersson symplectic form on $\mathcal{M}_g$. This is a very natural probabilistic setting, in
which one can hope to accurately describe \emph{typical} surfaces.

In her breakthrough articles \cite{mirzakhani2007,mirzakhani2013}, Mirzakhani developed a toolbox
allowing to study the geometry of random hyperbolic surfaces sampled according to the probability
$\Pwpo$, especially in the large-genus limit. These tools have since then been applied in an
ever-growing number of articles, analyzing the geometric properties of random surfaces
\cite{monk2021a,parlier2021,nie2023,hide2022}, their spectral gap
\cite{wu2022,lipnowski2021,hide2022a,hide2022} and eigenfunctions \cite{gilmore2021}, as well as the
statistics of their length spectrum \cite{mirzakhani2019,wu2022a} and Laplacian spectrum
\cite{monk2022,rudnick2022}.

\subsection{The spectral gap of a compact hyperbolic surface}
\label{sec:spectral-gap-14}

While many of the results presented in this article are purely geometric, they are all deeply
motivated by an important question in \emph{spectral theory}, which we shall now present.

The spectral gap of a compact hyperbolic surface is the smallest non-zero eigenvalue $\lambda_1>0$
of the (positive) Laplace--Beltrami operator on the surface.  Surfaces with a large spectral gap are
known to be well-connected \cite{cheeger1970,buser1982}, fast-mixing for the geodesic flow and
random walks \cite{ratner1987,golubev2019}, and of small diameter \cite{magee2020b}.  Finding (rich)
families of such surfaces has been an objective shared by many, whether in the context of
arithmetics and number theory \cite{selberg1965,kim2003}, spectral geometry \cite{buser1988}, and
more recently random hyperbolic geometry
\cite{magee2022,wu2022,lipnowski2021,hide2021,hide2022,louder2022}.

In the large-genus regime, Huber \cite{huber1974} proved that the spectral gap is bounded above by a
quantity going to $1/4$ as $g \rightarrow + \infty$ ($1/4$ being the bottom of the spectrum of the
hyperbolic plane). The existence of surfaces of large genus with a near-optimal spectral gap was
conjectured by Burger--Buser--Dodziuk \cite{buser1988} in the 80's, and only solved very recently by
breakthrough work of Hide--Magee \cite{hide2022} using random covers.

{ In this series of papers, we prove that hyperbolic surfaces with a near-optimal spectral gap not
  only exist, but are \emph{typical}. }

\begin{thm}
  \label{con:gap}
  For any $\epsilon >0$,
  \begin{equation*}
    \lim_{g \rightarrow + \infty} \Pwp{\lambda_1 \geq \frac 14 - \epsilon} = 1.
  \end{equation*}
\end{thm}

The literature so far contains two probabilistic spectral gap results in the Weil--Petersson
setting. First, Mirzakhani proved in 2013 that random hyperbolic surfaces satisfy
$\lambda_1 > 0.002$ with probability going to $1$ as $g \rightarrow + \infty$
\cite{mirzakhani2013}. This bound has been vastly improved by two independent teams in 2021, Wu--Xue
\cite{wu2022} and Lipnowski--Wright~\cite{lipnowski2021}, who proved that for all $\epsilon > 0$,
$\lambda_1 \geq 3/16-\epsilon$ with probability going to $1$ as $g \rightarrow + \infty$.

We reach our final objective, and prove Theorem \ref{con:gap}, in the second article of this series.
In \cref{sec:canc-selb-trace}, we set up several key elements of the proof of Theorem~\ref{con:gap},
which motivate the objects and properties studied in this article.  Our proof is based on the trace
method (a standard method to study spectral gaps, as used in~\cite{wu2022,lipnowski2021}), with an
original cancellation argument.

In a previous version of this article, we used its main results to prove that typical hyperbolic
surfaces have a spectral gap greater than $\frac 29 - \epsilon$. The details of this intermediate result
can be found in \cite{anantharaman_29}.

\Cref{con:gap} is a result of the same nature as Alon's famous conjecture~\cite{alon1986} stating
that random $d$-regular graphs with $n \gg 1$ vertices typically have a near-optimal spectral
gap. It was solved by Friedman in~\cite{friedman2003}, after 20 years of active research. Compact
hyperbolic surfaces and regular graphs share a variety of geometric and spectral properties, and the
results presented in this article have counterparts in Friedman's proof of Alon's conjecture. 

\subsection{Averages of geodesic counting functions}
\label{sec:aver-geom-count}

A natural approach to access the geometry and spectrum of random
hyperbolic surfaces consists in reducing problems to the study of
averages of the form
\begin{equation}
  \label{e:def_av_all}
  \av[\text{all}]{F} := \Ewp{\sum_{\gamma \in \mathcal{G}(X)} F(\ell_X(\gamma))}
\end{equation}
where
\begin{itemize}
\item $\mathcal{G}(X)$ is the set of primitive oriented closed geodesics
  $\gamma$ on the surface $X$;
\item for any closed geodesic $\gamma$ on $X$, $\ell_X(\gamma)$ is the length of $\gamma$;
\item $F : \R_{\geq 0} \rightarrow \R$ is a test function, i.e. a bounded, compactly supported
  measurable function.
\end{itemize}
Such averages have been used to obtain geometric results in
\cite{mirzakhani2013,mirzakhani2019,monk2021a,nie2023}. Importantly, they appear
in trace methods when taking the expectation of the \emph{Selberg trace
  formula}, a formula relating the eigenvalues of the Laplacian to the lengths
of all closed geodesics on the surface (see \cref{sec:canc-selb-trace}).

Unfortunately, the methods developed by Mirzakhani
in~\cite{mirzakhani2007,mirzakhani2013} only allow to study such
sums if they are restricted to \emph{simple} geodesics,
i.e. geodesics with no self-intersection:
\begin{equation}
  \label{e:def_av_simple}
  \av[\simple]{F} := \Ewp{\sum_{\gamma \text{ simple}}
    F(\ell_X(\gamma))}.
\end{equation}
This has proven to be very restrictive, and dealing with non-simple closed geodesics is often a
challenging aspect of the study of random hyperbolic surfaces
\cite{mirzakhani2019,wu2022,lipnowski2021,rudnick2022}.

In this paper, we provide new information on the contribution of non-simple
geodesics to the average $\av{F}$. We hope the tools we develop can be used in
various settings.

\begin{rem}
  Wu--Xue proved in \cite{wu2022} that, for any $\eta > 0$,
  \begin{equation}
    \label{eq:only_simple_intro}
    \av{F} = \av[\simple]{F}
    + \O[\eta]{\frac{1}{g} \, \norminf{F(\ell) \, e^{(1+\eta)\ell}}}.
  \end{equation}
  As a consequence, at the leading order as $g \rightarrow + \infty$, non-simple
  geodesics do not contribute to the average $\av{F}$.
  However, we will see in this article that some non-simple geodesics yield
  contributions decaying like $1/g$ in the average $\av{F}$, which means that
  \cref{eq:only_simple_intro} cannot be extended past the precision $1/g$.
  
  We explain in \cref{sec:canc-selb-trace} why, in order to reach the optimal spectral gap
  $1/4 - \epsilon$, all computations need to be performed with arbitrary high precision, i.e. with
  errors decaying in $1/g^\ord$ for arbitrary large $\ord$.  The spectral gap $3/16 - \epsilon$ previously obtained in \cite{wu2022, lipnowski2021} then
  appears to be the threshold at which a description of the contribution of non-simple geodesics to
  the average $\av{F}$ becomes essential.  In \cite{anantharaman_29}, we show how the
    intermediate spectral gap $2/9 - \epsilon$ can be obtained by entirely analyzing the
    contribution of size $1/g$ of the average $\av{F}$, building on results from the present
    article.
\end{rem}

\subsection{Local topological types of geodesics}
\label{sec:non-simple-closed}

In order to study the average $\av{F}$ where the sum runs over all closed
geodesics, we regroup its terms according to what we call the \emph{local
  (topological) type} $\type$ of $\gamma$, a new notion coarser than the notion of mapping-class group orbit. This is done in
\cref{sec:topol-types-curv}, and we refer the reader to \cref{sec:preliminaries}
for the definitions of topological notions appearing below.

The data of a local topological type is given by a pair $\eqc{\Sf,\curve}$, where
$\Sf$ is a \emph{topological surface with boundary}, and $\curve$ is a \emph{filling
  loop} on $\Sf$. Several examples are represented in \cref{fig:local_types}. For instance,
all simple geodesics are grouped in a local type, given by a simple loop in a cylinder. Another
type, which we will describe in detail in this article, is the \emph{figure-eight}, i.e. geodesics
with exactly one self-intersection, represented at the top of \cref{fig:local_pop}. For this type,
$\Sf$ is a \emph{pair of pants} (a surface of signature $(0,3)$).

\begin{figure}[h!]
  \centering
  \begin{subfigure}[t]{0.3\textwidth}
    \centering
    \includegraphics[scale=0.4]{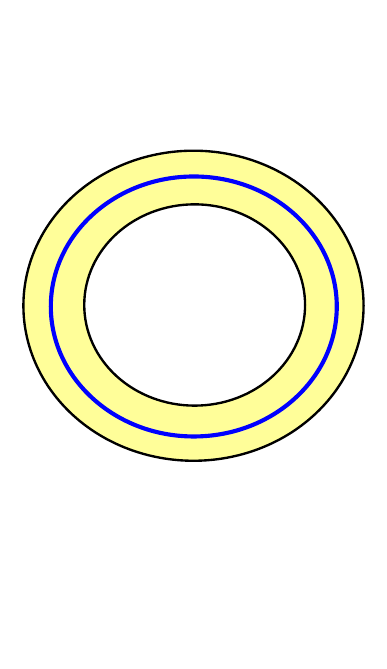} 
    \caption{The type ``simple''.}
    \label{fig:local_simple}
  \end{subfigure}%
  \begin{subfigure}[t]{0.35\textwidth}
    \centering
    \includegraphics[scale=0.4]{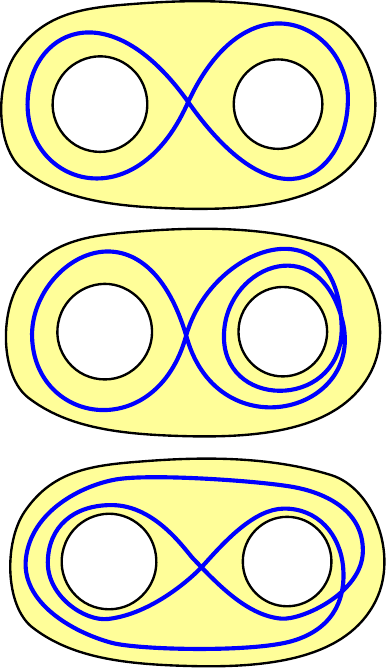}
    \caption{Types filling a pair of pants.}
    \label{fig:local_pop}
  \end{subfigure}%
  \begin{subfigure}[t]{0.3\textwidth}
    \centering
    \includegraphics[scale=0.4]{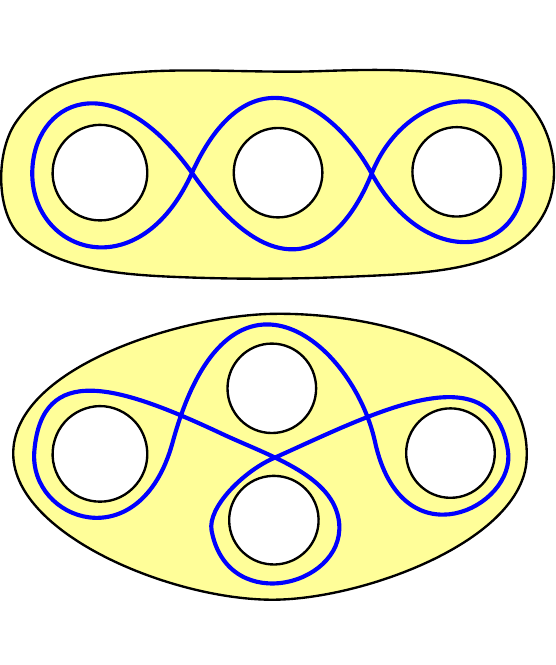}
    \caption{Other types.}
    \label{fig:local_other}
  \end{subfigure}%
  \caption{A few examples of local types.}
  \label{fig:local_types}
\end{figure}

Now, we say a closed geodesic $\gamma$ on a compact hyperbolic surface $X$ is \emph{of local type
  $\type = \eqc{\Sf,\curve}$} if there exists an embedding $\Sf \rightarrow X$ sending $\curve$ on
$\gamma$. For a local type $\type$ and a test function $F$, we define
\begin{equation*}
  \av[\type]{F}
  := \Ewp{\sum_{\gamma \text{ of type } \type} F(\ell_X(\gamma))}
\end{equation*}
which allows us to rewrite
\begin{equation}
  \label{eq:decomp_types_intro}
  \av{F}
  = \sum_{\type \text{ local type} }\av[\type]{F}.
\end{equation}

\begin{rem}
  The word \emph{local} is an emphasis on the fact that the notion of local type
  only depends on the topology of the geodesic itself, and not the way it is
  embedded in the surface of genus $g$. Notably, all simple closed geodesics
  form one local type. This notion should not be confused with the notion of
  \emph{topological type}, which is often used to refer to mapping-class-group
  orbits (in Mirzakhani's work for instance). We prove in \cref{lem:local_vs_mcg}
  that the notion of mapping-class-group equivalence is finer than the notion of
  local equivalence. Hence, every local type can be decomposed as a disjoint
  union of topological types, which we refer to as its
  \emph{realizations}. Realizations correspond to the different ways the pair
  $(\Sf,\curve)$ can be embedded in a surface of genus $g$.
\end{rem}

\subsection{Statement of the main results on the averages}

In \cref{thm:express_average}, we provide an expression for the averages
$\av[\type]{F}$ in terms of Weil--Petersson volumes of moduli spaces of bordered
hyperbolic surfaces. We use this expression to prove the following.

\begin{thm}[{Theorems \ref{prp:existence_density} and
    \ref{thm:exist_asympt_type}}]
  \label{thm:exist_asympt_type_intro}
  Let $\type$ be a local topological type.
  \begin{itemize}
  \item For any $g \geq 2$, there exists a unique locally integrable function
    $V_g^\type : \R_{> 0} \rightarrow \R_{\geq 0}$, called \emph{volume function}, such that, for
    any test function $F$,
    \begin{equation*}
      \av[\type]{F} = \frac{1}{V_g} \int_{0}^{+\infty} F(\ell) V_g^{\type}(\ell) \d \ell.
    \end{equation*}
  \item There exists a unique family of locally integrable functions
    $(f_k^\type)_{k \geq \chi(\type)}$ such that, for any $\ord \geq 0$, any large
    enough $g$, any $\epsilon >0$,
  \begin{equation}
    \label{eq:exist_asympt_type_intro}
    \frac{V_g^\type(\ell)}{V_g}
    = \sum_{k=\chi(\type)}^\ord \frac{f_k^\type(\ell)}{g^k}
    + \Ow[\ord,\chi(\type),\epsilon]{\frac{\exp \paren*{(1+\epsilon)\ell}}{g^{\ord+1}}}
  \end{equation}
  where $\chi(\type)$ denotes the absolute Euler characteristic of the local type $\type$.
  \end{itemize}
\end{thm}

\begin{rem}
  The notation $\Ow{\cdot}$ is a weak version of the usual Landau notation
  $\O{\cdot}$, and introduced in \cref{sec:notations}.
\end{rem}

\Cref{thm:exist_asympt_type_intro} is new in all cases except for the local
type ``simple'', where it comes as a consequence of
\cite{mirzakhani2007,mirzakhani2015,wu2022}. In this case, we know the value of
$f_0^{\simple}(\ell) = \frac{4}{\ell} \sinh^2 \div{\ell}$, and this
expression is an essential component of many recent results
\cite{mirzakhani2019,wu2022,lipnowski2021,rudnick2022}.

We highlight the fact that the leading-order term of $\av[\type]{F}$ for the local type $\type $
decays like $1/g^{\chi(\type)}$, where $\chi(\type) = \chi(\Sf) \geq 0$ is the absolute value of the Euler
characteristic (or for short, ``absolute Euler characteristic'') of the filled surface $\Sf$, when $\type = \eqc{\Sf,\curve}$. This is the reason why
only simple geodesics contribute to the leading term of $\av{F}$: for all other local types,
$\chi(\type) \geq 1$.

\begin{rem}
  \label{rem:difficult_non_simple}
  Let us describe the obstacle to the study of non-simple geodesics in
  Mirzakhani's work. Mirzakhani's integration formula \cite{mirzakhani2007}
  allows to write an explicit formula for the average $\av[\simple]{F}$. This
  is done by considering a random hyperbolic surface~$X$ of genus $g$ containing
  a simple closed geodesic $\gamma$ of length $\ell$, and analyzing the topology
  of the surface $X \setminus \gamma$ obtained by \emph{cutting} $X$ along the
  geodesic $\gamma$. Because the geodesic is simple, the result is a (possibly
  disconnected) hyperbolic surfaces with two \emph{geodesic boundary components}
  of length $\ell$.
  
  This unfortunately ceases to be true for non-simple geodesics.  In order to
  remedy this, we rather cut $X$ along the boundary components of the surface
  $S(\gamma)$ filled by the non-simple geodesic~$\gamma$. This approach was used
  to a certain extent in \cite{mirzakhani2019,wu2022,lipnowski2021}, but we push
  it further, which allows us to write a formula for $\av[\type]{F}$
  (\cref{thm:express_average}).  The formula is more involved, because it
  contains an average on all possible hyperbolic metrics on the filled surface
  $S(\gamma)$ for which $\gamma$ has length $\ell$.
\end{rem}

\begin{exa}
  In \cref{exa:int_pop}, we compute $\av[\type]{F}$ for the type
  $\type_8 := \eqc{\pop,\curve_8}$, where $\pop$ is a pair of pants and $\curve_8$ a
  figure-eight. A metric on the pair of pants $\pop$ is entirely described by the
  lengths $\ell_1, \ell_2, \ell_3$ of its three boundary components. As a
  consequence, the formula for $\av[\type]{F}$ in this setting takes the form of an
  integration on the two-dimensional level-set
  \begin{equation}
    \label{eq:level_set_P}
    \left\{ (\ell_1, \ell_2, \ell_3) \in \R_{>0}^3 \, : \,
    \cosh \div{\ell} = 2 \cosh \div{\ell_1} \cosh \div{\ell_2} + \cosh
    \div{\ell_3}
    \right\}
  \end{equation}
  which corresponds to the metrics on $\pop$ for which the length of the figure-eight is
  exactly~$\ell$.
\end{exa}

In \cref{sec:existence-asym-ext}, we explain how to extend
\cref{thm:exist_asympt_type_intro} to the overall average $\av{F}$, obtained by
summing over all closed geodesics.

\begin{thm}[Theorems \ref{prp:existence_density_all} and
  \ref{thm:existence_asym_all}] \quad
  \label{thm:existence_asym_all_intro}
  \begin{itemize}
  \item There exists a unique locally integrable function
  $V_g^{\mathrm{all}} : \R_{> 0} \rightarrow \R_{\geq 0}$ such that, for any test function $F$,
  \begin{equation*}
    \av{F} = \frac{1}{V_g} \int_{0}^{+\infty} F(\ell) V_g^{\mathrm{all}}(\ell) \d \ell.
  \end{equation*}
\item   There exists a unique family of locally integrable functions
  $(f_k^{\mathrm{all}})_{k \geq 0}$ such that, for any $A \geq 1$, 
  $\ord \geq 0$, $\epsilon > 0$, any large enough $g$, if $L := A \log(g)$,
  \begin{equation}
    \label{eq:existence_asym_all_intro}
    \frac{V_g^{\mathrm{all}}(\ell)}{V_g} \1{[0, L]}(\ell)
    = \sum_{k=0}^{\ord} \frac{f_k^{\mathrm{all}}(\ell)}{g^k} \1{[0, L]}(\ell)
    + \Ow[\epsilon,\ord,A]{\frac{\exp((1+\epsilon) \ell)}{g^{\ord+1}}}.
  \end{equation}
  \end{itemize}
\end{thm}

\begin{rem}
  The indicator function in \cref{eq:existence_asym_all_intro} is used to reduce
  the number of local topological types that need to be summed when computing
  $\av{F}$. Indeed, if we do not assume that we only look at geodesics of 
  length $\leq A \log(g)$, we a priori need to take into account the
  geodesics filling the whole surface of genus $g$, for instance.
\end{rem}

Now that we know that the averages $\av[\type]{F}$ and $\av{F}$ admit asymptotic
expansions in powers of $1/g$, we shall be concerned with the form of the
coefficients $(f_k^\type)_{k \geq 0}$ and $(f_k^{\mathrm{all}})_{k \geq 0}$
appearing in these expansions.

\subsection{Friedman--Ramanujan functions}
\label{sec:friedm-raman-funct}

An essential step in Friedman's proof of Alon's conjecture is the introduction
of a notion of \emph{Ramanujan functions} \cite[Section 7]{friedman2003}.  We
introduce a similar notion in the context of random hyperbolic geometry. We then show
that this class of functions arises naturally when studying the lengths of
geodesics on random hyperbolic surfaces, and its relevance to the spectral gap
problem.

\begin{defa}[\cref{def:FR}]
  A locally integrable function $f : \R_{\geq 0} \rightarrow \C$ is said to be a
  \emph{Friedman--Ramanujan function} if there exists a polynomial function $p$
  and constants $c >0$, $\rN \geq 1$ such that
  \begin{equation*}
    \forall \ell \geq 0, \quad 
    |f(\ell) - p(\ell) \, e^\ell| \leq c (\ell + 1)^{\rN-1} e^{\frac \ell 2}.
  \end{equation*}
  We denote as $\FR$ the class of Friedman--Ramanujan functions, and as $\FRrem$
  the subset of Friedman--Ramanujan function for which $p=0$.
  We similarly define a notion of Friedman--Ramanujan function \emph{in the
    weak sense} using the weaker $\Ow{\cdot}$.
\end{defa}

This is a natural adaptation of Friedman's definition of \emph{Ramanujan
  functions} for $d$-regular graphs, namely functions $f : \Z_{>0} \rightarrow \C$
such that
\begin{equation*}
  |f(\ell) - p(\ell) \,(d-1)^\ell| \leq c \ell^{\rN-1} (d-1)^{\frac \ell 2}
\end{equation*}
for a polynomial function $p$ and constants $c>0$, $\rN \geq 1$. The quantities
$e^\ell$ and $(d-1)^\ell$ are the growth-rate of balls in the
hyperbolic plane and the $d$-regular tree respectively.

\begin{rem}
  The name ``Ramanujan'' was chosen by Friedman in relationship to the
  breakthrough work by Lubotzky--Phillips--Sarnak \cite{lubotzky1988}, in which
  the authors prove the existence of large $d$-regular graphs with an optimal
  spectral gap (such graphs are called \emph{Ramanujan graphs} due to the use
  of the Ramanujan conjecture in \cite{lubotzky1988}). We have chosen the name
  ``Friedman–Ramanujan'' with the wish to both maintain the link with the
  original article that inspired this work and emphasise Friedman’s impressive
  contribution to the study of random $d$-regular graphs.
\end{rem}

\begin{rem}
  An alternative way to understand the definition of
  Friedman--Ramanujan function and its relation to the spectral gap
  problem is to look at the prime number theorem with error terms,
  proven by Huber \cite{huber1974} (see also \cite[Theorem
  9.6.1]{buser1992}). This theorem states that, for a fixed
  hyperbolic surface $X$ and a large $\ell$,
  \begin{equation*}
    N_X(\ell) := \#\{\gamma :  \ell_X(\gamma) \leq \ell \}
    = \li(e^\ell)
    + \sum_{j: 0 < \lambda_j < \frac{3}{16}}
    \li \big( e^{s_j \ell} \big)
    + \O{\frac{e^{\frac 3 4 \ell }}{\ell}}
  \end{equation*}
  where $\li(x) := \int_2^x \frac{\d \tau}{\log \tau} \sim x/\log(x)$ and
  $s_j := \frac 12 + \sqrt{\frac 14 - \lambda_j} $.  The leading term
  $\li(e^\ell) \sim e^\ell/\ell$ comes from the eigenvalue $\lambda_0 = 0$. We observe that small
  eigenvalues (or at least the ones smaller than $3/16$) correspond to subdominant contributions to
  $N_X(\ell)$. The exponent gap in the definition of Friedman--Ramanujan functions, between the
  exponent $e^\ell$ in the main term and the exponent $e^{\ell/2}$ in the remainder, corresponds
  exactly to the gap between the trivial eigenvalue $0$ and the optimal spectral gap $1/4$.
\end{rem}

\subsection{Link between Friedman--Ramanujan functions and spectral gaps}

The motivation behind the introduction of Friedman--Ramanujan functions is that
one can exhibit \emph{cancellations} in the Selberg trace formula thanks to
their structure.  We discuss this relationship in \cref{sec:canc-selb-trace}. It
motivates the following objective.

\begin{named}{Objective (FR)}
  \label{chal:FR_type}
  Let $\type$ be a local type other than simple. Prove that, for any $k \geq 0$, the function
  $f_k^\type$ is a Friedman--Ramanujan function in the weak sense.
\end{named}

The local type ``simple'' is singled out because the functions $\ell \mapsto f_k^{\simple}(\ell)$
have a singularity of order one at $0$. We have already mentioned that
$f_0^{\simple}(\ell) = \frac{4}{\ell} \sinh^2 \div{\ell}$. Clearly, the function
$\ell \mapsto 4 \sinh^2 \div{\ell} = \ell f_0^\simple(\ell)$ is a Friedman--Ramanujan function. For
higher-order terms in the average $\av[\simple]{F}$, \cref{prop:simple} relies on our previous
work \cite{anantharaman2022} and states that $\ell \mapsto \ell f_k^\simple(\ell)$ is a
Friedman--Ramanujan function for any $k \geq 0$.

In this article, we prove \ref{chal:FR_type} for the following local types.

\begin{thm}
  \label{thm:FR_type_proven_intro}
  For any local type $\type$ filling a surface of absolute Euler characteristic~$1$, all functions
  $(f_k^\type)_{k \geq 0}$ are Friedman--Ramanujan in the weak sense.
\end{thm}

Surfaces of absolute Euler characteristic $1$ are the pair of pants and the once-holed torus. The
proof of Theorem \ref{thm:FR_type_proven_intro} goes by explicitly writing down the form of the
coefficients $f_k^\type$ and directly exhibiting the Friedman--Ramanujan property. The
arguments are presented in Sections \ref{sec:generalised_convolutions} (for the figure-eight filling
a pair of pants) and \ref{sec:other-geodesics} (for all other loops filling a pair of pants or
once-holed torus). The technical difficulties of such an explicit approach are hinted at in \cref{rem:difficult_non_simple}.

Because the expansion in \cref{thm:exist_asympt_type_intro} starts with the term $k=\chi(\type)$,
\cref{thm:FR_type_proven_intro} has the following immediate consequence.

\begin{cor}
  \label{cor:FR_type_first_order_intro}
  For \emph{any} local type $\type$, the function $\ell \mapsto \ell f_k^\type(\ell)$ is
  Friedman--Ramanujan in the weak sense for $k=0$ and $1$.
\end{cor}

The analysis presented in Section \ref{sec:def_FR} explains why \ref{chal:FR_type} is
the key ingredient to our proof of the optimal spectral gap result, \cref{con:gap}.  Already Corollary \ref{cor:FR_type_first_order_intro}, in a quantitative version and with some extensions, can be used to prove that
typical surfaces have a spectral gap $2/9 - \epsilon$, as we have done in an expanded version of this paper \cite{anantharaman_29}.

We fulfil \ref{chal:FR_type} for any local type and any $k$ in the second paper of this series.  The
proof of \cref{con:gap} requires Sections \ref{sec:preliminaries} to \ref{sec:existence-asym-ext} of
this paper, where several new concepts are defined and important results
established.

Sections~\ref{sec:generalised_convolutions} and \ref{sec:other-geodesics} are dedicated to the proof
of \cref{thm:FR_type_proven_intro} and not technically required for the proof of
\cref{con:gap}. However, they do illustrate how the coefficients of volume functions can be 
computed, and highlight several difficulties that arise when dealing with the level-set integrals
appearing in their expression, in an explicit, elementary case. The approach here is
more direct, and not only proves the Friedman--Ramanujan property but also calculates the polynomial
term; this is not the case in the second article. 
When pushed to its maximum precision, this explicit approach contains the reward of the full
information about the polynomial part of the Friedman--Ramanujan functions, which is of independent interest and contains information about fine spectral statistics.

\subsection{The challenge of tangles}
\label{sec:challenge-tangles}

Another striking demonstration of the intimate relationship between
Friedman--Ramanujan functions and the spectral gap problem can be found in our
proof of the following statement.

\begin{thm}[\cref{prp:non_FR}]
  \label{thm:non_FR_intro}
  The function $\ell \mapsto \ell f_1^{\mathrm{all}}(\ell)$ is not a
  Friedman--Ramanujan function in the weak sense.
\end{thm}

This might seem surprising, because
\cref{cor:FR_type_first_order_intro} implies that
$\ell \mapsto \ell f_1^{\mathrm{all}}(\ell)$ is a countable sum of
Friedman--Ramanujan functions in the weak sense, and this property
is stable by linear combination. The proof of this result consists
in proving that, if the counting functions are Friedman--Ramanujan,
then we can obtain quantitative information on the spectral gap.

\begin{lem}[{\cref{lem:FR_contradiction}}]
  \label{thm:FR_contradiction_intro}
  If $\ell \mapsto \ell f_1^{\mathrm{all}}(\ell)$ is a Friedman--Ramanujan
  function in the weak sense, then for small $\delta>0$ and large enough $g$,
    \begin{equation}
      \label{eq:FR_contradiction_intro}
      \Pwp{\delta < \lambda_1 < \frac{5}{72}} = \O[\delta]{\frac{1}{g^{5/4}}}.
  \end{equation}
\end{lem}

The contradiction then arises from the following estimate on the probability for
a surface to have a small eigenvalue.

\begin{thm}[\cref{lem:prob_small_eig}]
  \label{lem:prob_small_eig_intro}
  There exists $c_1, c_2>0$ such that, for small enough $a>0$ and large
  enough $g$,
  \begin{equation*}
    c_1 \frac{a^2}{g} \leq \Pwp{\lambda_1 \leq a} \leq c_2 \frac{a}{g}\cdot
  \end{equation*}
\end{thm}

In particular the rate of growth $g^{-5/4}$ in \cref{thm:FR_contradiction_intro}
is too fast. \Cref{lem:prob_small_eig_intro} is obtained by observing that, for
a small $a>0$, the probability for a random surface to contain a once-holed torus
with a boundary of length $\leq a$ is roughly $a^2/g$. By the min-max principle,
if a surface contains such a piece, then $\lambda_1 \leq a$.

More generally, embedded subsurfaces with a short boundary are linked to the
presence of small eigenvalues, because in this case the surface is poorly
connected \cite{cheeger1970,buser1982}. We call such subsurfaces \emph{tangles}
-- this notion appears in \cite{monk2021a,lipnowski2021}.

The value $3/16$ is the threshold at which tangles start to manifest, because their probability is
of size $1/g$. In order to go past $3/16$, we need to \emph{remove tangles}.  In Friedman's proof of
the Alon conjecture, the presence of tangles is a significant challenge: those issues are explained
in \cite[Section 2]{friedman2003} and are the motivation for introducing a notion of \emph{selective
  trace}.  The main technical challenge to go from \cref{thm:FR_type_proven_intro} to the
  spectral gap $2/9 - \epsilon$ is the necessity to perform a much more detailed
  inclusion-exclusion. We show in~\cite{anantharaman_29} how that can be managed by an explicit
  computation (for this level of precision), at the cost of a tedious topological enumeration.  We
  develop a more systematic approach to the removal of tangles in \cite{Moebius} which, together
  with the proof of \ref{chal:FR_type} for any local type, allows to prove \cref{con:gap} in the
  second article.

\subsection*{Acknowledgements}

The authors would like to express their gratitude to Joel Friedman, for
explaining his proof of Alon's conjecture to us in detail, which lead to
significant advances in our project. We would also like to thank Michael
Lipnowski and Alex Wright for sharing their insight on the spectral gap problem, and Yuhao Xue for useful comments on a first version of the paper.
We are grateful to Nir Avni and Steve Zelditch for the conference they organised
in Northwestern University, where we met Joel Friedman and presented some of
these results for the first time.

This research was funded by the guest program of the Max-Planck Institute
for Mathematics during the year 2021-2022, the EPSRC grant EP/W007010/1 since
2022, the prize L’Oréal-UNESCO Young Talents France for Women in
Science, and by the European Research Council (ERC) under the
European Union’s Horizon 2020
research and innovation programme (Grant agreement No. 101096550).


\section{Preliminaries}
\label{sec:preliminaries}

In this section, we introduce many objects relevant to this article, for the
sake of clarity and self-containment. For a more detailed exposition of these
notions, we refer the reader to \cite{buser1992} for hyperbolic geometry, and
\cite{wright2020,monk2021} for the theory of random hyperbolic surfaces.

\subsection{Notations}
\label{sec:notations}

For two quantities $F_1, F_2$, we write $F_1 = \O{F_2}$ if there exists a
constant $C>0$ such that $\abso{F_1} \leq C F_2$ for any choice of parameters
within the allowed ranges. If the constant depends on a parameter $\alpha$, we
write $F_1 = \O[\alpha]{F_2}$.

For a locally integrable function $F_1 : \R_{\geq 0} \rightarrow \C$ and a non-decreasing positive
function $F_2 : \R_{\geq 0} \rightarrow \R_{>0}$, we say that $F_1 = \Ow{F_2}$ if $F_1$ is bounded
by $F_2$ in a \emph{weak sense}, i.e. if there exists a constant $C>0$ such that, for all
$L \geq 0$, we have $\int_0^L |F_1(\ell)| \d \ell \leq C F_2(L)$. If the constant $C$ depends on a
parameter~$\alpha$, we rather write $F_1 = \Ow[\alpha]{F_2}$.

\subsection{Hyperbolic geometry and closed geodesics}

\subsubsection{Compact and bordered surfaces}
\label{s:compact_bordered}

All surfaces in this article are assumed to be oriented, connected and of finite
type (with a finitely generated fundamental group).

A \emph{compact hyperbolic surface} $X$ is a closed surface equipped with a Riemannian metric of
constant curvature $-1$. The topology of $X$ is therefore entirely determined by its genus
$g \geq 2$. By the Gauss--Bonnet formula, $X$ has finite area, equal to $2 \pi \chi(X)$, where
$\chi(X) = 2g-2 > 0$ is the absolute Euler characteristic of $X$.

The study of compact hyperbolic surfaces is the core focus of this
article. However, in doing so, we will need to cut these surfaces along some
simple closed geodesics -- which shall lead us to consider surfaces with a
geodesic boundary. A \emph{bordered hyperbolic surface} is a surface equipped
with a Riemannian metric of curvature $-1$, with a (finite) set of boundary
components, labelled $\{1, \ldots, n\}$, which are either closed geodesics or
cusps (which we will abusively refer to as components of length $0$).  The
signature of~$X$ is the pair $(g,n)$, where $g$ is its genus. The Gauss--Bonnet
formula extends to this setting, with $\chi(X) = 2g-2+n$. The case $n=0$
corresponds to the compact case above.

\subsubsection{Primitive closed geodesics}
\label{sec:clos-paths-geod}

A \emph{loop} on~$X$ is a piece-wise smooth map~\mbox{$\gamma : \R \diagup \Z \rightarrow X$}. Notice
that our loops are oriented. Two loops $\gamma_0$ and $\gamma_1$ are \emph{homotopic} if there
exists a continuous map \mbox{$h : [0,1] \times \R \diagup \Z \rightarrow X$} such that
$h_{|\{0\} \times \R \diagup \Z} = \gamma_0$ and $h_{|\{1\} \times \R \diagup \Z} = \gamma_1$
. We say the loop~$\gamma$ is \emph{non-primitive} if there exists an integer $m \geq 2$ and a loop
$\gamma_0$ such that $\gamma$ is homotopic to $\gamma_0^m$, and primitive otherwise. A loop is
called \emph{essential} if it is neither contractible nor homotopic to a boundary component or a
cusp of $X$ (the second condition only matters if $X$ is a bordered surface).

We denote as~$\mathcal{G}(X)$ the set of homotopy-classes of primitive essential
loops on~$X$. It can alternatively be seen as the set of \emph{primitive
  oriented closed geodesics} on $X$, because each homotopy class in
$\geod(X)$ contains a unique geodesic representative. For
$\gamma \in \geod(X)$, we denote as~$\ell_X(\gamma)$ the length of the
geodesic representative in the homotopy class~$\gamma$.

In the following, we will often abusively refer to elements of $\geod(X)$ as homotopy classes,
loops, or closed geodesics; in the latter two cases we will always talk about them up to
homotopy. In particular, we say that two elements $\gamma_0$ and $\gamma_1$ of $\geod(X)$ are
\emph{distinct} (and denote $\gamma_0 \neq \gamma_1$) if $\gamma_0$ is not homotopic to $\gamma_1$.
These elements are called \emph{disjoint} if $\gamma_0 \neq \gamma_1$,
$\gamma_0 \neq \gamma_1^{-1}$, and if the homotopy classes $\gamma_0$, $\gamma_1$ admit
representatives which have no intersections. An element of $\geod(X)$ is \emph{simple} if it admits a
representative with no self-intersections (which implies that the geodesic representative also has
no self-intersections). A \emph{multi-curve} is an ordered family $(\gamma_1, \ldots, \gamma_k)$ of
disjoint simple elements of $\geod(X)$; taking the geodesic representative of each homotopy class in
this family yields a family of simple, disjoint geodesics on $X$ (i.e. the geodesics have no
self-intersections and no mutual intersections). Note that, with this definition, we require that
the components of a multi-curve are all essential, but in the following, we will sometimes relax
this and allow for some of them to be homotopic to boundary components of $X$ if $X$ is bordered.

\begin{rem}
  \label{rem:convention_oriented}
  In most papers of the field,
  e.g. \cite{mirzakhani2007,mirzakhani2013,mirzakhani2019,wu2022,lipnowski2021,rudnick2022},
  geodesics are considered to be non-oriented, and orbits and stabilisers are
  defined for non-oriented loops, and therefore different from ours. Here, we
  choose to consider all loops and multi-curves to be oriented, because the
  Selberg trace formula classically runs over all oriented geodesics. We believe
  this convention to make a few discussions about constants appearing in
  formulas slightly simpler.
\end{rem}

\subsubsection{Geodesic counting}
The set of primitive closed geodesics on a hyperbolic surface is discrete, and
we shall need to count geodesics of a bounded length. Several counting arguments
will appear in this article, the simplest being the following.

\begin{lem}
  \label{lem:bound_number_closed_geod}
  Let $X$ be a hyperbolic surface, compact or bordered. For any $L>0$,
  \begin{equation}
    \# \{ \gamma \in \geod(X) \, | \, \ell_X(\gamma) \leq L \}
    \leq 205 \, \chi(X) \, e^{L}.
  \end{equation}
  {
  As a consequence, if $F$ is a bounded function supported in $[0, L]$, then
  \begin{align}\label{e:basic}
  \sum_{\gamma \in \geod(X)}| F(\ell(\gamma))| \leq 560 \, {\chi(X)} (L+1) \norm{F(\ell)e^{\ell}}_\infty.
  \end{align}}  
\end{lem}

\begin{proof}
  First, if $X$ is compact of genus $g$, then by \cite[Theorem 4.1.6 and Lemma
  6.6.4]{buser1992},
  \begin{equation*}
    \# \{ \gamma \text{ primitive geodesic } | \, \ell_X(\gamma) \leq L \}
    \leq 3g-3 + (g-1) \, e^{L+6} 
  \end{equation*}
  which implies the result, because $\chi(X) = 2g-2$.

  Following the proof of \cite[Proposition 4.5]{mirzakhani2019}, we extend the result to surfaces
  with a boundary, by doubling the surface: we take two copies of the surface $X$ and glue them
  along their matching boundary components. We obtain a compact surface $X'$, of absolute Euler
  characteristic $2 \chi(X)$ by additivity of the Euler characteristic. Each primitive closed
  geodesic on $X$ can be sent injectively on two primitive closed geodesics on~$X'$ of the same
  length, and hence the number of primitive closed geodesics~$\leq L$ on~$X$ is smaller than half
  the number of primitive closed geodesics on $X'$.

  {
  The bound \eqref{e:basic} is obtained by observing that
  \begin{align*}
    \sum_{\gamma \in \geod(X)}| F(\ell(\gamma))|
    & \leq \norm{F(\ell)e^{\ell}}_\infty \sum_{\substack{\gamma \in \geod(X) \\ \ell_X(\gamma) \leq L}} e^{- \ell_X(\gamma)}
  \end{align*}
  and then cutting the sum in small intervals,
  \begin{equation*}
    \sum_{\substack{\gamma \in \geod(X) \\ \ell_X(\gamma) \leq L}} e^{- \ell_X(\gamma)}
    \leq \sum_{k=0}^{\lfloor L \rfloor} e^{-k} \# \{\gamma \in \geod(X) \, : \, k \leq
    \ell_X(\gamma) < k+1\} \leq 205 (L+1) e^{L+1} \chi(X) .
  \end{equation*}}
\end{proof}

\subsubsection{Filling geodesics and Wu--Xue's improved geodesic counting}
\label{sec:fill-geod-impr}

When studying a closed geodesic $\gamma$ on a surface $X$, it is often very
convenient to introduce a subsurface of $X$ that is filled by $\gamma$ in the
following sense.

\begin{defa}
  \label{def:fill}
  Let $\Sf$ be a topological surface, possibly with a boundary. We say a loop
  $\gamma$ on $\Sf$ \emph{fills} the surface $\Sf$ if each connected component of
  $\Sf \setminus \gamma$ is either contractible or an annular region around a
  boundary component of $\Sf$.
\end{defa}

For a fixed $\Sf$, one can wonder how many geodesics of length $\leq L$ fill
$\Sf$. An impressive counting result on this quantity was obtained by Wu--Xue in
\cite{wu2022}.

\begin{thm}
  \label{thm:wu_xue_counting}
  For any $\eta > 0$, any topological surface $\Sf$ with boundary, there
  exists a constant $C_{\chi(\Sf), \eta}>0$ such that, for any hyperbolic metric $Y$
  on $\Sf$, any $L > 0$,
  \begin{equation*}
    \# \{\gamma \text{ primitive loop filling } \Sf \, | \, \ell_Y(\gamma) \leq L \}
    \leq C_{\chi(\Sf), \eta}
    \exp \paren*{L - \frac{1 - \eta}{2} \, \ell_Y(\partial \Sf)}.
  \end{equation*}
\end{thm}

Here $\ell_Y(\partial \Sf)$ is the total length of the boundary of $\Sf$ for the metric $Y$.  This
result is an improvement of the naive bound from \cref{lem:bound_number_closed_geod}, thanks to the
decaying properties of the term $\exp(-(1-\eta)\ell_Y(\partial \Sf)/2)$.  It is a central part of
Wu--Xue's proof that typical surfaces have a spectral gap at least $3/16 - \epsilon$.

\subsection{Random hyperbolic surfaces}
\label{sec:rand-hyperb-surf}

Let $g \geq 2$. In this article, we sample random hyperbolic surfaces of genus
$g$ according to the Weil--Petersson probabilistic setting, which we shall now
introduce briefly.

\subsubsection{The moduli space}

We sample our random surfaces in the \emph{moduli space}
\begin{equation*}
  \mathcal{M}_g :=
  \{ \text{compact hyperbolic surfaces } X \text{ of genus } g\} \diagup \text{isometry}.
\end{equation*}
In order to study the moduli space, it is very convenient to introduce its
universal covering, the \emph{Teichm\"uller space} $\mathcal{T}_g$, which can be
seen as the set of \emph{marked hyperbolic surfaces}. More precisely, we fix a
surface $S_g$ of genus $g$, which we call the \emph{base surface}. Then,
\begin{equation*}
  \mathcal{T}_g = \left\{ (X, \phi), \quad
  \parbox{6.2cm}{$X$ compact hyperbolic surface \\
    $\phi : S_g \rightarrow X$  positive homeomorphism} \right\}
\diagup \tei
\end{equation*}
where the quotient is defined by saying that $(X_1, \phi_1) \tei (X_2, \phi_2)$
if there exists an isometry $m : X_1 \rightarrow X_2$ such that $m \circ \phi_1$
and $\phi_2$ are isotopic.  The \emph{mapping class group}
\begin{equation*}
  \mcg_g :=
  \{\text{positive homeomorphisms } \psi : S_g \rightarrow S_g \} \diagup \text{isotopy}
\end{equation*}
naturally acts on the Teichm\"uller space by precomposition of the marking:
\begin{equation*}
\psi \cdot (X, \phi) := (X, \phi \circ \psi^{-1}).
\end{equation*}
Then the moduli space, as the space of ``unmarked'' hyperbolic surfaces, is
obtained by forgetting the marking, i.e. $\mathcal{M}_g = \mathcal{T}_g \diagup \mcg_g$.

\subsubsection{Length functions}
\label{s:length_functions}

Closed geodesics on a marked surface $(X,\phi)$ are in a natural correspondence with
homotopy-classes of loops on the base surface $S_g$, thanks to the marking
$\phi : S_g \rightarrow X$. Indeed, for any $(X, \phi) \in \mathcal{T}_g$, the marking
$\phi : S_g \rightarrow X$ provides a one-to-one correspondence between $\geod(S_g)$ and
$\geod(X)$. We can therefore define, for a $(X, \phi) \in \mathcal{T}_g$ and $\gamma \in \geod(S_g)$,
the length $\ell_{(X,\phi)}(\gamma) := \ell_X(\phi(\gamma))$ to be the length of the geodesic
representative in the homotopy class $\phi(\gamma)$ on $X$. Note that we will often abusively remove
the mention of the marking, so that we will sometimes write $\ell_X(\gamma)$ for a
$X \in \mathcal{T}_g$ and $\gamma \in \geod(S_g)$; in this case, it is implied that the overall
quantity that we are studying is $\mcg_g$-invariant, so that the marking does not need to be
emphasised.

The mapping class group $\mcg_g$ naturally acts on loops on the base surface
$S_g$, by composition $\psi \cdot \gamma := \psi \circ \gamma$. The orbit of
$\gamma$ for this action is denoted as $\orb(\gamma)$, and the stabilisator
$\Stab(\gamma)$.  We write $\gamma_1 \eqmcg \gamma_2$ if there exists a
$\psi \in \mcg_g$ such that $\psi \cdot \gamma_1 = \gamma_2$, in which case
$\gamma_1$ and $\gamma_2$ are said to have the same \emph{(global) topological
  type}.  This action also extends naturally to an action on multi-curves, or on
families of loops, and we use the same notations in these cases.

\subsubsection{Weil--Petersson form and probability measure}
\label{sec:weil-peterss-volume}

The \emph{Weil--Petersson form} $\omega_{g}^{\mathrm{WP}}$ is a natural
symplectic structure on the Teichm\"uller space $\mathcal{T}_{g}$, which is
invariant by the action of $\mcg_{g}$ and therefore descends to the moduli space
$\mathcal{M}_{g}$ \cite{weil1958}.

A \emph{pair of pants} is a surface of signature $(0,3)$, and a \emph{pair of
  pants decomposition} of $S_g$ is a multi-curve
$(\gamma_1, \ldots, \gamma_{3g-3})$, that cuts $S_g$ into $2g-2$ pairs of
pants. For $(X, \phi) \in \mathcal{T}_g$, after homotopy, this multi-curve is
sent to a decomposition of $X$ in hyperbolic pairs of pants, with boundary
lengths and twists $(\ell_i, \tau_i)_{1 \leq i \leq 3g-3}$. These numbers,
called \emph{Fenchel--Nielsen parameters}, are global coordinates on the
Teichm\"uller space $\mathcal{T}_g \simeq (\R_{>0} \times \R)^{3g-3}$.  Wolpert
proved in \cite{wolpert1981} that Fenchel--Nielsen coordinates are symplectic
coordinates for the Weil--Petersson form:
\begin{equation}
  \label{eq:Wolpert_magic}
  \omega_{g}^{\mathrm{WP}}
  = \sum_{i=1}^{3g-3} \d \ell_i \wedge \d \tau_i.
\end{equation}

As any symplectic form does, the Weil--Petersson form induces a volume form on
the Teichm\"uller space and moduli space, defined by
$\Vol_{g}^{\mathrm{WP}} := (\omega_{g}^{\mathrm{WP}})^{\wedge
  (3g-3)}/(3g-3)!$. This volume form is the Lebesgue measure
$\d \ell_1 \d \tau_1 \ldots \d \ell_{3g-3} \d \tau_{3g-3}$ in Fenchel--Nielsen
parameters.  The total mass of the moduli space is finite, and we shall denote
it as $V_{g}$. As a consequence, we can renormalize the Weil--Petersson volume
form, and hence equip the moduli space $\mathcal{M}_g$ with the
\emph{Weil--Petersson probability measure}
\begin{equation*}
\mathbb{P}_g^{\mathrm{WP}} := \frac{1}{V_g} \Vol_g^{\mathrm{WP}}.
\end{equation*}

\subsubsection{Spaces of bordered surfaces}
\label{s:bordered_teich}

As mentioned in \cref{s:compact_bordered}, we will need to consider not only compact surfaces but
also bordered surfaces for the purposes of this article. The definitions above naturally extend to
define, for $(g,n)$ such that $2g-2+n>0$ and $n>0$, $\x = (x_1,
\ldots, x_n) \in \R_{\geq 0}^n$, the moduli space
\begin{equation*}
  \mathcal{M}_{g,n}(\x)
  := \left\{
    \parbox{9cm}{bordered hyperbolic surface $X$ of signature $(g,n)$ \\
      $\forall i$, the $i$-th component of $X$ has length $x_i$}
  \right\} \diagup \text{isometry}
\end{equation*}
where the quotient is over positive isometries that preserve each individual boundary component
setwise.  Similarly, we fix a base surface $S_{g,n}$ of signature $(g,n)$, which allows us to write
\begin{equation*}
\mathcal{M}_{g,n}(\x) = \mathcal{T}_{g,n}(\x) \diagup \mcg_{g,n},
\end{equation*}
where the Teichm\"uller space $\mathcal{T}_{g,n}(\x)$ is the space of marked bordered hyperbolic
surfaces and $\mcg_{g,n}$ is the mapping class group of $S_{g,n}$ (considering only homeomorphisms
fixing each individual boundary component of $S_{g,n}$ setwise).

In this more general setting, there is also a Weil--Petersson symplectic form
$\omega_{g,n,\x}^{\mathrm{WP}}$ defined on both $\mathcal{M}_{g,n}(\x)$ and
$\mathcal{T}_{g,n}(\x)$, which has the same expression \eqref{eq:Wolpert_magic}
for any pair of pants decomposition $(\gamma_1, \ldots, \gamma_{3g-3+n})$ of the
base surface $S_{g,n}$. The volume form induced by this symplectic structure is
denoted as $\Vol_{g,n,\x}^{\mathrm{WP}}$. The quantity $V_{g,n}(\x)$ denotes the
total mass of the moduli space, with the exception that
$V_{1,1}(x) := \frac{1}{2} \Vol_{1,1,x}(\mathcal{M}_{1,1}(x))$ (this symmetry
constant $1/2$ reflects the existence of an involution symmetry for every
once-holed torus with boundary -- see \cite[Section 2.8]{wright2020}). We shall
omit the mention of the length-vector $\x$ whenever it is equal to
$(0, \ldots, 0)$, i.e. when all boundary components are cusps, hence making
sense of the notations $\mathcal{M}_{g,n}$, $\mathcal{T}_{g,n}$ and $V_{g,n}$.

\subsection{Mirzakhani's integration formula}
\label{s:mirz_int}

Let $g \geq 2$, $k \geq 1$ and $\gamma = (\gamma_1, \ldots, \gamma_k)$ be a
multi-curve on the base surface $S_g$.  For a measurable function
$F : \R_{\geq 0}^k \rightarrow \R$ bounded with compact support (or decaying fast
enough) and an element $X \in \mathcal{M}_g$, we define
\begin{equation}
  \label{eq:def_F_gamma}
  F^{\gamma}(X) := \sum_{(\alpha_1, \ldots, \alpha_k) \in \orb(\gamma)}
  F(\ell_X(\alpha_1), \ldots, \ell_X(\alpha_k)).
\end{equation}
These functions are called \emph{geometric functions}.

For any $(X,\phi) \in \mathcal{T}_g$, cutting the surface $X$ along the
multi-geodesic representative of the multi-curve $\phi(\gamma)$ yields a family of
$q \geq 1$ bordered hyperbolic surfaces. We pick a numbering for these surfaces,
and for $1 \leq i \leq q$ denote as $(g_i,n_i)$ the signature of the $i$-th
surface. If $\x \in \R_{\geq 0}^k$ is a list of values for the respective
lengths of $\phi(\gamma_1), \ldots, \phi(\gamma_k)$ on $X$, then for every $i$,
the lengths of the boundary components of the $i$-th surface is a vector
$\x^{(i)} \in \R_{\geq 0}^{n_i}$. Note that each component of $\x$ is present
exactly twice in the overall family of vectors $(\x^{(i)})_{1 \leq i \leq q}$,
and $\sum_{i=1}^{q} n_i=2k$, because the $\phi(\gamma_i)$s each have two sides.
Then, Mirzakhani's integration formula reads as follows.

\begin{thm}[\cite{mirzakhani2007}]
  \label{thm:int_mirz}
  For $g \geq 3$, the integral of $F^\gamma$ over the moduli space is equal to
  \begin{equation}
    \label{eq:int_mirz}
    \int_{\mathcal{M}_g} F^\gamma(X) \d \Vol_g^{\mathrm{WP}}(X)
    = \int_{\R_{\geq 0}^k} F(\x) \, \prod_{i=1}^q  V_{g_i,n_i}(\x^{(i)}) \,
    \prod_{i=1}^k x_i \d x_i.
  \end{equation}
\end{thm}

\begin{exa}
  \label{exa:formule_mirz_simple}
  Let us demonstrate how we can use \cref{thm:int_mirz} to compute the average~$\av[\mathbf{s}]{F}$,
  defined by the sum \eqref{e:def_av_simple} over all primitive simple closed geodesics.  We define
  the following loops on the base surface $S_g$.
  \begin{itemize}
  \item We take $\gamma_0$ to be a simple loop such that $S_g \setminus \gamma_0$ is
    connected (we call such a loop a \emph{non-separating} loop).
  \item For $1 \leq i \leq g-1$, $\gamma_i$ is a simple loop such that
    $S_g \setminus \gamma_i$ has two connected components: on the left side of
    $\gamma_i$, a surface of signature $(i,1)$, and on the right side, a surface
    of signature $(g-i,1)$.
  \end{itemize}
  Then, any simple (oriented) loop on $S_g$ lies in the orbit of exactly one
  $\gamma_i$ for a $i \geq 0$. Hence,
  \begin{equation*}
    \av[\mathbf{s}]{F}
    = \Ewp{\sum_{\gamma \text{ simple}} F(\ell_X(\gamma))}
    = \frac{1}{V_g} \sum_{i=0}^{g-1} \int_{\mathcal{M}_g} F^{\gamma_i}(X) \d \Vol_g^{\mathrm{WP}}(X).
  \end{equation*}
  We then apply Mirzakhani's integration formula to each of these multi-curves,
  to conclude that
  \begin{equation*}
    \av[\mathbf{s}]{F}
    = \frac{1}{V_g} \int_0^{+ \infty} F(\ell)
    \paren*{V_{g-1,2}(\ell, \ell) + \sum_{i=1}^{g-1} V_{i,1}(\ell) V_{g-i,1}(\ell)} \ell \d \ell.
  \end{equation*}
\end{exa}

\begin{rem}
  \Cref{thm:int_mirz} appears in the literature in various forms, and there is
  always a symmetry factor $c_\gamma \in (0,1]$ in the right hand side of
  \cref{eq:int_mirz}
  \cite{mirzakhani2007,mirzakhani2013,wright2020,wu2022,lipnowski2021}. No such
  constant appears for us due to the combination of the following choices.
  \begin{itemize}
  \item A factor $2^{-M(\gamma)}$, where $M(\gamma)$ is the number of components of
    $S_g \setminus \gamma$ that are of signature $(1,1)$, is removed thanks to our convention
    $V_{1,1}(x) := \frac{1}{2} \Vol_{1,1,x}^{\mathrm{WP}}(\mathcal{M}_{1,1}(x))$.
  \item There is often a symmetry factor $1/\mathrm{Sym}(\gamma)$, which varies
    throughout literature depending on the conventions that are adopted. For
    instance, in \cite{mirzakhani2007}, $\mathrm{Sym}(\gamma)$ is said to be the
    index of the subgroup $\bigcap_{i=1}^k \Stab(\gamma_i)$ of
    $\Stab(\gamma)$. The reason for this discrepancy is that, in
    \cite{mirzakhani2007}, the function $F^\gamma$ is defined by averaging a
    function $F$ that is invariant by permutations, and hence, when $\gamma$ has
    non-trivial symmetries, several terms in the function $F^\gamma$ are
    systematically identical.
  \item Additional factors, depending on the symmetries of $\gamma$ with respect
    to changing orientations of some of its components, appear in
    \cite{wright2020,wu2022,lipnowski2021}. They come from the fact that the
    multi-curves are usually considered to be non-oriented, as opposed to our
    convention (see \cref{rem:convention_oriented}, and
    \cref{rem:constant_not_important} below).
  \item The presence of an additional factor $1/2$ whenever $g=2$, mentioned in \cite{wright2020},
    due to the existence of the hyper-elliptic involution for surfaces of genus~$2$, is the reason
    why we assume that $g \geq 3$.
  \end{itemize}
\end{rem}

\begin{rem}
  \label{rem:constant_not_important}
  In their proof of the $3/16 - \epsilon$ spectral gap result, both
  teams \cite{wu2022,lipnowski2021} rely heavily on the presence of
  a factor $1/2$ in the right hand side of \cref{eq:int_mirz},
  whenever we apply \cref{thm:int_mirz} to a single simple
  non-separating closed geodesic $\gamma$. This argument is
  reproduced in \cref{sec:trivial_eig}. The distinction here comes
  from the fact that $\gamma$ is non-oriented for them, and oriented
  for us. We can compare the two formulas by observing that
  \begin{equation*}
    \sum_{\gamma \text{ non-oriented}} F(\ell_X(\gamma))
    = \frac 12 \sum_{\gamma \text{ oriented}} F(\ell_X(\gamma)).
  \end{equation*}
  Contrarily, our new approach does not require much knowledge on the constants appearing (or not)
  in \cref{thm:int_mirz}.
\end{rem}

\subsection{Estimates on Weil--Petersson volumes}
\label{sec:estim-weil-peterss}

\Cref{thm:int_mirz} allows us to reduce the question of estimating
$\Ewpo[g][F^\gamma]$ to the study of the Weil--Petersson volumes. 
Many estimates are known on the behaviour of
$V_{g,n} = V_{g,n}(0, \ldots, 0)$ in terms of $g$ and $n$
\cite{mirzakhani2013,mirzakhani2015,mirzakhani2019,nie2023}.  We
shall use several of these estimates throughout this article,
referencing them carefully.

In terms of asymptotic expansions, Mirzakhani and Zograf have proved
in \cite{mirzakhani2015} the existence of coefficients
$(a_{k,n})_{k >0}$ and $(b_{k,n})_{k >0}$, for $n \geq 0$, such that for $\ord \geq 0$,
\begin{align}
  \label{eq:asymp_dev_ratio_same_euler}
  \frac{V_{g-1,n+2}}{V_{g,n}} = 1 + \sum_{k=1}^\ord \frac{a_{k,n}}{g^k} + \O[\ord,n]{\frac{1}{g^{\ord+1}}} \\
  \label{eq:asymp_dev_ratio_add_cusp}
  \frac{V_{g,n+1}}{8 \pi^2 g V_{g,n}} = 1 + \sum_{k=1}^\ord \frac{b_{k,n}}{g^k} + \O[\ord,n]{\frac{1}{g^{\ord+1}}}.
\end{align}

As a function of $\x$, Mirzakhani has proven in \cite{mirzakhani2007} that
$V_{g,n}(\x)$ is a polynomial function of degree $6g-6+2n$. The bound \cite[Lemma
3.2]{mirzakhani2013} on its coefficients directly implies the following two
upper bounds:
\begin{align}
   \label{lem:increase_Vgn}
    & V_{g,n}(\x) \leq V_{g,n} \left( 1+ \max_i |x_i| \right)^{6g-6+2n} \\  
  \label{e:increase_Vgn_bound}
    & V_{g,n}(\x) \leq V_{g,n} \exp \div{x_1+ \ldots +x_n}.
\end{align}
The former is good to use for fixed values of $g, n$ while the latter is better-suited to the
description of the large-genus limit.
The first-order approximation of $V_{g,n}(\x)$ in the large-genus limit is well-known (see
\cite[Proposition 3.1]{mirzakhani2019} and \cite[Proposition 2.5]{anantharaman2022}):
\begin{equation}
  \label{e:Vgn_first_order}
  \frac{V_{g,n}(\x)}{V_{g,n}}
  = \prod_{i=1}^n \frac{2}{x_i} \sinh \div{x_i}
  + \O[n]{\frac{1+\max_i |x_i|}{g} \exp \div{x_1+\ldots+x_n}}.
\end{equation}
In our previous paper, we have shown the following asymptotic expansion, which
will be useful for expanding the averages $\av[T]{F}$.

\begin{thm}[{\cite[Corollary 1.4]{anantharaman2022}}]
  \label{thm:expansion_volume}
  Let $n \geq 1$. There exists a unique family $(v_{k,n})_{k \geq 0}$ of
  functions $\R_{\geq 0}^n \rightarrow \R$ such that, for any order $\ord \geq 0$, any genus
  $g \geq 1$, any $\x \in \R_{\geq 0}^n$,
  \begin{equation}
    \label{eq:expansion_volumes}
    \frac{V_{g,n}(\x)}{V_{g,n}}
    = \sum_{k=0}^\ord \frac{v_{k,n}(\x)}{g^k}
    + \O[\ord,n]{ \frac{(1+\max_i |x_i|)^{3\ord+1}}{g^{\ord+1}} \exp \div{x_1+\ldots+x_n}}.
  \end{equation}
  Furthermore, for any $k \geq 0$, the function $v_{k,n}$ is a linear
  combination of functions
  \begin{equation}
    \label{eq:expansion_volumes_form}
    \x \mapsto \prod_{i=1}^n x_i^{2k_i} \prod_{i \in V_+} \cosh \div{x_i}
    \prod_{i \in V_-} \frac{1}{x_i}\sinh \div{x_i}
  \end{equation}
  where $(k_i)_{1 \leq i \leq n}$ are integers and $V_\pm$ are two
  disjoint subsets of $\{1, \ldots n\}$.
\end{thm}

\begin{rem}
  The fact that the powers $x_i^{2k_i}$ in \cref{eq:expansion_volumes_form} are only even is not
  explicitly stated in \cite{anantharaman2022}, but comes as a straightforward consequence of the
  fact that $V_{g,n}(\x)$ is even in every variable. 
\end{rem}


\section{Friedman--Ramanujan functions}
\label{sec:def_FR}

In this section, we introduce and study the main object of this article, Friedman--Ramanujan
functions. We explain in \cref{sec:simple-geod} how these functions naturally appear in random
hyperbolic geometry. We prove their stability by convolution in \cref{sec:stab-conv}, and explain
their relevance to the spectral gap question in \cref{sec:canc-selb-trace}.

\subsection{Definition and notations}
\label{sec:definition}

\begin{defa}
  \label{def:FR}
   {
  A locally integrable function $f: \R_{\geq 0} \rightarrow \C$ is said to be a
  \emph{Friedman--Ramanujan function} if there exists a polynomial $p \in \C[X]$ and constants
  $c>0$, $\rN \geq 1$ such that
  \begin{equation}
    \label{eq:def_FR}
    \forall \ell \geq 0,
    \quad \abso{f(\ell)-p(\ell) \, e^{\ell}}
    \leq c \, (\ell+1)^{\rN-1} e^{\frac \ell 2}.
  \end{equation}
  It is said to be a \emph{Friedman--Ramanujan function in the weak sense} if there exists a
  polynomial $p \in \C[X]$ and constants $c>0$, $\rN \geq 1$ such that
  \begin{equation}
    \label{eq:def_FR_w}
    \forall \ell \geq 0, \quad
    \int_{0}^{\ell} \abso{f(s) - p(s) \,e^s}
    \d \ell \leq c \, (\ell+1)^{\rN-1} e^{\frac \ell 2}.
  \end{equation}}
\end{defa}

Of course, these sets of functions form two vector spaces, that we denote as
$\FR$ and $\FR_w$ respectively. As the name suggests, the strong definition
implies the weak one. 
 
If $f$ is a Friedman--Ramanujan function (weakly or strongly), then the polynomial~$p$ satisfying
the definition is uniquely defined.  The term $p(\ell) \, e^\ell$ is called the \emph{principal
  term} of $f$, and~$p$ its \emph{polynomial}.  The space of Friedman--Ramanujan functions with no
principal term, also called \emph{remainders}, are denoted as $\FRrem$ and $\FRrem_w$.

In the  following, it will be convenient to split the spaces $\FR$, $\FRrem$ (and their weak
versions) more precisely depending on the exponents appearing. 

{
\begin{nota}
  For $\rK \geq 0$, $\rN \geq 1$, we denote as $\FR^{\rK,\rN}$ and $\FR_w^{\rK,\rN}$ the set of
  Friedman--Ramanujan functions (strong and weak, respectively), of polynomial of degree $< \rK$,
  and satisfying \eqref{eq:def_FR} or \eqref{eq:def_FR_w} with the constant $\rN$. We shall denote
  as $\FRrem^\rN=\cF^{0,\rN}$ and $\FRrem_w^\rN=\FR_w^{0,\rN}$ the sets of remainders dominated by $(\ell+1)^{\rN-1} e^{\frac \ell 2}$.
\end{nota}
}

For the sake of convenience in our following estimates, we introduce a family of norms on~$\FR$ and
$\FR_w$, using the $\ell^\infty$-norm $\| \, \cdot \,\|_{\ell^\infty}$ on the set of polynomials.

{
\begin{defa}
  We define the norm $\| \cdot \|_{\FR^{\rK,\rN}}$ on $\FR^{\rK,\rN}$ by setting
  \begin{equation}
    \label{eq:def_norm_FR}
    \| f \|_{\FR^{\rK,\rN}} := \| p \|_{\ell^\infty}
    + \sup_{\ell \geq 0} \frac{|f(\ell) - p(\ell) e^\ell|}{(\ell+1)^{\rN-1}e^{\ell/2}}
  \end{equation}
  for any Friedman--Ramanujan function $f$ of polynomial $p$.  We similarly define the weak norm
  \begin{equation}
    \label{eq:eq_def_norm_FR_w}
    \| f \|_{\FR^{\rK,\rN}}^w := \| p \|_{\ell^\infty}
    + \sup_{\ell \geq 0} \frac{\int_0^\ell|f(s) - p(s) e^s| \d s}{(\ell+1)^{\rN-1}e^{\ell/2}}.
  \end{equation}
\end{defa}}

\subsection{Motivation to geodesic counting: the case of simple geodesics}
\label{sec:simple-geod}

One of the motivations to study Friedman--Ramanujan functions is that they appear
naturally when counting closed geodesics on random hyperbolic surfaces
(or closed paths on random $d$-regular graphs, for Friedman).  Let us illustrate
this in the most elementary case, the counting of \emph{simple} closed
geodesics.

We saw in \cref{exa:formule_mirz_simple} that Mirzakhani provided an explicit
formula for a function $V_g^{\mathbf{s}} : \R_{>0} \rightarrow \R$ such that,
for any bounded measurable function $F$ with compact support,
\begin{equation*}
  \av[\mathbf{s}]{F} = 
  \Ewp{\sum_{\gamma \text{ simple}} F(\ell_X(\gamma))}
  =  \frac{1}{V_g} \int_0^{+ \infty} F(\ell) V_g^{\mathbf{s}}(\ell) \d \ell.
\end{equation*}
We prove the following.

\begin{prp}
  \label{prop:simple}
  There exists a unique family of functions $(f_k^{\mathbf{s}})_{k \geq 0}$ such
  that, for any integer $\ord \geq 0$, any $\ell > 0$, any large enough $g$,
  \begin{equation}
    \label{e:asymp_exp_simple}
     \frac{V_g^{\mathbf{s}}(\ell)}{V_g}
     = \sum_{k=0}^\ord \frac{f_k^{\mathbf{s}}(\ell)}{g^k}
     + \O[\ord]{\frac{(\ell+1)^{c_\ord} e^\ell}{g^{\ord+1}}}.
  \end{equation}
  Furthermore, for all $k$, $\ell \mapsto \ell f_k^{\mathbf{s}}(\ell)$ is a
  Friedman--Ramanujan function.
\end{prp}

In other words, Friedman--Ramanujan functions naturally appear when computing the terms of the
asymptotic expansion of $\av[\mathbf{s}]{F}$. This result means that \ref{chal:FR_type} holds for the
local topology ``simple''.

\begin{rem}
  \label{rem:first_order_simple}
  One can show that $f_0^{\mathbf{s}}(\ell) = \frac{4}{\ell} \sinh^2 \div{\ell}$
  by using the expression of $V_g^{\mathbf{s}}$ and the first-order estimates
  \eqref{eq:asymp_dev_ratio_same_euler} and \eqref{e:Vgn_first_order}, as well
  as \cite[Lemma 3.3]{mirzakhani2013}. It is clear that
  $\ell \mapsto \ell f_0^{\mathbf{s}}(\ell) = 4 \sinh^2 \div{\ell}$ is a
  Friedman--Ramanujan function.
\end{rem}

\begin{proof}
  Let us fix a $\ord \geq 0$. We recall that the expression of $V_g^{\mathbf{s}}$
  is:
  \begin{equation}
    \label{eq:expression_Vgsimple}
    V_g^{\mathbf{s}}(\ell) =  \ell \, V_{g-1,2}(\ell, \ell)
    +  \sum_{i=1}^{g-1} \ell \, V_{i,1}(\ell) V_{g-i,1}(\ell).
  \end{equation}
  Let us break down this expression and examine its terms.

  We first observe that we can reduce the number of terms in
  \cref{eq:expression_Vgsimple} so that it only depends on $\ord$, and not on
  $g$. Indeed, applying \eqref{e:increase_Vgn_bound} and \cite[equation
  (3.19)]{mirzakhani2013} yields:
  \begin{equation*}
    \sum_{\frac \ord 2 +1 \leq i \leq g-\frac \ord 2 -1}
    V_{i,1}(\ell) V_{g-i,1}(\ell) = \O[\ord]{\frac{e^{\ell} V_g}{g^{\ord+1}}}.
  \end{equation*}
  Hence, provided that $g$ is large enough, we can rewrite \eqref{eq:expression_Vgsimple} as 
  \begin{equation}
    \label{e:exp_vsimple_comvenient2}
    \frac{\ell V_g^{\mathbf{s}}(\ell)}{V_g}
    = \frac {V_{g-1,2}}{V_g}  \frac{\ell^2 V_{g-1,2}(\ell, \ell)}{V_{g-1,2}}
    +  2 \sum_{i=1}^{\lceil \frac \ord2 \rceil} \ell V_{i,1}(\ell) \frac{V_{g-i,1}}{V_{g}}
    \frac{\ell V_{g-i,1}(\ell)}{V_{g-i,1}} + \O[\ord]{\frac{\ell \,e^\ell}{g^{\ord+1}}}.
  \end{equation}
  Note that we have used the symmetry of the sum to only have terms for which $i \leq \lceil \frac
  \ord 2 \rceil$. 

  Now, we observe that \cref{thm:expansion_volume} taken with $n=1$ and $2$
  directly implies that for any fixed $i$,
  \begin{equation*}
    \frac{\ell V_{g-i,1}(\ell)}{V_{g-i,1}}
    \qquad \text{and} \qquad
    \frac{\ell^2 V_{g-1,2}(\ell, \ell)}{V_{g-1,2}}
  \end{equation*}
  admit an asymptotic expansion of the desired form, with all coefficients
  belonging in $\FR$. Indeed, after multiplication by $\ell^n$, the coefficients
  of these expansions are proven to be linear combinations of
  functions of the form
  \begin{itemize}
  \item $\ell^{2k+1} \cosh \div{\ell}$, $\ell^{2k} \sinh \div{\ell}$
    and $\ell^{2k+1}$ for $n=1$;
  \item $\ell^{2k+2} \cosh^2 \div{\ell}$,
    $\ell^{2k+1} \cosh \div{\ell} \sinh \div{\ell}$,
    $\ell^{2k} \sinh^2 \div{\ell}$, $\ell^{2k+2} \cosh \div{\ell}$,
    $\ell^{2k+1} \sinh \div{\ell}$ and $\ell^{2k+2}$ for $n=2$;
  \end{itemize}
  for $k$ non-negative integers, all of which are Friedman--Ramanujan functions.
  
  We know by equations \eqref{eq:asymp_dev_ratio_same_euler} and
  \eqref{eq:asymp_dev_ratio_add_cusp} that the quantities $V_{g-1,2}/V_g$ and
  $V_{g-i,1}/V_g$ (for any fixed $i$) have an asymptotic expansion in powers of
  $1/g$. Also, for any fixed $i$, the function
  $\ell \mapsto \ell \, V_{i,1}(\ell)$ is a polynomial function. This is all we
  need to conclude to the existence and form of the asymptotic expansion.

  Now that the existence of an expansion is established, the uniqueness is
  obtained by fixing an arbitrary value of $\ell$ and using the uniqueness of
  asymptotic expansions in powers of~$1/g$.
\end{proof}

\subsection{Stability by convolution}
\label{sec:stab-conv}

When proving Alon's conjecture, Friedman proved a statement analogous to \cref{prop:simple}, for
more complicated paths. This is achieved by using a decomposition of a general path into simple
paths, together with the result for simple paths. In doing so, a key argument is the stability of the class of
$d$-Ramanujan functions by \emph{convolution}, Theorem 7.2 in \cite{friedman2003}. Indeed, as
represented in \cref{fig:convolution_graph}, the length $\ell$ of a non-simple path can be written
as a sum of lengths $\ell_1, \ell_2$ of simpler closed paths. In the expectation for a random graph,
this becomes a sum over all possible values of $\ell_1, \ell_2$ such as $\ell = \ell_1+\ell_2$,
i.e. a convolution.

\begin{figure}[h]
  \centering
    \includegraphics[scale=0.5]{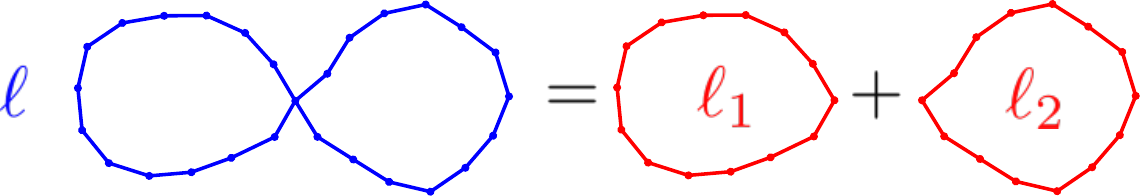}
  \caption{Decomposition of a figure-eight on a graph.}
  \label{fig:convolution_graph}
\end{figure}

In our new setting, we shall also prove that the class $\FR$ is
stable by convolution.  For two continuous functions
$f_1, f_2 : \R_{>0} \rightarrow \C$, we define
\begin{equation*}
  f_1 * f_2(\ell) := \int_{0}^\ell f_1(t)f_2(\ell-t) \d t = \int_{0}^\ell f_1(\ell-t)f_2(t) \d t.
\end{equation*}
Let us prove the following, which is a direct adaptation of the proof
given by Friedman in the case of graphs \cite[Theorem
7.2]{friedman2003}.

\begin{prp}
  Let $f_1, f_2 \in \FR$. Then, $f_1* f_2\in \FR$.
\label{p:conv}
\end{prp}

\begin{rem}
  In the following, we will not use Proposition \ref{p:conv} exactly
  as such: unfortunately, in hyperbolic geometry, when we
  ``concatenate'' two closed geodesics by creating an intersection
  point, the length of the newly created closed geodesic is not the
  sum of the two original length (see \cref{fig:eight_pop}). This is
  a major difference between negative (but finite) curvature, and
  curvature $-\infty$ (i.e. the case of graphs).  However, we
  believe the proof is quite enlightening in its simplicity -- very
  similar techniques, yet more complex, are used in
  \cref{sec:generalised_convolutions}.
\end{rem}

\begin{proof}
  Write $f_i (\ell) = p_i(\ell) \, e^\ell + r_i(\ell)$ with
  $|r_i(\ell)| \leq c_i (\ell+1)^{\rN_i-1} \, e^{\frac \ell 2}$. Then,
  \begin{align*}
    f_1* f_2=(p_1 \exp)* (p_2 \exp)+ (p_1 \exp)*  r_2+ (p_2 \exp)*  r_1 + r_1* r_2.
  \end{align*}
  First, we observe that
  \begin{equation*}
    (p_1 \exp)* (p_2 \exp)(\ell) =
    \int_{0}^\ell p_1(t) \, e^t \, p_2(\ell-t) \, e^{\ell-t} \d t
    = P(\ell) \, e^\ell
  \end{equation*}
  where $P = p_1* p_2$ is a polynomial.
  Next, we have
  \begin{align*}
    |r_1* r_2(\ell)|\leq c_1c_2 \, e^{\frac \ell 2}\int_0^\ell (\ell-t+1)^{\rN_1-1} (t+1)^{\rN_2-1} \d t
    \leq c_1 c_2 e^{\frac \ell 2} (\ell+1)^{\rN_1+\rN_2-1}.
  \end{align*}
  Finally, we examine the crossed term $(p_1 \exp)* r_2$.
  \begin{align*}
    (p_1 \exp)*   r_2(\ell)
    &= e^\ell \int_0^\ell p_1(\ell-t) \, r_2(t) \, e^{-t} \d t\\
    & = e^\ell \int_0^\infty p_1(\ell-t) \, r_2(t) \, e^{-t} \d t
      - e^\ell \int_\ell^\infty p_1(\ell-t) \, r_2(t) \, e^{-t} \d t.
  \end{align*}
  The function $\ell \mapsto \int_0^\infty p_1(\ell-t) \, r_2(t) \, e^{-t} \d t$
  is a polynomial function. For the last term, it is bounded by
  \begin{align}\label{e:PR}
    \|p_1\|_{\ell^\infty} c_2 \, e^\ell \int_\ell^\infty (t-\ell+1)^{\deg p_1} (t+1)^{\rN_2-1}e^{- \frac t 2} \d t.
  \end{align}
  By an integration by parts, \eqref{e:PR} is a function of the form
  $\ell \mapsto q(\ell)e^{\frac \ell 2}$ with $q$ a polynomial, and in particular is bounded by
  $c' (\ell+1)^{\rN'-1} e^{\frac \ell 2}$ for constants $c'>0$, $\rN'\geq 1$. The same argument of
  course applies to $(p_2 \exp)* r_1$ and shows the announced result.
\end{proof}

\subsection{Cancellations in the Selberg trace formula}
\label{sec:canc-selb-trace}

We have seen in \cref{sec:simple-geod} that Friedman--Ramanujan functions arise
naturally when computing expectations of sums over (simple) closed
geodesics. The aim of this section is now to show how this information can be
used, in particular in the study of the spectrum of the
Laplacian through the Selberg trace formula.

This section is entirely dedicated to the link between the length spectrum and the Laplacian spectrum, so we invite the reader only
interested in our new geometric techniques to skip it at first read. Indeed, while the ideas and
results that we present below are a motivation for many results developed below, they only come into
play in Sections \ref{sec:second-order-term} and the second part of this article, when we
  actually study the spectral gap of random hyperbolic surfaces.

\subsubsection{The Selberg trace formula}
\label{sec:selb-trace-form}

This beautiful formula, proven by Selberg in \cite{selberg1956}, relates the
spectrum of the Laplacian on a hyperbolic surface to the lengths of all its
closed geodesics. It reads, for a smooth even function $h : \R \rightarrow \R$,
\begin{equation}
  \label{eq:selberg}
  \sum_{j=0}^{+ \infty} \hat{h}(r_j(X))
  =  (g-1) \int_\R \hat{h}(r) \tanh (\pi r) r \d r
  + \sum_{\gamma \in \geod(X)} \sum_{k=1}^{+ \infty}
  \frac{\ell_X(\gamma) \, h(k \ell_X(\gamma))}{2 \sinh \div{k \ell_X(\gamma)}}
\end{equation}
where for all $j$, $r_j(X) \in \R \cup i [- \frac 12, \frac 12]$ is
a solution of $\lambda_j(X) = \frac 14 + r_j(X)^2$, and the Fourier
transform $\hat{h}$ is defined by
$\hat{h}(r) := \int_{\R} h(\ell) \, e^{-i r \ell} \d \ell$.  The
formula is valid for a class of ``nice'' functions $h$; for our
purposes, we will only consider functions $h$ of compact support, in
which case the Selberg trace formula holds and both sums are absolutely
convergent \cite[Theorem 5.8]{bergeron2016}.

Let us briefly describe the three terms of \cref{eq:selberg}.
\begin{itemize}
\item The left hand side term is called the \emph{spectral side} of the trace
  formula, and we will use this term to try and access information on the
  spectral gap $\lambda_1 = \frac 14 + r_1^2$.
\item The first term on the right hand side is called the \emph{topological
    term}, or \emph{integral term}. The name ``topological'' refers to the fact
  that this term does not depend on the hyperbolic structure on the surface $X$,
  but only on its genus $g$. In particular, when studying random hyperbolic
  surfaces of genus $g$, this term is deterministic.
\item The last term is the \emph{geometric term}, in which appear
  every closed geodesic on the surface. We draw the attention to the
  fact that non-simple geodesics appear here, and there is a priori
  no known similar formula including only simple geodesics. Dealing
  with non-simple closed geodesics in the Selberg trace formula is
  one of the challenges we address in this article.
\end{itemize}

\subsubsection{Spectral gap v.s. exponential growth}
\label{sec:spectr-gap-expon}

Due to the presence of a Fourier transform, and summations on the whole spectrum
and all closed geodesics, the link made by the Selberg trace formula between
geometry and spectrum is quite intricate, and using this formula requires a good
choice of test function. A classic approach to access information on the
spectral gap $\lambda_1$ using the Selberg trace formula, used in
\cite{wu2022,lipnowski2021} notably, is to observe that, if
$\lambda_1 = \frac 14 + r_1^2 < \frac 14$, then $r_1 \in i\R$, and hence
\begin{equation*}
  \hat{h}(r_1) = 2 \int_0^{+ \infty} h(\ell) \cosh(\ell \abso{r_1}) \d \ell.
\end{equation*}
The Fourier transform is therefore an integral against a growing exponential, at
the rate $\abso{r_1} = \sqrt{\frac 14 - \lambda_1}$, rather than an oscillatory
term.

We make the following choice of test function, similarly to
\cite{wu2022,lipnowski2021}, that will allow us to exploit this
exponential increase.

\begin{nota}
  \label{nota:h_L}
  Let $h : \R \rightarrow \R_{\geq 0}$ be a smooth even function, with compact
  support $[-1,1]$, such that $\hat{h}$ is non-negative on
  $\R \cup i[- \frac 12, \frac 12]$. For any $L \geq 1$, let
  $h_L(\ell) := h(\frac{\ell}{L})$.
\end{nota}
\begin{rem}
  Such a $h$ can be obtained by taking a square convolution
  $h := H * H$ of a smooth function $H \geq 0$ supported on
  $[-\frac 12, \frac 12]$, so that $\hat{h} = \hat{H}^2 \geq 0$.
\end{rem}

\begin{rem}
  The scaling parameter $L$ plays the role of a
  length-scale. Indeed, since the support of $h_L$ is $[-L,L]$, only
  geodesics of length $\leq L$ will contribute to the geometric term
  of the Selberg trace formula applied to $h_L$.

  By analogy with graphs, we expect the natural length-scale we need to consider for the spectral
  gap problem to be $L = A \log(g)$, where $A \geq 1$ is a fixed constant. In
  \cite{wu2022,lipnowski2021}, the value $A=4$ is used to obtain the spectral gap $3/16 -
  \epsilon$. We take $A=6$ to prove that $\lambda_1 \geq 2/9 - \epsilon$ in
    \cite{anantharaman_29}, and see that arbitrarily large values of $A$ are required to reach
  $1/4 - \epsilon$.
\end{rem}

The following lemma allows us to relate the size of the spectral gap of a
surface $X$ with the rate of exponential growth of the term $\hat{h}_L(r_1(X))$
of the Selberg trace formula.

\begin{lem}
  \label{lem:growth_h_r1}
  Let $\alpha \in (0, \frac 12)$. For any
  $0 < \epsilon < \frac 1 4 - \alpha^2$, there exists a constant
  $C_{\alpha,\epsilon} > 0$ (depending on $h$) such that, for any
  hyperbolic surface $X$, any $L \geq 1$,
  \begin{equation}
    \label{eq:growth_h_r1}
    \lambda_1(X) \leq \frac 14 - \alpha^2 - \epsilon
    \quad \Rightarrow \quad
    \hat{h}_{L}(r_1(X)) \geq C_{\alpha,\epsilon} \, e^{(\alpha + \epsilon) L}.
  \end{equation}
\end{lem}

\begin{proof}
  If $\lambda_1 \leq \frac 14 - \alpha^2 - \epsilon < \frac 14$, then in
  particular $r_1 \in i \R$. Then, by definition of $h_L$,
  \begin{equation}
    \label{eq:fourier_h_L}
    \hat{h}_L(r_1)
    = 2 L \int_0^1 h(\ell) \, e^{\abso{r_1} L \ell} \d \ell.
  \end{equation}
  The hypothesis on $\lambda_1$ further implies that
  $\abso{r_1} \geq \sqrt{\alpha^2 + \epsilon}$.  For $\alpha < 1/2$,
  we have that $\frac 14-\alpha^2 < 1-2\alpha$, and hence
  $\epsilon < 1 - 2 \alpha$, which implies
  $\alpha + \epsilon < \sqrt{\alpha^2 + \epsilon}$. Hence,
  \begin{equation*}
    \hat{h}_L(r_1)
    \geq 2 \int_{\frac{\alpha + \epsilon}{\sqrt{\alpha^2 + \epsilon}}}^1
    h(\ell) \, e^{\sqrt{\alpha^2 + \epsilon} \, L \ell} \d \ell
    \geq C_{\alpha, \epsilon} \, e^{(\alpha + \epsilon)L}
  \end{equation*}
  since $h \geq 0$ by the hypothesis in \cref{nota:h_L}. The implied
  constant can be taken to be
  $C_{\alpha, \epsilon} := 2 \int_{(\alpha +
    \epsilon)/\sqrt{\alpha^2 + \epsilon}}^1 h(\ell) \d \ell$, which
  is positive because the support of the non-negative function $h$
  is exactly $[-1,1]$.
\end{proof}

The following handy lemma clears up the Selberg trace formula, so that we can
focus only on the terms which shall be crucial to our analysis.

\begin{lem}
  \label{lem:bound_r1}
  Let $L \geq 1$, and $F : \R \rightarrow \R$ be a smooth even function, supported
  on $[-L,L]$, with $\hat{F} \geq 0$ on $\R \cup i [- \frac 12, \frac
  12]$. Then, for any $g \geq 2$,
  \begin{equation*}
    \Ewpo \brac*{\hat{F}(r_1(X))}
    \leq \avb{\ell \, F(\ell) \, e^{- \frac \ell 2}}
    + C_{F} \, L^2 g
  \end{equation*}
  for a constant
  $C_{F} := c \norm{F}_\infty
  + \norm{r \hat{F}(r)}_{\infty}< + \infty$, where
  $c$ is a universal constant independent of $g$, $L$ and $F$.
\end{lem}

\begin{rem}
  The constant $C_{F}$ is finite because $F$ is compactly supported, and hence
  $\hat{F}$ decays faster than any polynomial at infinity.
\end{rem}

\begin{rem}
  The function $h_L$ defined in \cref{nota:h_L} clearly satisfies the hypotheses
  of the lemma. We have formulated the result in terms of a function $F$ with
  precise hypotheses because we shall later apply it to other test functions.
\end{rem}

Lemmas \ref{lem:growth_h_r1} and \ref{lem:bound_r1} provide us with a strategy
to prove probabilistic lower bounds on $\lambda_1$. First, we use
\cref{lem:growth_h_r1} to write
\begin{equation*}
  \Pwp{\lambda_1 \leq \frac{1}{4} - \alpha^2 - \epsilon}
  \leq \Pwp{\hat{h}_L(r_1) \geq C_{\alpha,\epsilon} \,  e^{(\alpha + \epsilon) L}}.
\end{equation*}
Using Markov's inequality allows us to obtain that
\begin{equation*}
  \Pwp{\lambda_1 \leq \frac{1}{4} - \alpha^2 - \epsilon}
  \leq \frac{\Ewpo \brac*{\hat{h}_L(r_1)}}{C_{\alpha,\epsilon} \, e^{(\alpha + \epsilon) L}}.
\end{equation*}
We can then use \cref{lem:bound_r1} to obtain that, for $L := A \log(g)$,
\begin{equation}
  \label{eq:trace_method_before_canc}
  \Pwp{\lambda_1 \leq \frac{1}{4} - \alpha^2 - \epsilon}
  = \O[\alpha, \epsilon, A]{
    \frac{\avb{\ell h_L(\ell) \, e^{- \frac \ell 2}}}{g^{(\alpha + \epsilon) A}}
    + (\log g)^2 g^{1-(\alpha+\epsilon) A}},
\end{equation}
since the constant $C_{h_L}$ can be bounded uniformly in $L$.

Let us pick a value of $A > 1/(\alpha + \epsilon)$, such as
$A := 1 / \alpha$, so that
$(\log g)^2 g^{1-(\alpha+\epsilon)A} \rightarrow
0$. \Cref{eq:trace_method_before_canc} then reduces the spectral gap
problem to proving that the geometric average
$\av{\ell h_L(\ell) \, e^{- \frac \ell2}}$ is negligible compared to
$e^{(\alpha + \epsilon) L}$.

As a conclusion, the trace method allows to bound
$ \Pwp{\lambda_1 \leq \frac{1}{4} - \alpha^2 - \epsilon}$ in terms
of the geometric average $\av{\ell h_L(\ell) \, e^{- \ell/2}}$ for
$L=A \log(g)$. The parameter~$\alpha$ needs to be small in order to
obtain a spectral gap close to $1/4$. This naturally requires to
look a length scale $L = A \log(g)$ with $A \geq 1/\alpha$, due to
the presence of linear terms in the Selberg trace formula.

\begin{proof}[Proof of \cref{lem:bound_r1}]
  By the positivity hypothesis on $\hat{F}$, $\Ewpo \brac{\hat{F}(r_1)}$ is
  smaller than the expectation of the Selberg trace formula. Its integral term
  is independent of the surface, and smaller than $g \norm{r \hat{F}(r)}_\infty$
  because $\tanh \leq 1$. We hence are left with comparing
  \begin{equation}
    \label{eq:proof_lemr1_exp_to_look}
    \Ewp{\sum_{\gamma \in \geod(X)}
      \sum_{k=1}^{+ \infty}
      \frac{\ell_X(\gamma) \, F(k \ell_X(\gamma))}{2 \sinh \div{k \ell_X(\gamma)}}}
  \end{equation}
  with the average $\av{\ell F(\ell) \, e^{- \frac \ell 2}}$.

  Let us first prove a bound on the sum over $k \geq 2$, i.e. the
  sum for non-primitive geodesics.  Note that
  $x/\sinh \div{x} = \O{(x+1) \, e^{-\frac x 2}}$ for $x>0$. Hence,
  for any $0 < \ell \leq L$,
  \begin{equation*}
    \sum_{k=2}^{+ \infty} \frac{\ell \, F(k \ell)}{2 \sinh \div{k \ell}}
    = \O{L \| F \|_\infty \sum_{k=2}^{+ \infty} e^{- \frac{k \ell}{2}}}
    = \O{L \| F \|_\infty \frac{e^{- \ell}}{\min(\ell, 1)}}
  \end{equation*}
  because $\sum_{k=2}^{+ \infty} e^{- k \ell/2} = e^{-
    \ell}/(1-e^{- \ell/2})$ and $1-e^{- \ell/2} \geq c \min(\ell,1)$
  for a $c>0$.
  
  We recall that $F$ is identically equal to zero outside $[-L,L]$. Hence, for a
  compact hyperbolic surface $X$ of genus $g$, we can apply the previous
  estimate to each $\ell = \ell_X(\gamma)$ appearing in the contribution of $X$
  to the expectation \eqref{eq:proof_lemr1_exp_to_look} and deduce
  \begin{equation}
    \label{eq:proof_lemr1_one_surf}
    \sum_{\gamma \in \geod(X)}
    \sum_{k=2}^{+ \infty} \frac{\ell_X(\gamma) \, F(k \ell_X(\gamma))}{2 \sinh \div{k
        \ell_X(\gamma)}}
    = \O{\frac{L \| F \|_\infty}{\min(\mathrm{sys}(X),1)}
    \sum_{\substack{\gamma \in \geod(X) \\ \ell_X(\gamma) \leq L}}
    e^{- \ell_X(\gamma)}}
  \end{equation}
  where $\mathrm{sys}(X)$ is the length of the systole of $X$, its shortest
  closed geodesic.  We bound uniformly in $X$ the sum above:
  \begin{align*}
    \sum_{\substack{\gamma \in \geod(X) \\ \ell_X(\gamma) \leq L}}
    e^{- \ell_X(\gamma)}
     \leq \sum_{j=0}^{\lceil L \rceil} e^{-j} \, \#\{\gamma \in \geod(X) \, : \, j \leq
    \ell_X(\gamma) < j+1\}
     = \O{L g}
  \end{align*}
  by \cref{lem:bound_number_closed_geod}. As a consequence, taking
  the average of \cref{eq:proof_lemr1_one_surf} yields
  \begin{equation*}
    \Ewp{\sum_{\gamma \in \geod(X)}
    \sum_{k=2}^{+ \infty} \frac{\ell_X(\gamma) \, F(k \ell_X(\gamma))}{2 \sinh \div{k
        \ell_X(\gamma)}}}
    =\O{L^2 g \|F\|_\infty \, \Ewp{\frac{1}{\min(\mathrm{sys}(X), 1)}}}. 
    \end{equation*}
    Mirzakhani proved in \cite[Corollary 4.2]{mirzakhani2013} that
    the expectation above is finite and bounded uniformly in $g$,
    which is enough to conclude for this term.

    All that is left to do to conclude is to substitute the
    $2 \sinh\div{\ell}$ by the very close value $e^{\ell /2}$ in the
    term $k=1$ of \cref{eq:proof_lemr1_exp_to_look}. Because
    \begin{equation*}
      \frac{1}{2 \sinh \div{\ell}} - e^{- \frac \ell 2} = \frac{1}{2 \, e^\ell \, \sinh \div{\ell}}
      \leq \frac{e^{- \ell}}{\ell},
    \end{equation*}
    the error in doing so is bounded by
    \begin{equation*}
      \|F\|_\infty \,
      \Ewp{\frac{1}{\min(\mathrm{sys}(X), 1)}
      \sum_{\substack{\gamma \in \geod(X) \\ \ell_X(\gamma) \leq L}} 
      e^{- \ell_X(\gamma)}} 
    \end{equation*}
    which we proved is bounded by a constant multiple of $gL \|F\|_\infty$.
\end{proof}

\subsubsection{Necessity of expansions in powers of $1/g$}

In \cite{wu2022,lipnowski2021}, Wu--Xue and Lipnowski--Wright
obtained the spectral gap $3/16 - \epsilon$ by using the method
above. More precisely, the intermediate value $3/16$ arises because
the average $\av{\ell h_L(\ell) \, e^{- \frac \ell 2}}$ is estimated
at the \emph{leading order} as $g \rightarrow + \infty$, i.e. the
computations are made up to errors with a $1/g$ decay. We now
explain why, in order to reach the optimal spectral gap
$1/4 - \epsilon$, we need to go further and perform \emph{asymptotic
  expansions} of averages $\av{F}$ in powers of $1/g$, which is one
of the core objectives of this article.

Let us imagine that we are able to compute, exactly, any average
$\av{F}$ up to error terms decaying as $1/g^{\ord+1}$ for a $\ord \geq 0$.
This means that we will know the average
$\av{\ell h_L(\ell) \, e^{- \frac \ell 2}}$ up to errors of size
roughly $e^{L/2}/g^{\ord+1}$, because the number of primitive closed
geodesics shorter than $L$ behaves like $e^L/L$, by \cite{huber1974}
(the factor $e^{L/2}$ comes from the presence of the exponential
decay $e^{-\ell/2}$ in the average).

We recall that we saw in \cref{eq:trace_method_before_canc} that, for $\alpha > 0$, in order to
prove that $\Pwp{\lambda_1 \leq 1/4 - \alpha^2 - \epsilon}$ goes to $0$ as $g \rightarrow + \infty$,
we need to prove that $\av{\ell h_L(\ell) \, e^{- \frac \ell 2}}$ grows slower than $e^{\alpha L}$,
for $L = A \log(g)$ and $A \geq 1/\alpha$. In particular, we will need the error term
  $\approx e^{L/2}/g^{\ord+1}$ produced when estimating this average to be smaller than $e^{\alpha L}$,
  which requires to assume that $A/ 2 - \ord - 1 \leq \alpha A$. Hence, the hypotheses made so far on
  the parameters in the trace method can be listed as:
\begin{equation*}
  \frac A 2 - \ord -1 \leq \alpha A
  \quad \text{and} \quad
  A \geq \frac 1 \alpha.
\end{equation*}
These conditions imply that {$\alpha \geq 1/(2(\ord+2))$}, which is a
lower bound on the precision~$\alpha$ that can be attained.

In other words, computing asymptotic expansions with remainders
decaying like {$1/g^{\ord+1}$} puts a natural limitation on the spectral
gap $\lambda_1 \geq 1/4 - \alpha^2 - \epsilon$ that can be
obtained. These critical levels are summed up in
\cref{tab:values_param}; we see that the spectral gap
$1/4 - \epsilon$ requires expansions of arbitrary precision.

\begin{table}[h]
  \begin{tabular}{|c|c|c|c|}
    \hline
    Order in expansion & Length scale $L$ & Parameter $\alpha$ & Hoped spectral gap \\
    \hline
    Leading (error $1/g$) & $4 \log g$ & $1/4$ & $\lambda_1 \geq \frac{3}{16} - \epsilon$ \\
    \hline
    Second (error $1/g^2$) & $6 \log g$ & $1/6$ &$\lambda_1 \geq \frac{2}{9} - \epsilon$ \\
    \hline
    \ldots & \ldots & \ldots & \ldots \\
    \hline
    Error $1/g^{\ord+1}$ & $2(\ord+2) \log g$ & $1/(2(\ord+2))$ & $\lambda_1 \geq \frac{1}{4} - \frac{1}{4(\ord+2)^2} -
                                                                          \epsilon$ \\
    \hline 
  \end{tabular}
  \vspace{1mm}
  \caption{The spectral gap one can hope to obtain using the trace method,
    depending on the order of the asymptotic expansion in powers of $1/g$ at
    which we compute $\av{F}$.}
\label{tab:values_param}
\end{table}

This game of parameters explains how, with additional work, we can prove a spectral gap
$2/9 - \epsilon$ using the results in this article in \cite{anantharaman_29}: that threshold
corresponds to fully understanding the second-order term, as is done here. In the companion
  paper, we understand the structure of terms of all orders, leading to the optimal spectral gap
  $1/4 - \epsilon$.

\subsubsection{The issue of the trivial eigenvalue}
\label{sec:trivial_eig}

Unfortunately, the contribution of the trivial eigenvalue $\lambda_0 = 0$, for
which $r_0 = i / 2$, will always be the dominant term in the Selberg trace
formula. Indeed, $\hat{h}_L(i/2)$ grows almost like $e^{L/2}$ by
\cref{eq:fourier_h_L}. This is much bigger than the size $e^{\alpha L}$ we need
to bound it with in order to prove that $\lambda_1 \geq 1/4 - \alpha^2 - \epsilon$ with high
probability.  Actually, this rate of growth $e^{L/2}$
is exactly what we obtain by using Huber's counting result \cite{huber1974} on
the number of closed geodesics $\leq L$.

As a consequence, the method sketched in \cref{sec:spectr-gap-expon} will
necessarily fail, if one does not find a mechanism to deal with the contribution
of the trivial eigenvalue $\lambda_0 = 0$. The fact that the spectral gap only
appears as a sub-dominant contribution in the trace method, which is hidden by a
much bigger leading order, is always a challenge in spectral gap problems, see
\cite{friedman2003,bordenave2020} for instance in the case of graphs.

In \cite{wu2022,lipnowski2021}, when proving that $\lambda_1 \geq 3/16-\epsilon$
typically, both teams rely on quite a miraculous phenomenon. They observe that
the contribution of the trivial eigenvalue, $\hat{h}_L(i/2)$, and the average of
the term corresponding to primitive \emph{simple} geodesics in the Selberg trace
formula are very close at the first order in $1/g$. Indeed, by using our
first-order approximation for simple geodesics, \cref{prop:simple}, and the
value of $f_0^{\mathbf{s}}$ from \cref{rem:first_order_simple}, we obtain that
\begin{align*}
  \left\langle \frac{\ell \, h_L(\ell)}{2 \sinh \div{\ell}} \right\rangle_g^{\mathbf{s}}
  & = \int_{0}^{+ \infty} \frac{\ell \, h_L(\ell)}{2 \sinh \div{\ell}} \,
    \frac{4}{\ell} \sinh^2 \div{\ell} \d \ell
  + \O{\frac{L^{c} e^{\frac L 2}\norminf{h_L}}{g}}\\
  & = \underbrace{2 \int_{0}^{+ \infty} h_L(\ell) \cosh \div{\ell} \d \ell}_{\hat{h}_L(i/2)}
  + \O{\norminf{h_L} \paren*{1+\frac{L^{c} e^{\frac L 2}}{g}}}.
\end{align*}
It is difficult to see how this approach could still function beyond the
first-order estimate, and if it did, it would require tremendous effort and very
accurate computation of all the coefficients appearing in the asymptotic
expansion.

{
We follow a fundamentally different approach to the one used in~\cite{wu2022,lipnowski2021},
which is more robust and ultimately allows us to reach the optimal bound
$\lambda_1 \geq \frac 14 - \epsilon$.}  The idea is to modify our test function to \emph{create a
  cancellation} at the trivial eigenvalue $\lambda_0 = 0$. More precisely, we want to apply the
Selberg trace formula to a function of Fourier transform $(\frac 14 + r^2)^m \hat{h}_L(r)$, which
therefore has a zero of order $m$ at $r_0=i/2$. This is achieved by considering the new test
function $\D^m h_L$, where $\D$ is the differential operator $\frac 14 - \partial^2$.  We prove the
following reformulation of the spectral gap problem.

\begin{lem}
  \label{lem:selberg_reformulated}
  Let $h$ be a function satisfying the hypotheses of \cref{nota:h_L}, and let us
  fix real numbers $\alpha \in (0, 1/2)$ and $A \geq 1$. For any $\delta > 0$,
  any $0 < \epsilon < \frac 14 - \alpha^2$, any integer $m \geq 1$, there exists a
  constant $C = C(h,\alpha, A, \epsilon, \delta, m)$ such that, for any large
  enough integer $g$ and for the length-scale $L = A \log g$,
  \begin{equation*}
    \Pwp{\delta \leq \lambda_1 \leq \frac 1 4 - \alpha^2 - \epsilon} 
    \leq \frac{C}{g^{(\alpha + \epsilon) A}}
    \avb{\ell \, e^{- \frac \ell 2} \, \mathcal{D}^m h_L(\ell) }
    + C \, \frac{g (\log g)^2}{g^{(\alpha + \epsilon) A}} \cdot
  \end{equation*}
\end{lem}

\begin{rem}
  The parameter $\delta$ is there because, if $\lambda_1$ is very small, then
  $\lambda_1^m \hat{h}_L(r_1)$ will be small. However, this shall not matter,
  because we already know that $\Pwp{\lambda_1 \leq \delta}$ goes to $0$ as
  $g \rightarrow + \infty$ provided $\delta$ is small enough
  \cite{mirzakhani2013,wu2022,lipnowski2021}.
\end{rem}

The mechanism at play here is that, by cancelling the leading order $\hat{h}_L(i/2)$ on the spectral
side of the Selberg trace formula, we create cancellations on the geometric side, so that the
average $\av{\ell e^{- \frac \ell 2} \, \mathcal{D}^mh_L(\ell)}$ is much smaller than it would be
without the application of the differential operator $\mathcal{D}^m$. We shall provide methods to
exhibit such cancellations in \cref{sec:friedm-raman-funct2}, which will further demonstrate the
importance of Friedman--Ramanujan functions for the study of the spectral problem.

The proof is almost the same as the one sketched in \cref{sec:spectr-gap-expon},
with a few small modifications. We provide the details here, because
multiplication by the operator $\mathcal{D}^m$ makes us loose the fact that all
terms in the Selberg trace formula are non-negative, which means we need to
proceed with extra caution.

\begin{proof}
  By \cref{lem:growth_h_r1} applied to the function $h_L$, if
  $\delta \leq \lambda_1(X) \leq 1/4 - \alpha^2 - \epsilon$, then
  $ \lambda_1(X)^m \hat{h}_L(r_1(X)) \geq \delta^m C_{\alpha, \epsilon} \,
  e^{(\alpha + \epsilon)L}$ for a constant $C_{\alpha, \epsilon} > 0$ (depending
  on $h$).  As a consequence, by Markov's inequality, using the non-negativity
  of $\lambda_1^m \hat{h}_L(r_1)$,
  \begin{equation*}
    \Pwp{\delta \leq \lambda_1 \leq \frac 1 4 - \alpha^2 - \epsilon}
      \leq \frac{1}{C_{\alpha,\epsilon} \delta^m g^{(\alpha+\epsilon)A}} \,
      \Ewpo \brac*{\lambda_1^m \hat{h}_L(r_1)}.
  \end{equation*}
  We then apply the Selberg trace formula to the function $\mathcal{D}^m h_L$,
  and more precisely the simplified version we have proven in
  \cref{lem:bound_r1}. Note that the function $\mathcal{D}^m h_L$ satisfies the
  hypotheses of \cref{lem:bound_r1}, because it is even, its support is included
  in the support of $h_L$, which is $[-L,L]$, and its Fourier transform is
  non-negative on $\R \cup i [- \frac 12, \frac 12]$. Then,
  \begin{equation*}
    \Pwp{\delta \leq \lambda_1 \leq \frac 1 4 - \alpha^2 - \epsilon}
    \leq \frac{C_{\alpha,\epsilon,\delta}}{g^{(\alpha+\epsilon)A}} \,
    \paren*{\av{\ell e^{- \frac \ell 2} \, \mathcal{D}^mh_L(\ell)}
      + C_{\mathcal{D}^m h_L} L^2 g}
  \end{equation*}
  and there is a universal constant $c$ such that
  \begin{equation*}
    C_{\mathcal{D}^m h_L}
    \leq c \norminf{\mathcal{D}^m h_L}
    + \norminf*{r \, \left(\frac 14 + r^2 \right)^m \hat{h}_L(r)}.
  \end{equation*}
  This quantity is bounded by a constant depending only on $m$ and
  $h$, because, for $L \geq 1$, the derivatives of
  $h_L(\ell) := h(\ell/L)$ are controlled by the derivatives of $h$,
  and because $\hat{h}_L(r) = L \hat{h}(rL)$, so the second norm is
  bounded by $\norminf{r \left(1/4+r^2\right)^m \hat{h}}$.
\end{proof}

\subsubsection{Friedman--Ramanujan functions and cancellations}
\label{sec:friedm-raman-funct2}

The reason for the introduction of Ramanujan functions in Friedman's work
\cite{friedman2003} is that they are functions which exhibit some
\emph{cancellations} when computing \emph{averages of trace formulas}. We shall
extend this observation to hyperbolic surfaces: we show that Friedman--Ramanujan
functions are functions which create non-trivial on-average cancellations in the
Selberg trace formula.

{
\begin{prp}
  \label{prp:int_FR}
  Let $f \in \cF^{\rK,\rN}_w$. Then, for any integer $m \geq \rK$, any $L \geq 1$,
  \begin{equation}
    \label{eq:int_FR}
    \abso*{\int_0^{L}  f(\ell) \,e^{- \frac \ell 2} \, \D^m h_L(\ell) \d \ell}
    \leq C_{\rN,m}  \norm{f}_{\cF^{\rK,\rN}}^w L^{\rN}.
  \end{equation}
\end{prp}}
 
\begin{rem}
  The constant $C_{\rN,m}$ depends on our fixed test function $h$. More precisely, it
  can be bounded by $C_{\rN,m}' \max_{0 \leq i \leq 2m} \norminf{h^{(i)}}$ for a
  constant $C_{\rN,m}'$ depending only on the integers $\rN$ and $m$.
\end{rem}

In other words, the integral in \cref{prp:int_FR} has at-most polynomial growth in $L$, as opposed
to the exponential growth one could expect, due to the fact that $f(\ell) \, e^{- \frac \ell 2}$ is
of size at most $\ell^{\rK-1} \, e^{\ell/2}$ for large $\ell$. The definition of Friedman--Ramanujan
function is made so that their principal term is always cancelled in integrals of the form
\eqref{eq:int_FR}. The reason for these cancellations is that functions of the form
$p(\ell) \, e^{\ell/2}$ with $p$ a polynomial function of degree $<m$ lie in the kernel of the
operator $\D^m$.

\begin{proof}[Proof of \cref{prp:int_FR}]
{  We write $f(\ell) = p(\ell) \, e^{\ell} + r(\ell)$ for a polynomial function
  $p$ of degree $< \rK$ and a remainder $r$ satisfying
  $\int_0^{L} \abso{r(\ell)} \d \ell
    \leq \norm{f}_{\cF^{\rK,\rN}}^w (L+1)^{\rN-1} e^{\frac L 2}$.

  Let us first estimate the integral of the remainder term. We write 
  \begin{align*}
    \abso*{\int_0^{L} r(\ell) \, e^{-\frac \ell 2} \, \D^mh_L(\ell)  \d \ell}& \leq
    \sum_{n=0}^{\lceil L \rceil -1}\int_{n}^{n+1} |r(\ell) \, e^{-\frac \ell 2} \, \D^mh_L(\ell) |\ \d \ell\\
     & \leq \norminf{\D^mh_L}  
   \sum_{n=0}^{\lceil L \rceil -1} e^{-n/2} \int_{n}^{n+1} |r(\ell)|  \ \d \ell  \\
&\leq   
 \norminf{\D^mh_L}     \norm{f}_{\cF^{\rK,\rN}}^w
   \sum_{n=0}^{\lceil L \rceil -1} e^{-n/2}    (n+2)^{\rN-1} e^{(n+1)/2}\\
  &\leq   
   C_{\rN,m}  \norm{f}_{\cF^{\rK,\rN}}^w L^{\rN} .
  \end{align*}
  We note that the derivatives of $h_L(\ell) = h(\ell/L)$ are bounded by that of $h$ for $L \geq 1$
  so that $\norminf{\D^mh_L} = \O[h,m]{1}$.   }

  We are therefore left with the integral
$ \int_0^L  p(\ell) \, e^{\frac \ell 2} \, \D^mh_L(\ell) \d \ell$.
  We can estimate it using several integration by parts. Indeed, for any smooth
  functions $H_1$, $H_2$, if $\ell \mapsto H_2(\ell)$ is identically equal to
  zero for $\ell \geq L$, then
  \begin{equation}
    \label{eq:ibp_H}
    \int_0^L H_1(\ell) \, \D H_2(\ell) \d \ell = \int_0^L \D H_1(\ell) \, H_2(\ell) \d \ell
    - H_1'(0) H_2(0) + H_1(0) H_2'(0)
  \end{equation}
  because $H_2(L) = H_2'(L)=0$. We apply this $m$ times to our integral, using
  the fact that $h_L$ and its derivatives vanish above $L$, and obtain that
  \begin{equation}
    \label{eq:FR_implies_small_IBP}
     \int_0^L  (p(\ell) \, e^{\frac \ell 2}) \, \D^mh_L(\ell) \d \ell = 
    \int_0^L \D^m [p(\ell) e^{\frac \ell 2}](\ell) \, h_L(\ell) \d \ell
    + \O[h,m]{\norm{f}_{\cF^{\rK,\rN}}^w}
  \end{equation}
  because the boundary terms appearing in \eqref{eq:ibp_H} are
  linear combinations of products of the form
  $\partial^i h_L(0) \, \frac{\partial^{j}}{\partial \ell^j}
  [p(\ell) e^{\ell/2}](0)$ for $i, j \leq 2m$. The integral in the
  right hand side of \cref{eq:FR_implies_small_IBP} is equal to
  zero, because $\D^m[p(\ell) \, e^{\frac \ell 2}] \equiv 0$ as soon
  as $m > \deg p$.
\end{proof}

In order to conclude this section with a strong motivation for the study of
Friedman--Ramanujan functions in the context of the spectral gap question, we
prove the following consequence of \cref{prp:int_FR}. This last statement uses
some notations and results obtained in Sections \ref{sec:topol-types-curv} and
\ref{sec:average-over-local}, but will not be used until Section
\ref{sec:second-order-term}.

\begin{prp}
  \label{cor:FR_implies_small}
  Let $\type$ be a local topological type.  If \ref{chal:FR_type} is true, then for any integer
  $\ord \geq 0$, there exists constants $c_\ord^\type, m_\ord^\type \geq 0$ such that for any large enough $g$,
  any $m \geq m_\ord^\type$, any $\eta >0$ and $L \geq 1$,
  \begin{equation*}
    \avb[\type]{\ell \, e^{- \frac \ell 2} \, \D^m h_L(\ell)}
    = \O[m,\type,\ord,\eta]{L^{c_\ord^\type} + \frac{e^{\frac L 2 + \eta L}}{g^{\ord+1}}}.
  \end{equation*}
\end{prp}

This result is our motivation to prove \ref{chal:FR_type} in all generality, {which we do in
  the second article}.  Indeed, recall that we presented in
\cref{lem:selberg_reformulated} a reformulation of the trace method, where we reduced the spectral
gap problem to proving that the average $\av{\ell \, e^{- \frac \ell 2} \, \D^m h_L(\ell)}$ is
negligible compared to $g^{(\alpha+\epsilon)A}$, where $L = A \log g$, $A = 2(\ord+2) = 1/\alpha$,
as explained in \cref{tab:values_param}. Here, \cref{cor:FR_implies_small} tells us that, \emph{if
  \ref{chal:FR_type} is true for a local type $\type$}, then this objective is attained for the
contribution of geodesics of local type $\type$ to the overall geometric average
$\av{\ell \, e^{- \frac \ell 2} \, \D^m h_L(\ell)}$.

For now, we have proved that \ref{chal:FR_type} holds for the type ``simple'' in \cref{prop:simple},
and will extend it to any loop filling a pair of pants or once-holed torus in
Sections~\ref{sec:generalised_convolutions} and~\ref{sec:other-geodesics}. In particular,
\cref{cor:FR_implies_small} is true in these cases.

\begin{proof}[Proof of \cref{cor:FR_implies_small}]
  Let $\ord \geq 0$. By \cref{thm:exist_asympt_type_intro},
  \begin{equation*}
    \avb[\type]{\ell \, e^{- \frac \ell 2} \, \D^m h_L(\ell)}
    = \int_{0}^L  F^\type_{g,\ord}(\ell) \, e^{- \frac \ell 2} \, \D^m h_L(\ell) \d \ell
    + \O[\type,\ord,\eta]{\frac{\norminf{e^{\frac \ell 2 + \eta \ell} \D^m h_L}}{g^{\ord+1}}},
  \end{equation*}
  where $F_{g,\ord}^\type(\ell) := \sum_{k=0}^{\ord} \ell f_k^\type(\ell)/g^k$. {By hypothesis,
  $F_{g,\ord}^\type$ is a Friedman--Ramanujan function in the weak sense. More precisely, for every
  integer $k$, \ref{chal:FR_type} tells us that there exists $m_k^\type$, $c_k^\type$ such that
  $\ell f_{k}^\type(\ell) \in \cF_w^{m_k^\type,c_k^\type}$. We assume w.l.o.g. that these indices
  are increasing functions of the order $k$. Then, by linearity,
  $F_{g,\ord}^\type \in \cF_w^{m_\ord^\type,c_\ord^\type}$ and furthermore
  \begin{equation*}
    \norm{F_{g,\ord}^\type}^w_{\cF^{m_{\ord}^\type,c_\ord^\type}}
    \leq \sum_{k=0}^{\ord} \norm{\ell f_k^{\type}(\ell)}^w_{\cF^{m_{\ord}^\type,c_\ord^\type}} =
    \O[\type,\ord]{1}
  \end{equation*}
  where here we use the fact that the embedding $\cF_w^{\rK,\rN} \rightarrow \cF_w^{\rK',\rN'}$ is
  continuous as soon as $\rK \leq \rK'$ and $\rN \leq \rN'$.}  We apply \cref{prp:int_FR} to
conclude.
\end{proof}


\section{Local topological types of loops}
\label{sec:topol-types-curv}

One of the aims of this article is to generalise methods to compute averages for
\emph{simple} geodesics to more elaborate topologies. In order to do so, we need
to introduce a few notations and concepts related to non-simple closed geodesics
on a surface.

\subsection{Surface filled by a loop}

A challenge faced when studying general loops is that the machinery developed
by Mirzakhani in \cite{mirzakhani2007,mirzakhani2013} only applies to
\emph{multi-curves}, i.e. families of \emph{simple} disjoint loops. A way around this
difficulty, already used in \cite{mirzakhani2019,wu2022,lipnowski2021}, is to
associate to any loop a surface that it fills, using the following procedure.

\begin{defa}
  \label{def:surf_filled}
  Let $X$ be a compact hyperbolic surface, and $\gamma$ be a loop on $X$. We
  assume that~$\gamma$ is in minimal position, i.e. that it minimises the number
  of self-intersections in its homotopy class.  We define the \emph{surface
    $S(\gamma)$ filled by $\gamma$} the following way.
  \begin{enumerate}
  \item We take a regular neighbourhood of $\gamma$ in $X$,
    $\mathcal{N}_\epsilon(\gamma) := \{ x \in X \, : \, \dist(x,\gamma) < \epsilon \}$ for
    $\epsilon>0$ small enough so that $\mathcal{N}_\epsilon(\gamma)$ retracts to $\gamma$.
  \item The bordered surface $X \setminus \mathcal{N}_\epsilon(\gamma)$ has
    $q_0 \geq 1$ connected components $C_1 \sqcup \ldots \sqcup C_{q_0}$. We
    take
    \begin{equation*}
      S(\gamma) := \mathcal{N}_\epsilon(\gamma) \cup \bigcup_{i: C_i \text{ is a disk}} C_i,
    \end{equation*}
    i.e. we add every disk to $\mathcal{N}_\epsilon(\gamma)$, to form $S(\gamma)$.
  \end{enumerate}
\end{defa}

The surface $S(\gamma)$ is a subsurface of $X$, possibly with a boundary, filled
by $\gamma$. It does not depend on the choice of $\epsilon >0$, in the sense
that the filled surfaces obtained using two small values of $\epsilon$ are
isotopic. Similarly, the metric on $X$ is solely used to define the regular
neighbourhood, and replacing it by another metric yields the same filled surface,
up to isotopy. The notion of filled surface only depends on the topology of the
loop $\gamma$ within the topological surface $X$.

The boundary of $S(\gamma)$ is a family of simple loops, which we orient so that $S(\gamma)$ lies on
the left side of each boundary components.  The motivation for introducing $S(\gamma)$ is that the
boundary the surface filled by $\gamma$ is (almost) a multi-curve, and can hence be dealt with using
Mirzakhani's tools (almost because there are two boundary loops $\gamma_1$ and $\gamma_2$ such that
$\gamma_1$ is homotopic to $\gamma_2^{-1}$ whenever there is a cylinder in $X \setminus S(\gamma)$).

\begin{figure}[h]
  \centering
  \begin{subfigure}[b]{0.33\textwidth}
    \centering
    \includegraphics[scale=0.45]{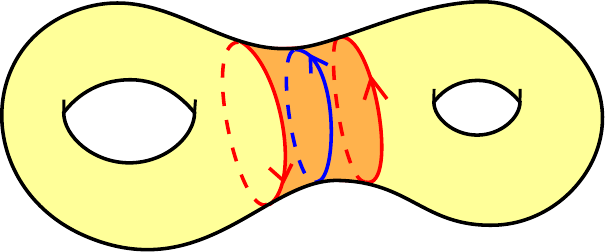}
    \caption{A simple loop.}
    \label{fig:filled_simple}
  \end{subfigure}%
  \begin{subfigure}[b]{0.67\textwidth}
    \centering
    \includegraphics[scale=0.45]{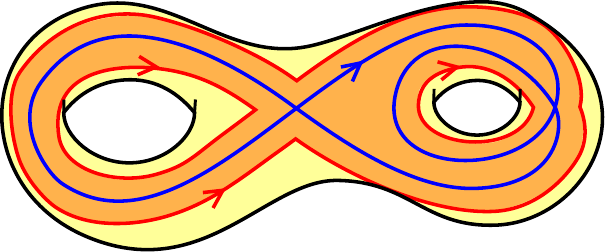}
    \includegraphics[scale=0.45]{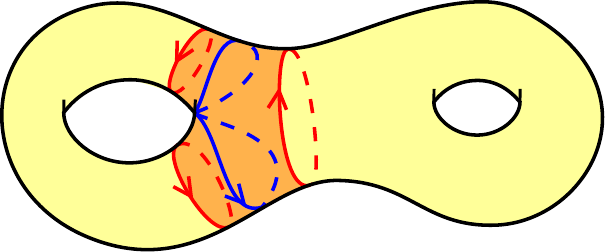}
    \caption{Two non-simple loops filling a pair of pants.}
    \label{fig:filled_non_simple}
  \end{subfigure}%
  \caption{Examples of filled surfaces.}
  \label{fig:filled_ex}
\end{figure}

\begin{exa}
  The surface filled by a simple non-contractible loop is a cylinder. If $\gamma$ is a loop with exactly
  one self-intersection, then $S(\gamma)$ is a pair of pants.
\end{exa}

Three examples of loops and their filled surfaces are represented in \cref{fig:filled_ex}. Note
that, in the first example of \cref{fig:filled_non_simple}, we added a disk to the regular
neighbourhood to form the filled surface. In the last picture, we replaced the filled surface by
another isotopic surface, for the sake of readability.

We prove the following, using a classic result of Graaf--Schrijver~\cite{graaf1997}.

\begin{lem}
  \label{lem:equivalent}
  Let $X$ be a compact hyperbolic surface.  If $\gamma$ and $\gamma'$ are two  loops in
  minimal position in $X$ in the same homotopy class, then there exists an isotopy of $X$
  sending $S(\gamma)$ on $S(\gamma')$ and $\gamma$ on a loop  homotopic to $\gamma'$ in
  $S(\gamma')$.
\end{lem}

As a consequence, the surface $S(\gamma)$ filled by $\gamma$ is well-defined (up to isotopy) for a
homotopy class $\gamma \in \mathcal{G}(X)$. We shall see in the proof that this is true thanks
to the fact that we added all disks to the regular neighbourhood of $\gamma$.

\begin{proof}
  First, we observe that if $\gamma$ and $\gamma'$ are \emph{isotopic}, that is
  to say if there exists an isotopy $\phi_t : X \rightarrow X$,
  $0 \leq t \leq 1$, such that $\phi_1 \circ \gamma = \gamma'$, then the claim
  is trivially satisfied. Indeed, in this case, for small enough $\epsilon >0$,
  we can modify the isotopy $(\phi_t)_t$ to obtain a new isotopy $(\psi_t)_t$
  which coincides with $(\phi_t)_t$ on all points of $\gamma$ and sends the
  regular neighbourhood $\mathcal{N}_\epsilon(\gamma)$ onto the regular
  neighbourhood $\mathcal{N}_\epsilon(\gamma')$ of $\gamma'$. Then, the isotopy
  $\psi_1$ is an homeomorphism from each connected components of
  $X \setminus \mathcal{N}_{\epsilon}(\gamma)$ to each component of
  $X \setminus \mathcal{N}_{\epsilon}(\gamma')$, and in particular sends
  contractible components to contractible components. Hence, the isotopy
  $(\psi_t)_t$ sends $S(\gamma)$ on $S(\gamma')$ and $\gamma$ on $\gamma'$, and
  our claim is satisfied.
  
  More generally, by \cite{graaf1997}, because $\gamma$ and $\gamma'$ are  homotopic and both
  in minimal position, there is a finite sequence of third Reidemeister moves that send $\gamma$ to
  a loop $\tilde \gamma$ isotopic to $\gamma'$. As a consequence, we simply need to prove our claim
  for two loops $\gamma$, $\gamma'$ differing by a third Reidemeister move, as represented in
  \cref{fig:reid_move_curves}.
  \begin{figure}[h]
    \centering
    \includegraphics[scale=0.5]{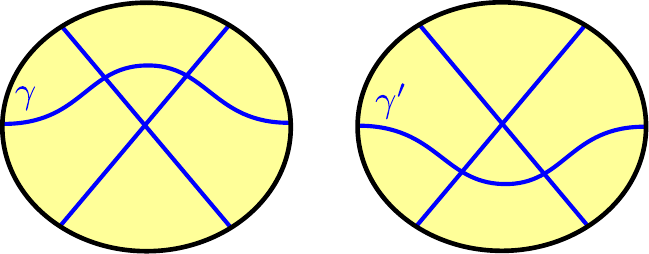} \hspace{1cm}
    \includegraphics[scale=0.5]{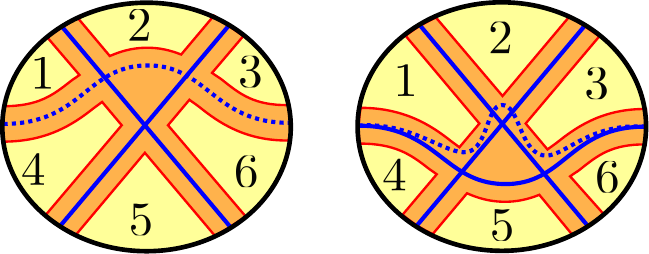}
    \caption{On the left hand side, a portion of two loops $\gamma$, $\gamma'$ differing only by a
      third Reidemeister move. On the right hand side, in orange, their respective regular
      neighbourhood, to which we added the central disk. We labelled the (possibly non-distinct)
      connected components of the complement of this region to highlight their correspondence.}
    \label{fig:reid_move_curves}
  \end{figure}

  We observe on \cref{fig:reid_move_curves} that, thanks to the
  addition of the central contractible component in the construction
  of $S(\gamma)$ and $S(\gamma')$, there exists an isotopy
  $(\phi_t)_{0 \leq t \leq 1}$ (identically equal to the identity
  outside the neighbourhood where the Reidemeister move occurs)
  sending $S(\gamma)$ to $S(\gamma')$. The image of $\gamma$ by such
  an isotopy is represented by the dotted line in the last part of
  \cref{fig:reid_move_curves}, and it is clear that
  $\phi_1 \circ \gamma$ and $\gamma'$ are  homotopic within
  $S(\gamma')$.
\end{proof}

The following observation on the boundary length of $S(\gamma)$ will be useful.

\begin{lem}
  \label{lem:length_fill} Let $X$ be a compact hyperbolic surface, and $\gamma \in
  \mathcal{G}(X)$. Let $S_X(\gamma)$ denote the surface isotopic to $S(\gamma)$ in $X$ with
  geodesic boundary. Then, $\ell_X(\partial S_X(\gamma)) \leq 2 \ell_X(\gamma)$.
\end{lem}
\begin{proof}
  For any $\eta > 0$, we can pick the $\epsilon$ for defining the regular
  neighbourhood of $\gamma$ such that the length of its boundary is
  $\leq 2 \ell_X(\gamma) + \eta$. Then, the length of the boundary only
  diminishes when adding disks to the complement, and when replacing every
  component of the boundary by a geodesic representative, so
  $\ell_X(\partial S_X(\gamma)) \leq 2 \ell_X(\gamma) + \eta$. We obtain the
  result by letting $\eta \rightarrow 0$.
\end{proof}

\subsection{Definition of local topological type}
\label{sec:def_type}

Let us define a notion of \emph{local (topological) type}. Examples of local types are presented in
\cref{fig:def_type}, the type ``simple'' being the leftmost one.

\begin{figure}[h]
    \centering
    \includegraphics[scale=0.42]{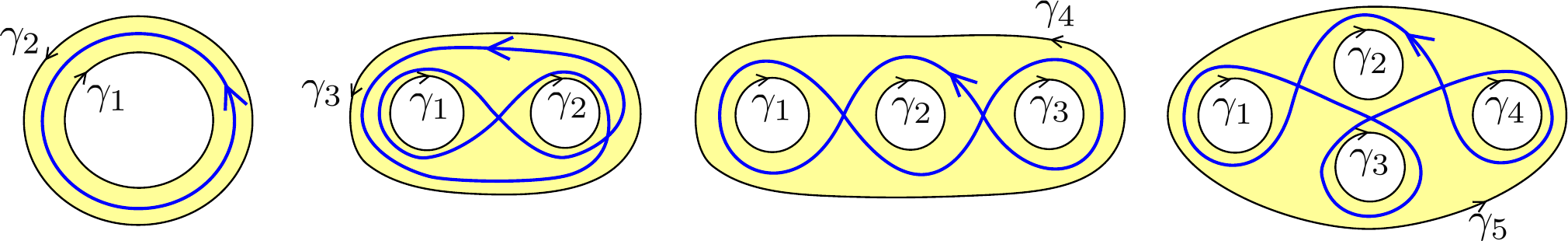}
  \caption{Examples of local topological types.}
  \label{fig:def_type}
\end{figure}

\begin{nota}
  Let $(g_\Sf, n_\Sf)$ be two non-negative integers. We assume that the absolute Euler
  characteristic $\chi(\Sf) := 2g_\Sf-2+n_\Sf$ is positive or that $(g_\Sf,n_\Sf)=(0,2)$. We shall
  associate to the pair $(g_\Sf,n_\Sf)$ a \emph{fixed} smooth oriented surface $\Sf$ of signature
  $(g_\Sf,n_\Sf)$. We further fix a numbering of the $n_\Sf$ boundary components of $\Sf$, and
  denote for each $1 \leq i \leq n_\Sf$ as $\gamma_i$ the $i$-th boundary loop of $\Sf$, oriented so
  that $\Sf$ lies on the left-hand-side of $\gamma_i$.  The data of the pair of integers
  $(g_\Sf, n_\Sf)$, or equivalently of the surface $\Sf$, is called a \emph{filling type}.
\end{nota}

\begin{defa}
  \label{def:local_type}
  A \emph{local loop} is a pair $(\Sf,\curve)$, where $\Sf$ is a filling type and $\curve$ is a
  primitive loop filling $\Sf$. Two local loops $(\Sf,\curve)$ and $(\Sf',\curve')$ are said to be
  \emph{locally equivalent} if $\Sf=\Sf'$ (i.e.  $g_\Sf = g_{\Sf'}$ and $n_\Sf=n_{\Sf'}$), and there
  exists a positive homeomorphism $\psi : \Sf \rightarrow \Sf$, possibly permuting the boundary
  components of $\Sf$, such that $\psi \circ \curve$ is homotopic to~$\curve'$. This defines an
  equivalence relation $\eq$ on local loops. Equivalence classes for this relation are denoted as
  $\type=\eqc{\Sf, \curve}$ and called \emph{local (topological) types} of loops.
\end{defa}

\begin{exa}
  \begin{figure}[h]
    \centering
    \includegraphics[scale=0.5]{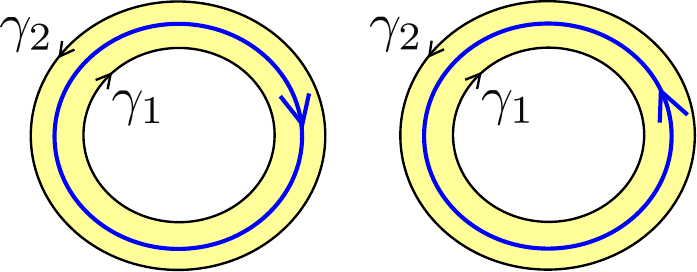}
    \caption{Two representatives of the local type ``simple''.}
    \label{fig:def_type_0-2}
  \end{figure}
  There is exactly one local topological type filling a cylinder (i.e. of filling type $(0,2)$),
  which we shall refer to as the local type ``\emph{simple}''. Indeed, there are exactly two homotopy
  classes of primitive loops filling a cylinder, represented in \cref{fig:def_type_0-2}. Taking a
  positive homeomorphism permuting the two boundary components of the cylinder allows to observe that
  these two local loops are equivalent.
\end{exa}

\begin{exa}
  \begin{figure}[h]
    \centering
    \includegraphics[scale=0.5]{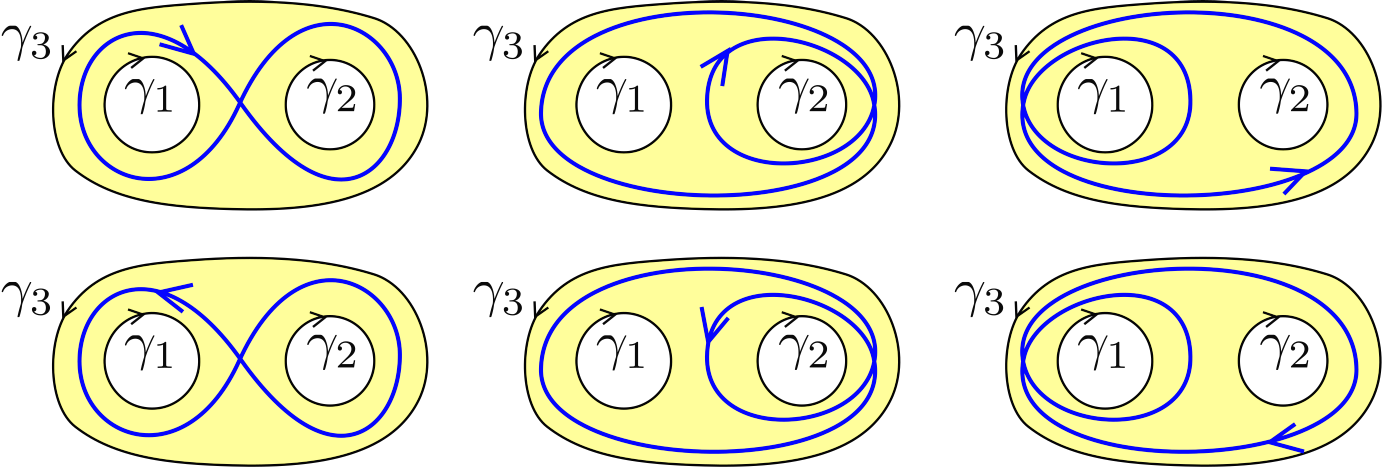}
    \caption{Six representatives of the local type ``eight''.}
    \label{fig:def_type_0-3-eight}
  \end{figure}
  We call ``\emph{figure-eight}'' the local type of which six representatives are depicted in
  \cref{fig:def_type_0-3-eight}. The figure-eight is of filling type $(0,3)$, i.e. it fills a pair
  of pants.  As in the previous example, these different representatives can be shown to be equivalent
  by applying a positive homeomorphism of the pair of pants permuting its boundary components. More
  precisely, the six representatives correspond to the six permutations of the set $\{1,2,3\}$; in
  reading order, $\id$, $(123)$, $(132)$, $(12)$, $(13)$, $(23)$.
\end{exa}

{
\begin{nota}
  We define the absolute Euler characteristic of a local type $\chi(\type) := \chi(\Sf)$ as the
  absolute Euler characteristic of its filling type.
\end{nota}

In particular, the local type ``simple'' is the only local type of Euler characteristic $0$. }

\subsection{Local topological type of a loop on a surface of genus $g$}
\label{sec:surf-fill-clos}

Let us now define a notion of local topological type for loops on a compact
hyperbolic surface of genus $g$, for $g \geq 2$. We shall do this for loops on
the base surface $S_g$, which we endow with a fixed hyperbolic metric for the
purpose of defining regular neighbourhoods.

\begin{defa}
  Let $\eqc{\Sf, \curve}$ be a local topological type.  A loop $\gamma$ on the base surface~$S_g$ is
  said to \emph{belong to the local topological type}~$\eqc{\Sf, \curve}$ if there exists a positive
  homeomorphism $\phi : S(\gamma) \rightarrow \Sf$ such that the loops $\phi \circ \gamma$ and
  $\curve$ are homotopic in~$\Sf$. In that case, we write $\gamma \sim \eqc{\Sf, \curve}$.  We
  say that two loops $\gamma$, $\gamma'$ on $S_g$ are \emph{locally equivalent}, and write
  $\gamma \eq \gamma'$, if $\gamma$ and $\gamma'$ belong to the same local topological type.
\end{defa}

It is clear that the definition does not depend on the choice of the representative $(\Sf,\curve)$ in
the local topological type. The fact that it does not depend on the choice of representative in the
homotopy classes $\gamma$ and $\gamma'$ either is a consequence of \cref{lem:equivalent}.

\begin{exa}
  \label{ex:simple_loc_eq}
  The loops on the base surface $S_g$ belonging to the local type ``simple'' are exactly all 
  simple loops on $S_g$.
  Indeed, take $\gamma \in \mathcal{G}(S_g)$.
  \begin{itemize}
  \item If $\gamma$ is simple, then its regular neighbourhood is a
    cylinder. Because $\gamma$ is not contractible, $S(\gamma)$
    is exactly this cylinder, and then $\gamma$ is its core. It
    directly follows that $\gamma$ belong to the local type
    ``simple''.
  \item If $\gamma$ belongs to the local type ``simple'', then $S(\gamma)$ is a topological
    cylinder, and $\gamma$ is homotopic to its core. As a consequence, the homotopy class
    $\gamma$ admits a simple representative, which means that $\gamma$ is simple.
  \end{itemize}
\end{exa}
 
\subsection{Comparison with mapping-class-group equivalence}
\label{sec:comp-with-mapp}

\Cref{ex:simple_loc_eq} shows that the notion of local equivalence and $\mcg_g$-equivalence do not
coincide. Indeed, there are several distinct orbits of simple loops for $\mcg_g$: non-separating
loops, and loops separating two surfaces on their left and right side, of respective signatures
$(i,1)$ and $(g-i,1)$ for a $1 \leq i \leq g-1$. This is a general fact: equivalence classes for
$\mcg_g$ are included in equivalence classes for $\eq$, as shown in the following lemma.

\begin{lem}
  \label{lem:local_vs_mcg}
  Let $\gamma$, $\gamma' \in \mathcal{G}(S_g)$. If $\gamma \eqmcg \gamma'$, then
  $\gamma \eq \gamma'$.
\end{lem}
\begin{proof}
  We assume that there exists a homeomorphism $\phi : S_g \rightarrow S_g$ sending $\gamma$ on
  $\gamma'$. Then, by definition of the filled surface associated to a loop, the image of
  $S(\gamma)$ by $\phi$ is isotopic to $S(\gamma')$. As a consequence, by composition, there
  exists a homeomorphism $\psi : S_g \rightarrow S_g$ sending $S(\gamma)$ on $S(\gamma')$ and
  $\gamma$ on $\gamma'$, which implies our claim.
\end{proof}

The choice of the name ``local equivalence'' comes from the fact that this notion only en-captures
the topology of the filled surface $S(\gamma)$ and of the loop $\gamma$ within it, but does no
say anything about the topology of the complement $S_g \setminus S(\gamma)$.  To the contrary, if
two loops $\gamma, \gamma'$ are in the same $\mcg_g$-orbit, then the topologies of
$S_g \setminus S(\gamma)$ and $S_g \setminus S(\gamma')$ need to be ``the same'', in a way which
will be made more precise in the following section.

\subsection{Realizations of a filling type}
\label{sec:real-fill-type}

In order to describe more precisely the way that a local equivalence class $\eqc{\Sf,\curve}$ is
partitioned in $\mcg_g$-orbits, we introduce the following notion.

\begin{defa}
  \label{def:realizations}
  Let $\Sf$ be a filling type, and $g \geq 2$. We call \emph{realization of $\Sf$
  in $S_g$} any pair $\mathfrak{R} = (\vec{I}, \vec{g})$, where:
  \begin{itemize}
  \item $\vec{I}$ is a partition of $\partial \Sf = \{1, \ldots, n_\Sf\}$ into
    $\mathfrak{q} \geq 1$ non-empty sets $I_1, \ldots, I_\mathfrak{q}$, numbered such that
    $j \mapsto \min I_j$ is an increasing function;
  \item $\vec{g} = (g_1, \ldots, g_\mathfrak{q})$ is a vector of non-negative integers;
  \item for any $1 \leq j \leq \mathfrak{q}$, if  $n_j := \# I_j$, then the absolute Euler
    characteristic $\chi_j := 2g_j-2+n_j$ is positive or $(g_j,n_j)=(0,2)$;
  \item for $\chi(S_g) = 2g-2$ and $\chi(\Sf) := 2g_\Sf-2+n_\Sf$, we have
   \begin{equation}
     \label{eq:add_euler_realization}
     \chi(\Sf) + \sum_{j=1}^\mathfrak{q} \chi_j = \chi(S_g).
   \end{equation}   
 \end{itemize}
 The set of realizations of $\Sf$ in $S_g$ is denoted as $R_g(\Sf)$.
\end{defa}

{Realizations enumerate all possible embeddings of the filling type $\Sf$ (and hence the loop
  $\curve$) into a surface of genus $g$. Indeed, if $\mathfrak{R} =(\vec{I}, \vec{g}) \in R_g(\Sf)$,
  we construct an embedding of $\Sf$ into a compact surface of genus $g$ by gluing, for all
  $1 \leq j \leq \mathfrak{q}$, a surface of signature $(g_j, n_j)$ on the boundary components of
  $\Sf$ belonging in $I_j$.   }

\begin{rem}
  Note that, if $(g_j,n_j) = (0,2)$, then the surface we glue is a cylinder, which corresponds to
  gluing two boundary components of $\Sf$ together. In all other cases, we glue an honest surface of
  negative Euler characteristic to $\partial \Sf$.
\end{rem}

{By \cref{eq:add_euler_realization}, the resulting surface is a compact surface of genus $g$,
  which we identify with our base surface $S_g$. In particular, any loop $\curve$ on $\Sf$ is now
  sent on a loop on the base surface $S_g$, depending on the realization $\mathfrak{R}$.  Different
  realizations might yield mapping-class-group equivalent loops on $S_g$, leading to multiplicities
  in our enumeration, described in the following lemma.

  \begin{lem}
    \label{lem:mult_real}
  Let $(\Sf,\curve)$ be a local loop. Two realizations $\mathfrak{R} = (\vec{I},\vec{g})$ and
  $\mathfrak{R}'=(\vec{I}',\vec{g}')$ of $\Sf$ in $S_g$ for a genus $g \geq 2$ yield the same
  mapping-class-group orbit for the loop $\curve$ in $S_g$ if and only if the following conditions
  are satisfied:
  \begin{itemize}
  \item there exists a positive homeomorphism $\phi : \Sf \rightarrow \Sf$, possibly permuting the
    boundary components of $\Sf$, such that $\phi(\curve)$ and $\curve$ are  homotopic in
    $\Sf$;
  \item for all $1 \leq j \leq \mathfrak{q}$, $\phi$ sends the components of $\partial \Sf$ lying in $I_j$ on the
    components lying in $I_{j'}'$ for a $1 \leq j' \leq \mathfrak{q}$ such that $g_j=g_{j'}'$.
  \end{itemize}
\end{lem}

\begin{rem}
  We notice that, whenever $\phi$ is isotopic to the identity, these conditions imply that
  $\mathfrak{R}=\mathfrak{R}'$ thanks to our numbering convention for $\vec{I}$.  As a consequence,
  several realizations can get associated to the same mapping-class-group orbit of loop, but only in
  the case where the loop $\curve$ has non-trivial symmetries in $\Sf$.
\end{rem}

  The following volume, associated to a realization, will play a key role in our integration formula.
\begin{nota}
  To any realization $\mathfrak{R} \in R_g(\Sf)$ we associate a \emph{volume function}
  \begin{equation*}
    V_{\mathfrak{R}}(x_1, \ldots, x_{n_\Sf})
    = \prod_{j=1}^{\mathfrak{q}} V_{g_j,n_j}(x_i, i \in I_j),
  \end{equation*}
  where we recall that $V_{g',n'}(\cdot)$ is the total volume of the moduli space
  $\cM_{g',n'}(\cdot)$ whenever $2g'-2+n'>0$ (with an additional factor $1/2$ if $(g',n')=(1,1)$)
  and with the convention that $V_{0,2}(x,y) := \frac 1 x \, \delta(x-y)$ where $\delta$ is the
  Dirac delta distribution.
\end{nota}}

\begin{rem}
  We observe that, since $\vec{I}$ is a partition of $\partial \Sf$, we have that
  $\sum_{j=1}^\mathfrak{q} n_j = n_\Sf$, and we can therefore rewrite \cref{eq:add_euler_realization} as
  \begin{equation}
    \label{eq:add_euler_realization_2}
    \sum_{j=1}^\mathfrak{q} g_j + g_\Sf + n_\Sf - \mathfrak{q} = g.
  \end{equation}
\end{rem}

\subsection{Multiplicity of a local type}

We introduce several combinatorial factors associated to a local type, which will allow to remove
the multiplicities described in \cref{lem:mult_real}.

\begin{defa}
  \label{def:def_mult}
  { Let $\type = \eqc{\Sf,\curve}$ be a local topological type.  We call \emph{multiplicity
      of~$\type$} the cardinality $n(\type)$ of the group of positive homeomorphisms (up to isotopy) of
    $\Sf$ stabilizing the homotopy class of $\curve$. We denote as $n_0(\type)$ the cardinality of the
    subgroup of those homeomorphisms which act trivially the boundary components of $\Sf$, and
    $m(\type) = n(\type)/n_0(\type)$ its index.}
\end{defa}

\begin{exa}
  The image of the figure-eight by all permutations of $\{1,2,3\}$ is represented in
  \cref{fig:def_type_0-3-eight}; we see that, in this case, $n(\type) = n_0(\type) =
  1$. {However, in the third example of \cref{fig:def_type}, the multiplicity is $n(\type)=2$,
    whilst $n_0(\type)=1$, because the permutation $(13)$ stabilises the loop $\curve$.}
\end{exa}


\section{Average over a local type}
\label{sec:average-over-local}

The aim of this section is to define the average $\av[\type]{F}$ of a test function over a local type
$\type$, and to provide a method to express and estimate it. In particular, we shall:
\begin{itemize}
\item provide a formula for $\av[\type]{F}$ in terms of Weil--Petersson volumes in
  \cref{thm:express_average};
\item prove it can be written as a density in \cref{prp:existence_density};
\item expand it in powers of $1/g$ in \cref{thm:exist_asympt_type}.
\end{itemize}
\subsection{Definition}
\label{sec:defin-local-types}

Let us define the average $\av[\type]{F}$ of a test function $F$ over a local topological type $\type$.

\begin{defa}
  We call \emph{test function} any measurable function $F : \R_{\geq 0} \rightarrow \C$ that is
  either bounded and compactly supported, or non-negative.
\end{defa}

\begin{rem}
  Actually, the results in this paper hold for a more general class of test
  functions, we only need to assume that they decay sufficiently fast at
  infinity so that all quantities mentioned converge.
\end{rem}

The invariance of local types by action of the mapping-class group allows us to make the following
definition.

\begin{defa}
  Let $\type$ be a local topological type. For any  test function $F$, any $g \geq 2$, we
  define the \emph{$\type$-average} of $F$ over surfaces of genus $g$ to be
  \begin{equation}
    \label{eq:def_av_T}
    \av[\type]{F} := \Ewp{\sum_{\gamma \sim \type} F(\ell_X(\gamma))}.
  \end{equation}
\end{defa}

We notice that this coincides with the definition of $\av[\mathbf{s}]{F}$ for
the local type ``simple''.  Obviously, we have that
\begin{equation*}
  \av{F} = \sum_{\type \text{ local type}} \av[\type]{F}.
\end{equation*}

\begin{rem}
  An interesting benefit from splitting the average $\av{F}$ by local
  topological type rather than $\mcg_g$-orbit is that the set of local types is
  fixed and independent of the genus $g$, whilst the number of $\mcg_g$-orbits of
  loops grows as a function of~$g$.
\end{rem}

\subsection{Integration formula for averages over a local type}
\label{sec:integr-form-aver}

We are now ready to write an integration formula for the average $\av[\type]{F}$, for any local type
$\type$. {The integration will take place on the following space, the space of metrics on the
filled surface $\Sf$.}

\begin{nota}
    Let $\Sf$ be a filling type. We define
  \begin{equation*}
    \mathcal{T}_{g_\Sf, n_\Sf}^{\star} := \{(\x, Y) \, : \, \x \in \R_{> 0}^{n_\Sf}, \,
    Y \in \mathcal{T}_{g_\Sf,n_\Sf}(\x)\}.
  \end{equation*}
\end{nota}

This space is the natural space in which we can define the function $\ell_Y(\curve)$, the length of
the geodesic representative of the loop $\curve$ on the surface $\Sf$ equipped with a metric $Y$. 

\begin{lem}
  \label{lem:teich_analytic}
  The space $\mathcal{T}_{g_\Sf, n_\Sf}^{\star}$ is a real-analytic manifold that can be
  identified with $\R_{>0}^{n_\Sf} \times (\R_{>0} \times \R)^{3g_\Sf-3+n_\Sf}$
  through Fenchel-Nielsen coordinates. The measure
  \begin{equation*}
    \d \Volwp[g_\Sf,n_\Sf](\x,Y) := \d \x \d \Volwp[g_\Sf,n_\Sf,\x](Y)
  \end{equation*}
  is the Lebesgue measure in these coordinates. Furthermore, for any filling
  loop $\curve$ on $\Sf$, the function
  \begin{equation*}
    \begin{cases}
      \mathcal{T}_{g_\Sf, n_\Sf}^{\star}
      & \rightarrow \R_{>0} \\
      (\x, Y) & \mapsto \ell_Y(\curve)
    \end{cases}
  \end{equation*}
  is a real-analytic function, which satisfies
  \begin{equation}
    \label{eq:teich_analytic_bound}
    \forall (\x, Y) \in \mathcal{T}_{g_\Sf,n_\Sf}^\star, \quad
    \ell_Y(\curve) \geq \frac{x_1 + \ldots +x_{n_\Sf}}{2} \cdot
  \end{equation}
\end{lem}

\begin{proof}
  We refer to \cite[Section 6.3]{buser1992} for a description of the
  real-analytic structure of Teichm\"uller spaces. The inequality on the length
  function is exactly \cref{lem:length_fill}.
\end{proof}

The integration formula then reads as below.
{
\begin{thm}
  \label{thm:express_average}
  Let $\type = \eqc{\Sf,\curve}$ be a local topological type. Then, for any $g \geq 3$, any test
  function $F$,
  \begin{equation}
    \label{eq:orbit_decomposition}
    \av[\type]{F}
    = \frac{1}{n(\type)}
    \int_{\cT^*_{g_\Sf,n_\Sf}}
    F(\ell_Y(\curve))
   \,  \phi_g^\Sf(\x)
    \d \Volwp[g_\Sf, n_\Sf](\x,Y)  
  \end{equation}
  where the function $\phi_g^\Sf$ is defined as a sum over realizations by
  \begin{equation}
    \label{eq:def_phi_S}
    \phi_g^\Sf(\x) := 
    \frac{x_1 \ldots x_{n_\Sf}}{V_g} \sum_{\mathfrak{R} \in R_g(\Sf)} V_{\mathfrak{R}}(\x).
  \end{equation}
  
\end{thm}

In other words, the average $\av[\type]{F}$ can be computed as an integral on the space of metrics
on $\Sf$ of the function $F(\ell_Y(\curve))$, multiplied by a density $\phi_g^\Sf(\x)$ which counts
the number of possible geometries for the complement of $\Sf$ in a surface of genus $g$. A striking
aspect of this formula is the fact that it completely disentangles the dependency of $\av[\type]{F}$
in the genus $g$ (which only appears in $\phi_g^\Sf$) and the specific loop $\curve$ filling $\Sf$
(which only appears in $F(\ell_Y(\curve))$).

Let us explicit the formula in two very simple examples.

\begin{exa}
  \label{exa:int_simple}
  The multiplicity of the type ``simple'' is $1$. We have
  \begin{equation}
    \label{eq:int_simple}
    \phi_g^{(0,2)}(x,y)
    = \frac{xy}{V_g} \left(V_{g-1,2}(x,y) + \sum_{i=1}^{g-1} V_{i,1}(x) V_{g-i,1}(y)\right).
  \end{equation}
    We therefore recover the expression for $\av[\mathbf{s}]{F}$ that was obtained in
  \cref{sec:simple-geod} by integrating against $\d \Volwp[0,2](\x,Y) = \frac{\delta(x-y)}{y} \d x
  \d y$.
\end{exa}}

\begin{exa}
  \label{exa:int_pop}
  Let $\curve$ be a loop filling the pair of pants $\mathbf{P}$ (i.e. our fixed surface
  of signature $(0,3)$). Then, the length of $\curve$ is an analytic function of
  the lengths of the three boundary components $(x_1, x_2, x_3)$ of $\mathbf{P}$, which
  we shall denote as $h_\curve : \R_{>0}^3 \rightarrow \R_{> 0}$. We
  then have that
  \begin{equation*}
    \av[\eqc{\mathbf{P}, \curve}]{F}
    = \frac{1}{n(\curve)}
    \iiint_{\R_{>0}^3} F(h_\curve(x_1, x_2, x_3)) \, \phi_g^{(0,3)}(x_1, x_2, x_3)
    \d x_1 \d x_2 \d x_3
  \end{equation*}
  where {$n(\curve)$ counts the positive homeomorphisms of $\mathbf{P}$ stabilizing $\curve$
    (including possibly some acting non-trivially on the boundary components of $\mathbf{P}$)} and
  \begin{multline*}
    \phi_g^{(0,3)}(\x)
    = \frac{x_1 x_2 x_3}{V_g}
    \left[ V_{g-2,3}(x_1,x_2,x_3)
    + \sum_{g_1+g_2+g_3=g} V_{g_1,1}(x_1) V_{g_2,1}(x_2) V_{g_3,1}(x_3)
    \right.  \\
    \left.
    + \sum_{\substack{\{i_1, i_2, i_3\} \\ = \{1,2,3\}}}
    \left( \frac{\delta(x_{i_1} - x_{i_2})}{x_{i_1}} V_{g-1,1}(x_{i_3})
    + \sum_{i=1}^{g-2}   V_{i,2}(x_{i_1}, x_{i_2}) V_{g-i-1,1}(x_{i_3}) \right)
     \right]. 
  \end{multline*}
\end{exa}

\begin{rem}
  In the two previous examples, {the integration over $(\x,Y) \in \cT^*_{g_\Sf,n_\Sf}$ is very
    simple because the length of the loop $\curve$ is entirely determined by the lengths $\x$ of the
    boundary components of the filled surface. However, in all cases but these ones, the integration
    on $\cT_{g_\Sf,n_\Sf}^*$ is much more complex to describe. Understanding the integration on
    $\cT_{g_\Sf,n_\Sf}^*$ is one of the key challenges tackled in the sequel of this article.}
\end{rem}

{
\begin{proof}[Proof of \cref{thm:express_average}]
  Let us fix a local loop~$(\Sf,\curve)$ in the equivalence class $\type$.  By definition,
  \begin{equation*}
    \av[\type]{F}
    = \Ewp{\sum_{\gamma \sim \type} F(\ell_X(\gamma))}
    = \frac{1}{m(\type)} \,
    \Ewp{\sum_{(Y,\gamma)} F(\ell_X(\gamma))}
  \end{equation*}
  where the sum runs over all the images $(Y, \gamma) = \phi(\Sf,\curve)$ of positive embeddings
  $\phi$ of $\Sf$ into the base surface $S_g$. Note that, while $Y$ has numbered boundary components, the coefficient $1/m(\type)$ removes the possible
  redundancies.

  We are now ready to split the orbit according to realizations.  For each realization
  $\mathfrak{R} \in R_g(\Sf)$, we fix a positive embedding
  $\phi_{g}^{\mathfrak{R}} : \Sf \rightarrow S_g$ associated to $\mathfrak{R}$. Then,
  \begin{equation}
    \label{eq:proof_int_form_type_1}
    \av[\type]{F}
    = \frac{1}{m(\type)} \sum_{\mathfrak{R} \in R_g(\Sf)}
    \Ewp{\sum_{\gamma \in \orb_g(\phi_{g}^{\mathfrak{R}}(\curve))}
     F(\ell_X(\gamma))}
 \end{equation}
 where $\orb_g(\phi_{g}^{\mathfrak{R}}(\curve))$ is the orbit of the loop $\phi_{g}^{\mathfrak{R}}(\curve)$ for
 the action of $\mcg_g$.

 We shall now see each individual term in this sum as a geometric function as per the conventions
 defined in \cref{s:mirz_int}. Indeed, for
 $\mathfrak{R} \in R_g(\Sf)$, we have
 \begin{equation*}
   \Ewp{\sum_{\gamma \in \orb_g(\phi_{g}^{\mathfrak{R}}(\curve))} F(\ell_X(\gamma))}
   = \Ewpo \big[{\Psi_{\mathfrak{R},\curve}^{\beta_g^\mathfrak{R}}}\big]
 \end{equation*}
 where the geometric function $\Psi_{\mathfrak{R},\curve}^{\beta_g^\mathfrak{R}}$ is associated to
 the following objects.
 \begin{itemize}
 \item The multi-curve $\beta_g^\mathfrak{R}$ is the image of $\partial \Sf$ by the homeomorphism
   $\phi_g^{\mathfrak{R}}$, with the numbering and orientation of $\partial \Sf$, and the following
   convention. For any index $j \in \{1, \ldots, \mathfrak{q}\}$ such that $(g_j,n_j)=(0,2)$, we do
   not include the largest component of $I_j$ in $\beta_g^{\mathfrak{R}}$. The resulting family of
   curves is therefore a multi-curve on the base surface $S_g$ with $k \leq n_\Sf$ components.
 \item The function $\Psi_{\mathfrak{R},\curve} : \R_{>0}^{k} \rightarrow \R$ is defined by
   \begin{equation*}
     \Psi_{\mathfrak{R},\curve}(x_{i_1}, \ldots, x_{i_k})
     := \int_{\cM_{g_\Sf,n_\Sf}(\x)} \sum_{\gamma \in \orb_{\Sf}(\curve)} F(\ell_Y(\gamma))  \d \Volwp[g_\Sf,n_\Sf,\x](Y)
   \end{equation*}
   with $\x \in \R_{>0}^{n_{\Sf}}$ the length vector obtained by completing $(x_{i_1}, \ldots, x_{i_k})$ with
   the identifications from the missing components of $\phi_g^{\mathfrak{R}}(\partial \Sf)$ corresponding to
   cylinders, and $\orb_{\Sf}(\curve)$ the mapping-class-group orbit of $\curve$ in $\Sf$.
 \end{itemize}
 Note that unfolding the integral defining $\Psi_{\mathfrak{R},\curve}$ allows to rewrite it as
 \begin{equation*}
   \Psi_{\mathfrak{R},\curve}(x_{i_1}, \ldots, x_{i_k})
   = \frac{1}{n_0(\type)}
   \int_{\cT_{g_\Sf,n_\Sf}(\x)} F(\ell_Y(\gamma))  \d \Volwp[g_\Sf,n_\Sf,\x](Y).
 \end{equation*}
 
 We now apply Mirzakhani's integration formula to compute the average of this geometric function,
 and obtain that $   \Ewpo \big[{\Psi_{\mathfrak{R},\curve}^{\beta_g^\mathfrak{R}}}\big]$ is equal to
 \begin{equation*}
\frac{1}{V_g} \int_{\R_{>0}^{k}}
\Psi_{\mathfrak{R},\curve}(x_{i_1}, \ldots, x_{i_k})
\,    x_{i_1} \ldots x_{i_k}
   \prod_{j: \chi_j >0} V_{g_j,n_j}(x_i, i \in I_j) \,
 \d x_{i_1} \ldots \d x_{i_k}.
 \end{equation*}
 We now observe that our definitions of $V_{0,2}$, $V_{\mathfrak{R}}$ and $\cT_{g_\Sf,n_\Sf}^*$
 allow us to rewrite this integral as
 \begin{equation*}
   \Ewpo \big[{\Psi_{\mathfrak{R},\curve}^{\beta_g^\mathfrak{R}}}\big]
   = \frac{1}{n_0(\type)} \frac{1}{V_g}
   \int_{\cT_{g_\Sf,n_\Sf}^*} F(\ell_Y(\curve)) \,
   x_{1} \ldots x_{n_\Sf} V_{\mathfrak{R}}(\x)
       \d \Volwp[g_\Sf, n_\Sf](\x,Y)  .
 \end{equation*}
 Summing over all $\mathfrak{R} \in R_g(\Sf)$ yields the claimed result since $m(\type)
 n_0(\type)=n(\type)$ and due to the expression of $\phi_g^\Sf$.
\end{proof}}

\subsection{Writing of the average as a density}
\label{s:density_writing}

Let us now justify that the averages $\av[\type]{F}$ can be written as densities
against the Lebesgue measure.

\begin{prp}
  \label{prp:existence_density}
  For any local type $\type$, there exists a unique locally integrable function
  $V_g^\type : \R_{> 0} \rightarrow \R_{\geq 0}$ such that, for any test function $F$,
  \begin{equation*}
    \av[\type]{F} = \frac{1}{V_g} \int_{0}^{+\infty} F(\ell) V_g^{\type}(\ell) \d \ell.
  \end{equation*}
\end{prp}

\begin{defa}
  We call $V_g^{\type}$ the \emph{volume function} associated with local type~$\type$ on surfaces of genus
  $g$.
\end{defa}

\begin{rem}
  By the collar lemma (see e.g. \cite[Theorem 4.2.2]{buser1992}), the length of any non-simple
  closed geodesic on a compact hyperbolic surface is greater than $2 \argcosh(3)$. In particular,
  for any local type $\type$ other than ``simple'' and any $g \geq 2$, the volume function $V_g^\type$ is
  identically equal to $0$ on $[0, 2 \argcosh(1)]$. In this article, we will focus mostly on the
  behaviour of $V_g^\type$ at infinity.
\end{rem}

The proof relies on the two following lemmas.

\begin{lem}
  \label{lem:abstract_density_writing}
  Let $\Omega \subset \R^d$ be a connected open set, and let $\lambda_\Omega$
  denote the Lebesgue measure on $\Omega$. For any non-constant real-analytic
  function $f : \Omega \rightarrow \R$, if
  \begin{equation}
    \label{eq:hyp_abstract_density_writing}
    \forall L > 0, \quad \int_\Omega \1{[0,L]}(f(\x)) \d \lambda_\Omega(\x) < + \infty,
  \end{equation}
  then, the push-forward of $\lambda_{\Omega}$ under $f$ admits a continuous
  density.  This statement also holds when pushing forward
  $v(\x) \d \lambda_{\Omega}(\x)$ for any continuous function
  $v : \Omega \rightarrow \R_{\geq 0}$, replacing $\lambda_\Omega$ in
  \eqref{eq:hyp_abstract_density_writing} with the measure $v(\x) \d \lambda_{\Omega}(\x)$.
\end{lem}

\begin{proof}
  First, we observe that under these hypotheses, the set $C \subset \Omega$ of critical points of
  $f$ has $0$-Lebesgue measure. Indeed, we prove by induction on the dimension $d$ that, for any
  real-analytic function $\tilde f : \Omega \rightarrow \R$ not identically equal to $0$, the set of
  zeros of $\tilde f$ has $0$-Lebesgue measure. For $d=1$, this comes from the fact that zeros are
  isolated. The induction from $d$ to $d+1$ uses the Fubini theorem. Our claim then follows by
  applying this intermediate result to the partial derivatives of $f$.

  Then, the Lebesgue measure on $\Omega$ coincides with the Lebesgue measure restricted to
  $\Omega \setminus C$. We can cover $\Omega \setminus C$ by a countable number of open sets
  $\Omega_i$, such that on each of those we can find a diffeomorphism $\varphi_i$ such that
  $f\circ\varphi_i(x_1, \ldots, x_d)=x_1$.  The push-forward of $\d \lambda_{\Omega_i}$, or
  $v(\x) \d \lambda_{\Omega_i}(\x)$, then obviously is absolutely continuous (the density is smooth, except on the critical set).
\end{proof}

\begin{proof}[Proof of \cref{prp:existence_density}]
  Let $\type$ be a local type of filling type $\Sf$.  Let us first notice that, by
  \cref{lem:bound_number_closed_geod}, there exists a constant $C>0$ such that
  for any $L \geq 0$,
  \begin{equation}
    \label{eq:bound_density_proof}
    \av[\type]{\1{[0,L]}}
    \leq \av{\1{[0,L]}}
    \leq C g e^L < + \infty.
  \end{equation}
  \Cref{thm:express_average} tells us that, in order to prove our claim, it is enough to
  apply \cref{lem:abstract_density_writing} to push forward the measure
  \begin{equation*}
    \frac{1}{n(\type)} \frac{1}{V_g}
    \, x_1 \ldots x_{n_\Sf}  V_{\mathfrak{R}}(\x) \d \Volwp[g_\Sf,n_\Sf] (\x, Y)
  \end{equation*}
  under the function $(\x, Y) \in \mathcal{T}_{g_\Sf,n_\Sf}^{\star} \mapsto \ell_Y(\curve)$, for
  each realization $\mathfrak{R} \in R_g(\Sf)$.
  
  Provided that $\mathfrak{R}$ is a realization for which $(g_j, n_j) \neq (0,2)$ for all $j$, the hypotheses
  of the lemma are satisfied, thanks to \cref{eq:bound_density_proof} and
  \cref{lem:teich_analytic}. Indeed, \cref{eq:teich_analytic_bound} implies that
  $(\x,Y) \mapsto \ell_Y(\curve)$ is not a constant function.

  Let us now briefly explain how to treat the case where some of the indices
  $1 \leq j \leq \mathfrak{q}$ satisfy $(g_j,n_j)=(0,2)$.  Rather than applying
  \cref{lem:abstract_density_writing} to the whole space $\mathcal{T}_{g_\Sf,n_\Sf}^{\star}$, we
  apply it to the lower-dimensional subspace where $x_{i}=x_{i'}$ for every pair of indices
  $(i, i')$ such that $I_j = \{i, i'\}$ and $g_j=0$. Once again, \eqref{eq:teich_analytic_bound}
  implies the length function is also non-constant on this new space, and we can conclude the same
  way.
\end{proof}

\subsection{Existence of an asymptotic expansion}

Let us now prove that the averages $\av[\type]{F}$ admit an asymptotic expansion in powers of $1/g$.

\begin{thm}
  \label{thm:exist_asympt_type}
  Let $\type$ be a local topological type. There exists a unique family of locally integrable functions
  $(f_k^\type)_{k \geq \chi(\type)}$ such that, for any $\ord \geq 0$, $\epsilon >0$, any large
  enough $g$,
  \begin{equation*}
    \frac{V_g^\type(\ell)}{V_g}
    = \sum_{k=\chi(\type)}^\ord \frac{f_k^\type(\ell)}{g^k}
    + \Ow[\ord,\chi(\type),\epsilon]{\frac{\exp \paren*{(1+\epsilon)\ell}}{g^{\ord+1}}}.
  \end{equation*}
\end{thm}

\begin{rem}
  We notice that the leading term of the asymptotic expansion of $\av[\type]{F}$ has order
  $1/g^{\chi(\type)}$. In particular, in all cases but the local type ``simple'', the leading order of
  $\av[\type]{F}$ decays as $1/g$ at least.
\end{rem}

\subsubsection{Rank of a realization}
\label{sec:rank-realization}

In order to compute asymptotic expansions in powers of $1/g$, it will be
convenient to introduce a notion of rank for a realization, which corresponds to
the height at which it appears in the expansion in powers of $1/g$.

\begin{defa}
  Let $\Sf$ be a filling type, and $g \geq 2$. We define the \emph{rank} {$\mathfrak{r}(\mathfrak{R})$} of a realization $\mathfrak{R} \in R_g(\Sf)$ by
  \begin{equation*}
    \mathfrak{r}(\mathfrak{R})
    := \chi(S_g) - \max_{1 \leq j \leq \mathfrak{q}} \chi_j
    = \chi(\Sf) + \sum_{j \neq j_+} \chi_{j} \geq 0
  \end{equation*}
  where $j_+$ is an index in $\{1, \ldots, \mathfrak{q}\}$ realizing
  $\max \{\chi_j, 1 \leq j \leq \mathfrak{q}\}$.
\end{defa}

\begin{rem}
  \label{rem:connected_realization}
  {Recall that, for us, we denote by $\chi$ the \emph{absolute} Euler characteristic, which is
    always non-negative.}  The rank of a realization is always greater than $\chi(\Sf)$.  In
  particular, the only filling type that can have realizations of rank $0$ is the filling type
  $(0,2)$, corresponding to simple loops.
\end{rem}

\begin{rem}
  Because cylinders, which have Euler characteristic $0$, are allowed in the definition of a
  realization, there exists several realizations of minimal rank $\chi(\Sf)$. Amongst them, there is
  a special one that we call the \emph{connected realization}, obtained by taking $\mathfrak{q}=1$,
  $I_1 = \partial \Sf$ and $g_1 = g - g_\Sf - n_\Sf + 1$. This is the only realization for which the
  complement of $\Sf$ in $S_g$ is connected.
\end{rem}

\begin{exa}
  Let us compute the ranks of all realizations of the filling type $(0,2)$.
  \begin{itemize}
  \item The connected realization $\mathfrak{q}=1$, $I_1 = \{1, 2\}$ and $g_1=g-1$, has rank $0$.
  \item Any other realization satisfies $\mathfrak{q}=2$, $I_1 = \{1\}$, $I_2 = \{2\}$, and
    $g_1, g_2 \geq 1$ such that $g_1 + g_2 = g$.  The rank of such a realization
    is $2 \min(g_1, g_2) - 1 \geq 1$.
  \end{itemize}  
  Note that the Weil--Petersson volume that appears in Mirzakhani's integration formula is
  $V_{g-1,2}/V_g \sim 1$ for the connected realization, and
  $V_{i,1} V_{g-i,1}/V_g \sim C(i)/g^{2i-1}$ for realizations of fixed rank $2i-1$. The notion of
  rank does correspond to the height of the realization in the asymptotic expansion of
  $\av[\mathbf{s}]{F}$ in powers of $1/g$, as we intended.
\end{exa}

The following lemma will allow us to reduce the number of ranks appearing when computing the
densities $\phi_g^\Sf$ introduced in \cref{thm:express_average}.

\begin{lem}
  \label{lem:limit_rank}
  For any filling type $\Sf$, any integer $\ord \geq \chi(\Sf)$, any large enough
  $g$,
  \begin{equation*}
    \frac{V_{g_\Sf,n_\Sf}}{V_g}
    \sum_{\substack{\mathfrak{R} \in R_g(\Sf) \\ \mathfrak{r}(\mathfrak{R}) \geq \ord}} \,
    \prod_{\substack{1 \leq j \leq \mathfrak{q} \\\chi_j>0}} V_{g_j, n_j}
    = \O[\chi(\Sf),\ord]{\frac{1}{g^\ord}}.
  \end{equation*}
\end{lem}

\begin{proof}[Proof of \cref{lem:limit_rank}]
  We start by proving the result for the order $\ord = \chi(\Sf)$.  First, we observe that the
  factor $V_{g_\Sf,n_\Sf}$ is unimportant because bounded by a constant depending only on
  $\chi(\Sf)$. We then split the quantity that we want to estimate depending on the partition
  $\vec{I}$ of $\partial \Sf$. Because $\partial \Sf$ has $\O[n_{\Sf}]{1}$ partitions, estimating
  each individual term is enough.  For a large enough~$g$, the condition that
  $\sum_{j=1}^\mathfrak{q} g_j = g - g_\Sf - n_\Sf + \mathfrak{q}$ implies that at least one
  coefficient $g_j$ is positive. We shall further split the sum we want to estimate depending on the
  subset $J \subsetneq \{1, \ldots, \mathfrak{q}\}$ of indices for which $(g_j,n_j)=(0,2)$, because
  there are $\O[n_\Sf]{1}$ such subsets.  As a consequence, we are left to bound for a fixed $\vec{I}$
  and $J \subsetneq \{ 1, \ldots, \mathfrak{q}\}$ the sum
  \begin{equation*}
    \sum_{\substack{\vec{g}: (\vec{I}, \vec{g}) \in R_g(\Sf)
        \\ \text{and } \chi_j =0 \Leftrightarrow j \in J}}
    \hspace{2pt}
    \prod_{j \notin J} V_{g_j, n_j}.
  \end{equation*}
  We observe that the condition on the indices $\vec{g}$ to obtain a realization can be rewritten as 
  \begin{equation*}
    \sum_{j \notin J} (2 g_j-2+n_j) = 2(g-g_\Sf) - n_\Sf = 2g'-2+n_\Sf,
  \end{equation*}
  for $g' := g-g_\Sf-n_\Sf+1$, because the summand is equal to zero for indices in
  $J$.

  By \cite[Lemma 3.2]{mirzakhani2019}, for any $k \geq 1$,
  $n_1, \ldots, n_k>0$ such that $\sum_{j=1}^k n_j$ and $n$ have the same
  parity, there exists $C = C(n,k,(n_i)_i)$ such that, for any large enough~$g$,
  \begin{equation}
    \label{lemm:cut_into_pieces}
    \sum_{(g_i)_i} \, \prod_{i=1}^k V_{g_i, n_i}
    \leq C \frac{V_{g,n}}{g^{k-1}}
  \end{equation}
  where the sum runs over all families of integers $(g_i)_{1 \leq i \leq k}$ such that
  $2g_i-2+n_i \geq 1$ and $\sum_{i=1}^k (2g_i-2+n_i) = 2g-2+n$.
  We apply this result with $k= \mathfrak{q} - \# J \geq 1$, the integers
  $(n_j)_{j \notin J}$, $n = n_\Sf$ and the genus $g' = g-g_\Sf-n_\Sf+1$. The parity
  condition on the integers is true because $\sum_{j \notin J} n_j$ has the same
  parity as $\sum_{j=1}^\mathfrak{q} n_j = n_\Sf$. We obtain
  \begin{equation*}
    \sum_{\substack{\vec{g}: (\vec{I}, \vec{g}) \in R_g(\Sf)
        \\ \text{and }\chi_j =0 \Leftrightarrow j \in J}}
    \hspace{2pt}
    \prod_{j \notin J} V_{g_j, n_j}
    = \O[\chi(\Sf)]{\frac{V_{g-g_\Sf-n_\Sf+1,n_\Sf}}{(g-g_\Sf-n_\Sf+1)^{\mathfrak{q}-\# J-1}}}
    = \O[\chi(\Sf)]{\frac{V_{g}}{g^{\chi(\Sf)}}}
  \end{equation*}
  because $\mathfrak{q}-\# J \geq 1$ and $V_{g-g_\Sf-n_\Sf+1,n_\Sf} = \O[\chi(\Sf)]{V_g/g^{2g_\Sf-2+n_\Sf}}$
  by \cite[Lemma 3.2]{mirzakhani2013}. This is exactly our claim for $\ord=\chi(\Sf)$.

  Let us now prove the result when $\ord > \chi(\Sf)$. The connected
  realization is the only realization for which $\mathfrak{q}=1$, and it has rank
  $\chi(\Sf) < \ord$; as a consequence, all realizations in the sum now satisfy
  $\mathfrak{q} \geq 2$. Let us pick a realization $(\vec{I},\vec{g})$ of rank $\geq \ord$,
  and let $j_+$ denote an index such that
  $\chi_{j_+} = \max_j \chi_j$.  By definition of the rank,
  $2g-2g_{j_+}-n_{j_+} \geq \ord$. On the other hand,
  \begin{equation*}
    \chi_{j_+} \geq \frac 1 {\mathfrak{q}} \sum_{j=1}^\mathfrak{q} \chi_j \geq \frac{2(g-g_\Sf) +
    n_\Sf}{n_\Sf} \underset{g \to \infty}{\longrightarrow} + \infty
  \end{equation*}
  and hence, provided $g$ is large enough,
  $2g_{j_+}-2+n_{j_+} \geq \ord$.

  Let us single out the term $V_{g_{j_+},n_{j_+}}$ in the quantity we want to estimate. We use the
  previous method to bound the summation over all possible $(g_j)_{j \neq j_+}$, in which we now
  include $V_{g_\Sf,n_\Sf}$, observing that
  \begin{align*}
    (2g_\Sf-2+n_\Sf) + \sum_{\substack{1 \leq j \leq \mathfrak{q} \\ j \neq j_+}} (2g_j-2+n_j) 
     = 2(g-g_{j_+})-n_{j_+}
     = 2 g''-2+n_{j_+}
  \end{align*}
  for $g'' := g-g_{j_+}-n_{j_+}+1$.  We therefore obtain that the sum we wish to
  estimate is bounded by a constant depending only on $\chi(\Sf)$ times
  \begin{equation}
    \label{eq:proof_rank_estimate_1}
    \sum_{\mathfrak{q} = 2}^{n_\Sf} \sum_{j_+=1}^{\mathfrak{q}} \sum_{n_{j_+}=1}^{n_\Sf-1}
    \sum_{\substack{g_{j_+} : \, 2g_{j_+}-2+n_{j_+} \geq \ord \\ \text{and } 2g-2g_{j_+}-n_{j_+} \geq \ord}}
    V_{g_{j_+},n_{j_+}} V_{g-g_{j_+}-n_{j_+}+1,n_{j_+}}.
  \end{equation}
  But \cite[Corollary 3.7]{mirzakhani2013} states that for any $\ord, n \geq 1$ and
  any large enough $g$,
  \begin{equation}
    \label{lem:cut_bounded_euler}
    \sum_{\substack{g_1, g_2 \\ g_1+g_2+n-1=g \\2g_i-2+n \geq \ord}}
    V_{g_1, n} V_{g_2, n} = \O[\ord,n]{\frac{V_{g}}{g^\ord}}.
  \end{equation}
  This allows us to conclude that \eqref{eq:proof_rank_estimate_1} is
  $\O[\chi(\Sf),\ord]{V_g/g^{\ord}}$, as announced.
\end{proof}

\subsubsection{Asymptotic expansion of the function $\phi_g^\Sf$}
\label{sec:asympt-expans-funct}

The key ingredient to proving \cref{thm:exist_asympt_type} is the following
asymptotic expansion on the function $\phi_g^\Sf$ associated to a filling type
$\Sf$.

\begin{prp}
  \label{lem:claim_exp_phi_T}
  For any filling type $\Sf$, there exists a unique family of distributions
  $(\psi_k^\Sf)_{k \geq \chi(\Sf)}$ satisfying the following. For any $\ord \geq 0$, there exists a
  constant $c_\ord^\Sf$ such that, for any large enough $g$, any $\x \in \R_{>0}^{n_\Sf}$,
  \begin{equation}
    \label{e:claim_exp_phi_T}
    \phi_g^\Sf(\x) = \sum_{k=\chi(\Sf)}^\ord \frac{\psi_k^\Sf(\x)}{g^k}
    + \mathrm{err}(\x)
  \end{equation}
  {with the error term satisfying the weak estimate, for any $L >0$,
  \begin{equation*}
    \int_{\sum_{i=1}^{n_\Sf} x_i \leq L} |\mathrm{err}(x_1, \ldots, x_{n_\Sf})| \d x_1 \ldots \d x_{n_\Sf}
    = \O[\ord,\chi(\Sf)]{\frac{(L+1)^{c_\ord^\Sf}}{g^{\ord+1}} e^{L/2}}.
  \end{equation*}}
  Furthermore, for all $k$, the function $\psi_k^\Sf$ can be uniquely written as a
  linear combination of distributions of the form
  \begin{equation}
    \label{eq:form_coeff_phi}
    \prod_{i \in V_0} x_i^{2k_i+1}
    \prod_{i \in V_+} x_i^{2k_i+1} \cosh \div{x_i}
    \prod_{i \in V_-} x_i^{2k_i} \sinh \div{x_i}
    \prod_{j=1}^k x_{i_j} \delta(x_{i_j}-x_{i_j'}),
  \end{equation}
  where {$V_0, V_+, V_-$ are disjoint subsets of $\{1, \ldots, n_\Sf\}$, of union denoted as
    $V$,} $(k_i)_{i \in V}$ are integers, and $\bigsqcup_{j=1}^k \{i_j,i_j'\}$ is a perfect matching
  of $\partial \Sf \setminus V$.
\end{prp}

The proof of this proposition is similar to the proof of \cref{prop:simple}.

\begin{proof}[Proof of \cref{lem:claim_exp_phi_T}]
  It comes as an easy consequence of \cref{lem:limit_rank} and the upper
  bound~\eqref{e:increase_Vgn_bound} on Weil--Petersson volumes that it is
  equivalent to prove this expansion for the full function $\phi_g^\Sf$ or its
  restriction $\phi_g^{\Sf, \ord}$ to all realizations of rank $= \ord$ for all
  $\ord \geq 0$. We shall therefore do the latter. The number of possibilities for
  the partition $\vec{I}$ of $\partial \Sf$ is fixed and independent of $g$, so it
  is furthermore enough to prove the result for the function
  $\phi_g^{\Sf,\ord,\vec{I}}$ restricted to realizations of partition $\vec{I}$, for
  every fixed partition~$\vec{I}$.

  Let $\vec{g}$ be a vector of integers such that $(\vec{I}, \vec{g})$ is a
  realization of rank $\ord$ of $\Sf$ in $S_g$. Let $j_+$ be an integer in
  $\{1, \ldots, \mathfrak{q}\}$ realizing the maximum $\max_j \chi_j$. The
  definition of the rank implies that, for all $j \neq j_+$,
  $\chi_j \leq \ord$. But the fact that $(\vec{I},\vec{g})$ is a realization in $S_g$
  means that
  \begin{equation*}
    \sum_{j=1}^\mathfrak{q} \chi_j = 2(g - g_\Sf) - n_\Sf
    \underset{g \rightarrow + \infty}{\longrightarrow} +\infty
  \end{equation*}
  and hence there exists an index $g_0 = g_0(\Sf,\ord)$ such that $\chi_{j_+} > \ord$ if
  $g \geq g_0$. In particular, if $g \geq g_0$, the maximal index $j_+$ is a uniquely defined index
  in $\{1, \ldots, \mathfrak{q}\}$. Hence we can further restrict ourselves to studying the restriction of the
  function $\phi_g^{\Sf,\ord,\vec{I}}$ to all realizations for which the maximal index is exactly $j_+$,
  for every fixed $1 \leq j_+ \leq \mathfrak{q}$. For the sake of readability, we shall only treat the case
  $j_+ = 1$, the others are similar.

  For such a realization to exist, by definition of the rank and of $R_g(\Sf)$, we need to assume that
  $\ord+n_1$ is an even integer $2m$, and then the genus $g_1$ of the surface attached to the
  boundary components of $\Sf$ lying in $I_1$ is determined to be equal to $g - m$. The quantity we
  need to estimate can therefore be rewritten as
  \begin{equation}
    \label{eq:proof_claim_exp_phi_T_rewriting}
    \frac{(\prod_{i \in I_1} x_i)V_{g-m, n_1}(x_i, i \in I_1)}{V_g}
    \sum_{\substack{g_2, \ldots, g_\mathfrak{q} \\ \sum_{j=2}^\mathfrak{q} \chi_j = \ord}}
    \prod_{j=2}^\mathfrak{q} \Big( \prod_{i \in I_j} x_i \Big) V_{g_j,n_j}(x_i, i \in I_j).
  \end{equation}
  The only term in this equation that depends on the genus $g$ is the ratio
  \begin{equation*}
    \frac{(\prod_{i \in I_1} x_i)V_{g-m, n_1}(x_i, i \in I_1)}{V_g}
    = \frac{V_{g-m, n_1}}{V_g}
    \frac{(\prod_{i \in I_1} x_i)V_{g-m, n_1}(x_i, i \in I_1)}{V_{g-m, n_1}}
  \end{equation*}
  which we can expand in powers of $1/g$ using \cref{thm:expansion_volume} and the expansions
  \eqref{eq:asymp_dev_ratio_same_euler} and \eqref{eq:asymp_dev_ratio_add_cusp} by Mirzakhani and
  Zograf. The dependency of the terms of this expansion as a function of $(x_i)_{i \in I_1}$ is made
  explicit in \cref{thm:expansion_volume}. The conclusion then follows directly, because
    $(\prod_{i \in I_j} x_i) V_{g_j,n_j}(x_i, i \in I_j)$ for $j \geq 2$ are polynomial functions in
  $(x_i)_{i \in I_j}$ odd in each variable if $(g_j,n_j) \neq (0,2)$, or equal to
  $x_i \delta(x_i - x_{i'})$ if $(g_j,n_j)=(0,2)$ and $I_j=\{i,i'\}$.

  The bound on the remainder comes from the fact that, by
  \cref{thm:expansion_volume}, the error term in the approximation of order $\ord$ of
  $V_{g-m,n_1}(x_i, i \in I_1)/V_{g-m,n_1}$ is
  \begin{equation*}
    \O[\ord,n_1]{\frac{(\norm{\x} +1)^{3\ord+1}}{g^{\ord+1}} \,
      \exp \Big( \frac 12 \sum_{i \in I_1} x_i \Big)}
  \end{equation*}
  which, once multiplied by the other volume polynomials and Dirac masses,
  yields a remainder that is of the claimed form.
\end{proof}

\subsubsection{Proof of \cref{thm:exist_asympt_type}}
\label{sec:expans-single-topol}

We are now ready to prove \cref{thm:exist_asympt_type}.

{
\begin{proof}
  For a $k \geq \chi(\Sf)$, inspired by \cref{thm:express_average}, we define
  \begin{equation}
    a_k^\type[F] :=
    \frac{1}{n(\type)}
    \int_{\cT_{g_\Sf,n_\Sf}^*}
    F(\ell_Y(\curve)) \, \psi_k^\Sf(\x) \d \Volwp[g_\Sf,n_\Sf](\x,Y)
  \end{equation}
  where we recall that $\psi_k^\Sf$ is the $k$-th term of the asymptotic expansion of $\phi_g^\Sf$.
  By the results of \cref{s:density_writing}, there exists a unique 
  density $f_k^{\type}$ such that
  \begin{equation*}
    a_k^\type[F] = \int_{0}^{+\infty} F(\ell) f_k^\type(\ell) \d \ell
  \end{equation*}
  for any test function $F$. 
  In order to conclude, we simply need to weakly bound the error made in this order $\ord$
  expansion or, in other words, the integral $    \int_0^L | \mathrm{Err}(\ell) | \d \ell$ where 
  \begin{equation*}
    \mathrm{Err}(\ell)
    := \frac{V_g^\type(\ell)}{V_g} - \sum_{k=\chi(\type)}^\ord
    \frac{f_k^\type(\ell)}{g^{k}},
  \end{equation*}
  for a $L >0$.
  We pick the test function $F(\ell) := \mathrm{sign}(\mathrm{Err}(\ell)) \1{[0,L]}(\ell)$ so that
  $$\int_0^L | \mathrm{Err}(\ell) | \d \ell = \int_0^\infty F(\ell) \, \mathrm{Err}(\ell) \d
  \ell.$$ Then, by definition of $V_g^\type$ and $f_k^\type$ for $\chi(\type) \leq k \leq \ord$, we
  can rewrite
  \begin{equation*}
    \int_0^L | \mathrm{Err}(\ell) | \d \ell
    = \frac{1}{m(\type)}
    \int_{\R_{>0}^{n_\Sf}}
    \int_{\cM_{g_\Sf,n_\Sf}(\x)}
    \sum_{\gamma \in \mathrm{Orb}_\Sf(\curve)} F(\ell_Y(\gamma)) \, \mathrm{err}(\x)
    \d \Volwp[g_\Sf,n_\Sf,\x](Y) \d \x 
  \end{equation*}
  where $\mathrm{err}(\x)$ is the remainder of the asymptotic expansion of order $\ord$ of
  $\phi_g^\Sf$. We bound quite roughly the quantity above by the number of geodesics filling $\Sf$,
  and obtain
  \begin{equation*}
    \int_0^L | \mathrm{Err}(\ell) | \d \ell
    \leq
    \int_{\R_{>0}^{n_\Sf}} \Bigg[
    \int_{\cM_{g_\Sf,n_\Sf}(\x)} \# \{ \gamma \text{ filling } \Sf : \ell_Y(\gamma) \leq L\}
     \d \Volwp[g_\Sf,n_\Sf,\x](Y)  \Bigg] |\mathrm{err}(\x)|  \d \x.
   \end{equation*}

   Let us note that the cardinal above is equal to $0$ unless $\sum_{i=1}^{n_{\Sf}} x_i \leq 2L$ by
   \cref{lem:length_fill}.  We now use \cref{thm:wu_xue_counting}
   and obtain
   \begin{align*}
     \# \{ \gamma \text{ filling } \Sf : \ell_Y(\gamma) \leq L \}
     = \O[\chi(\Sf),\epsilon]{e^{(1+\epsilon/2) L}
     \exp \Big(-\frac{1}{2}  \sum_{i=1}^{n_\Sf} x_i \Big)}.
   \end{align*}
   As a consequence, 
   \begin{align*}
     \int_0^L | \mathrm{Err}(\ell) | \d \ell
     & = \O[\chi(\Sf),\epsilon]{e^{(1+\epsilon/2)L}
       \int_{\sum_i x_i \leq 2L} V_{g_\Sf,n_\Sf}(\x) |\mathrm{err}(\x)|
       \exp \Big(-\frac{1}{2} \sum_{i=1}^{n_\Sf} x_i \Big) \d \x}.
   \end{align*}
   We use the naive bound \eqref{lem:increase_Vgn} on the factor $V_{g_\Sf,n_\Sf}(\x)$ and split the
   integral to apply the bound on $\mathrm{err}(\x)$ from \cref{lem:claim_exp_phi_T}. We obtain that
   \begin{align*}
     \int_0^L | \mathrm{Err}(\ell) | \d \ell
     & = \O[\chi(\Sf),\epsilon]{(L+1)^{3\chi(\Sf)} e^{(1+\epsilon/2)L}
       \sum_{j=0}^{\floor{2L}} e^{-j/2}
       \int_{j \leq \sum_{i}x_i < j+1}  |\mathrm{err}(\x)|
        \d \x} \\
     & = \O[\ord,\chi(\Sf),\epsilon]{(L+1)^{3\chi(\Sf)+c_\ord^\Sf+1}\frac{ e^{(1+\epsilon/2)L}}{g^{\ord+1}}
       } 
   \end{align*}
   which leads to the claimed result.
\end{proof}}

\begin{rem}
  \label{rem:proof_asymptotic}
  Let $\type = \eqc{\Sf,\curve}$ be a local topological type.
  The proof  yields that, for any $k \geq 0$, the $k$-th term $f_k^\type$ of the asymptotic
  expansion of $V_g^\type/V_g$ can be computed using the following relation, true for any test
  function $F$: {
  \begin{align}
    \label{eq:def_fk}
      \int_0^{+ \infty} F(\ell) \, f_k^\type(\ell) \d \ell
      & = \frac{1}{n(\type)}
        \int_{\mathcal{T}_{g_\Sf,n_\Sf}^\star}
        F(\ell_Y(\curve)) \, \psi_k^\Sf(\x) \d \Volwp[g_\Sf,n_\Sf](\x,Y) 
  \end{align}}
  where $\psi_k^\Sf$ is the $k$-th term of the asymptotic expansion of $\phi_g^\Sf$ from
  \cref{lem:claim_exp_phi_T}.
\end{rem}

\subsection{Useful generalizations}
\label{sec:gener-bord-hyperb}

We briefly explain here how to extend the notations and results of Sections
\ref{sec:topol-types-curv} and \ref{sec:average-over-local} to broader settings.

A first useful observation is that we can extend the definition of local
topological types to types $\type = \eqc{\Sf, \vec{\curve}}$, where $\vec{\curve}$ is
a \emph{multi-loop} filling $\Sf$, i.e. an ordered family
$(\curve_1, \ldots, \curve_\cc)$ of closed loops such that all connected
components of $\Sf \setminus (\bigcup_{i=1}^\cc \curve_i)$ are disks or annular
regions around boundary components of $\Sf$. For any family of  test
functions $F_i: \R_{\geq 0} \rightarrow \R$, $1 \leq i \leq \cc$, we can
naturally define
\begin{equation*}
  \av[\type]{F_1, \ldots, F_\cc}
  := \Ewp{\sum_{(\gamma_1, \ldots, \gamma_\cc) \sim \type}
    \prod_{i=1}^\cc F_i(\ell_X(\gamma_i))}.
\end{equation*}
Note that, more generally, we could also have averaged any test function
$\R_{>0}^\cc \rightarrow \R$ (i.e. not necessarily a product of $\cc$ functions).

In this new setting, the integration formula proven in \cref{thm:express_average} remains identical:
\begin{equation*}
  \av[\type]{F_1, \ldots, F_\cc}
  = \frac{1}{n(\type)}
  \int_{\cT_{g_\Sf,n_\Sf}^*} \prod_{i=1}^\cc F_i(\ell_Y(\curve_i)) \, \phi_g^\Sf(\x) \d \Volwp[g_\Sf,n_\Sf](\x)
\end{equation*}
with the multiplicity constant $n(\type)$ now counting positive homeomorphisms of $ \Sf$ stabilising
the multi-loop $\vec{\curve}$ (up to isotopy). Notably, the function $\phi_g^\Sf$ present in the
integration formula is unchanged. We can therefore write the average $\av[\type]{F_1, \ldots, F_\cc}$
as an integral against a generalised volume function $V_g^\type$, now a function of $\cc$ variables:
\begin{equation*}
  \av[\type]{F_1, \ldots, F_\cc}
  = \frac{1}{V_g} \int_{\R_{>0}^\cc}
  \paren*{\prod_{i=1}^\cc F_i(x_i)} V_g^\type(\x) \d \x.
\end{equation*}
The ratio $V_g^\type/V_g$ can be expanded in powers of $1/g$ by replacing the
function $\phi_g^\Sf$ by its asymptotic expansion as in the case of one single
loop.

Actually, the multi-loop $(\curve_1, \ldots, \curve_\cc)$ does not necessarily need to fill the
surface $\Sf$ for the definition of the average $\av[\type]{F}$ to make sense. A case that can be
interesting to consider is the case where $\cc := n_\Sf$ and
$(\curve_1, \ldots, \curve_{n_\Sf}) := \partial \Sf$.

Another useful generalisation consists in considering families of local types
$\vec{\type} = (\type_1, \ldots, \type_m)$, where for all $i$, $\type_i = \eqc{\Sf^i, \curve_i}$,
where the family of filling types is given by some integer vectors
$\mathbf{g} = (g_i)_{1 \leq i \leq m}$ and $\mathbf{n} = (n_i)_{1 \leq i \leq m}$. We then say a
family of loops $(\gamma_1, \ldots, \gamma_m)$ on $S_g$ belongs to the local type $\type$ if:
\begin{itemize}
\item the filled surfaces $S(\gamma_i)$ for $1 \leq i \leq m$  are
  \emph{disjoint};
\item there exists a family of positive homeomorphisms
  $\phi_i : S(\gamma_i) \rightarrow \Sf^i$ such that the loops
  $\phi \circ \gamma_i$ and $\curve_i$ are homotopic in $\Sf^i$.
\end{itemize}
We then can write for a family of test functions
$F_i : \R_{\geq 0} \rightarrow \R$
\begin{equation*}
  \av[\vec{\type}]{F_1, \ldots, F_m}
  = \Ewp{\sum_{(\gamma_1, \ldots, \gamma_m) \sim \vec{\type}}
    \prod_{i=1}^m F_i(\ell_X(\gamma_i))}
  =  \frac{1}{V_g} \int_{\R_{>0}^m} \prod_{i=1}^m F_i(x_i) \, V_g^{\vec \type}(\x) \d \x
\end{equation*}
for a generalized density $V_{g}^{\vec \type}$, which can be computed using
\begin{equation*}
  \av[\vec \type]{F_1, \ldots, F_m}
  = \prod_{i=1}^m\frac{1}{n(\type_i)}
  \int_{\cT_{\mathbf{g}, \mathbf{n}}^*}
  \prod_{i=1}^m F_i(\ell_{Y_i}(\curve_i)) 
  \, \phi_g^{\vec{\Sf}}(\x) \d \Volwp[\mathbf{g},\mathbf{n}](\vec{\x}, \vec{Y})
\end{equation*}
where:
\begin{itemize}
\item the integration takes place on the product space $\cT_{\mathbf{g}, \mathbf{n}}^* =
  \prod_{i=1}^m \cT_{g_i,n_i}^*$  equipped with the Weil--Petersson volume $\prod_{i=1}^m \d
  \Volwp[g_i,n_i](\x^{(i)},Y_i)$;
\item the function $\phi_g^{\vec{\Sf}}$ enumerates all possible realizations of
  $\Sf^1 \sqcup \ldots \sqcup \Sf^m$ in $S_g$ (as in the case $m=1$, we allow realizations
  containing cylinders, i.e. it is possible to glue a boundary component of $\Sf^i$ to a boundary
  component of $\Sf^{i'}$ for $i \neq i'$, unless they are both cylinders).
\end{itemize}
In this situation, the function
$\phi_g^{\vec{\Sf}}$ can once again be expanded in powers of $1/g$, which yields an asymptotic
expansion of $V_g^{\vec{\type}}/V_g$ in powers of $1/g$.


\section{Average over all geodesics}
\label{sec:existence-asym-ext}

Let us now extend some of the observations of \cref{sec:average-over-local} to the average $\av{F}$
over \emph{all} primitive closed geodesics.  First, we observe that adding all local topological
types leads to the following statement, extending \cref{thm:express_average} to $\av{F}$.

\begin{thm}
  \label{thm:express_average_all}
  For any $g \geq 3$, any  test function $F$, 
  \begin{equation}
    \label{eq:orbit_decomposition_all}
    \av{F}
    = \sum_{\Sf \text{ filling type}} \frac{1}{n_\Sf!}
    \int_{\cM^*_{g_\Sf,n_\Sf}(\x)}
    \sum_{\gamma \text{ filling } \Sf} F(\ell_Y(\gamma))
    \, \phi_g^\Sf(\x) \d \Volwp[g_\Sf,n_\Sf](\x,Y).
  \end{equation}
\end{thm}

We also observe that the average $\av{F}$ can be written as a density.

\begin{prp}
  \label{prp:existence_density_all}
  There exists a unique locally integrable function
  $V_g^{\mathrm{all}} : \R_{>0} \rightarrow \R_{\geq 0}$ such that, for any
   test function $F$,
  \begin{equation*}
    \av{F} = \frac{1}{V_g} \int_{0}^{+\infty} F(\ell) V_g^{\mathrm{all}}(\ell) \d \ell.
  \end{equation*}
\end{prp}

The proofs of both these statements are the same that in the case of one local
type.

Let us now prove the following result, which is an expansion for the average $\av{F}$ obtained by
summing over \emph{all} closed geodesics.

\begin{thm}
  \label{thm:existence_asym_all}
  There exists a unique family of continuous functions $(f_k^{\mathrm{all}})_{k \geq 0}$ such that,
  for any integer $A \geq 1$, $\ord \geq 0$, $\epsilon > 0$, and any large enough $g$,
  $L = A \log(g)$,
  \begin{equation}
    \label{eq:existence_asym_all}
    \frac{V_g^{\mathrm{all}}(\ell)}{V_g} \1{[0, L]}(\ell)
    = \sum_{k=0}^{\ord} \frac{f_k^{\mathrm{all}}(\ell)}{g^k} \1{[0, L]}(\ell)
    + \Ow[\epsilon,\ord,A]{\frac{\exp{((1+\epsilon) \ell)}}{g^{\ord+1}}}.
  \end{equation}
\end{thm}

The proof of this result is very similar to the proof of \cref{thm:exist_asympt_type}.  The only
difference is that we would \emph{a priori} need to sum over all possible filling types $\Sf$ in
order to expand $\av{F}$, which would be an issue since all constants in
\cref{thm:exist_asympt_type} depend on $\Sf$.  It is in order to address this difficulty that we
restrict ourselves to examining geodesics of length $\leq L = A \log(g)$, by including $\1{[0,L]}$
in \cref{eq:existence_asym_all}. This allows to use the following proposition to restrict the number
of filling types in the sum.

 \begin{prp}\label{p:apriori}
   Let $L \geq 1$. For any large enough $g$, the probability for a random hyperbolic surface of
   genus $g$ to contain a multi-loop of length $\leq L$ filling a surface of absolute Euler
   characteristic $> \chi$ is $\O[\chi]{L^{c(\chi)} e^{L} /g^{\chi+1}}$ for a constant $c(\chi)$.
\end{prp} 
\begin{proof}
  If there is a multi-loop of length $\leq L$ filling a surface of absolute Euler characteristic
  greater than $\chi$, then there exists a subsurface of absolute Euler characteristic $\chi+1$ and
  boundary length $\leq 2L$, by \cite[Lemma 4.13]{Moebius}.  By Markov's inequality,
  \begin{align*}
    & \Pwp{X \text{ contains a multi-loop } \gamma : \ell_X(\gamma) \leq L
      \text{ and }\chi(S(\gamma)) > \chi} \\
    & \leq \sum_{\Sf: \chi(\Sf)=\chi+1}
      \Ewpo \brac*{\#  \{ Y\subset X : Y \text{ homeomorphic to } \Sf,
      \ell(\partial Y) \leq 2L\}} \\
    & \leq \sum_{\Sf: \chi(\Sf)=\chi+1} \frac{1}{n_\Sf!} 
      \int_{\R_{>0}^{n_\Sf}} \1{[0,2L]}(x_1+\ldots+x_{n_\Sf}) \,
      \phi_g^\Sf(\x) V_{g_\Sf,n_\Sf}(\x) \d \x
  \end{align*}
  as soon as $2g-2 > 2 (\chi+1)$, by Mirzakhani's integration formula, with $(\phi_g^\Sf)_\Sf$ the
  functions defined in \cref{eq:def_phi_S}. Let $\Sf$ be a filling type. Using the upper
  bounds~\eqref{lem:increase_Vgn} and \eqref{e:increase_Vgn_bound} on Weil--Petersson volume
  polynomials, we obtain that the integral above is at most
  \begin{align*}
       (2L)^{2n_\Sf+3\chi(\Sf)} e^L \frac{V_{g_\Sf,n_\Sf}}{V_g}
      \sum_{\mathfrak{R}  \in R_g(\Sf)}
       \prod_{\substack{1 \leq j \leq \mathfrak{q} \\ \chi_j>0}} V_{g_j,n_j}
  \end{align*}
  which leads to our claim using \cref{lem:limit_rank} for $\ord=\chi(\Sf) = \chi+1$.
\end{proof}

{
\begin{proof}[Proof of \cref{prp:existence_density_all}]
  Applying \cref{p:apriori} together with Lemma \ref{lem:bound_number_closed_geod}
  allows to decompose the average $\av[\text{all}]{F \1{[0,L]}} $ into
  \begin{align*} \label{e:bound_chi}
    \av[\text{all}]{F \1{[0,L]}}
    & = \sum_{\type : \chi(\type) \leq \chic} \av[\type]{F}
      + \Ewp{\sum_{\substack{\gamma \in \mathcal{G}(X) \\ \chi(S(\gamma)) > \chic}}
    F(\ell_X(\gamma)) \1{[0,L]}(\ell_X(\gamma))} \\
   & = \sum_{\type : \chi(\type) \leq \chic} \av[\type]{F}
    + \cO_\chic\Big(\frac{L^{c(\chic)+1} e^{L}}{g^{\chic}} \norm{F(\ell) e^\ell}_\infty\Big) 
  \end{align*}
  by a simple $L^\infty$ upper bound on the expectation (note that here we used the fact that the
  Euler characteristic of a surface of genus $g$ is linear in $g$). Taking $L=A\log g$, when we
  specify $\chic= A + \ord +3$, then the error term is
  $\cO_{\ord,A}( \norm{F (\ell) e^{\ell}}_\infty/g^{\ord+1})$. As a consequence, it is enough to
  study the term $\sum_{\chi(\type) \leq \chic} \av[\type]{F \1{[0,L]}}$ and apply the asymptotic
  expansion to each term separately.
\end{proof}}


\section{The case of the figure-eight}
\label{sec:generalised_convolutions}

The aim of this section is to prove \ref{chal:FR_type} in the case of a figure-eight
in a pair of pants, which is the local topological type represented in
\cref{fig:eight_pop}.

\begin{figure}[h]
  \centering
  \includegraphics[scale=0.5]{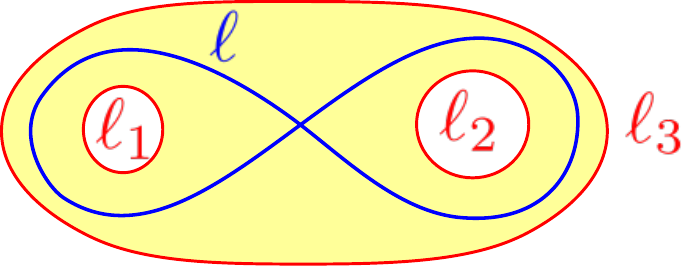}
  \caption{A figure-eight in a pair of pants.}
  \label{fig:eight_pop}
\end{figure}

The figure-eight is the simplest example of loop filling the pair of pants, which is why we address it
first.  We prove the following.

\begin{thm}
  \label{thm:eight}
  Let $\curve$ be a figure-eight filling a pair of pants $\mathbf{P}$, and
  $\type = [\mathbf{P}, \curve]_{\mathrm{loc}}$. Then, for all $k \geq 0$, the function
  $f_k^{\type}$ is a Friedman--Ramanujan function.
\end{thm}

By ``simplest'', we mean that it is the loop with the least self-intersections filling the pair of
pants~$\mathbf{P}$, as opposed to more complicated loops such as the one represented in the left part
of \cref{fig:filled_non_simple}. However, when we prove \ref{chal:FR_type} for all other loops
filling $\mathbf{P}$ in \cref{sec:other-geodesics}, we will notice that it is actually the one for
which the analysis is the hardest.

\subsection{Expression of the density function as an integral}
\label{sec:expression-as-an}

Let $F : \R_{\geq 0} \rightarrow \R$ be a test function. We recall that, by definition, the
functions $f_k^{\type}$ will appear when computing the asymptotic expansion of the average
\begin{equation*}
  \av[{\type}]{F}
  = \Ewp{\sum_{\gamma \sim {\type}} F(\ell_X(\gamma))}.
\end{equation*}
We saw in \cref{thm:express_average} that this expectation can be expressed as
\begin{equation}
  \label{eq:eight_pop_integral}
  \av[{\type}]{F}
  =  \iiint_{\R_{>0}^3} F(h(\ell_1, \ell_2, \ell_3)) \,
  \phi_g^{\mathbf{P}}(\ell_1, \ell_2, \ell_3) \,
  \d \ell_1 \d \ell_2 \d \ell_3,
\end{equation}
where 
\begin{itemize}
\item $\phi_g^{\mathbf{P}}$ is the sum over all realisations of the pair of pants, defined
  in \eqref{eq:def_phi_S} and computed in \cref{exa:int_pop};
\item the length function $h : \R_{>0}^3 \rightarrow \R_{>0}$ associates to
  $(\ell_1, \ell_2, \ell_3)$ the length of the figure-eight going around the components $1$ and $2$
  in the pair of pants of boundary components of lengths $\ell_1, \ell_2, \ell_3$. It is computed in
  \cite[equation 4.2.3]{buser1992}:
  \begin{equation}
    \label{eq:formula_length_eight}
    \cosh \div{h(\ell_1, \ell_2, \ell_3)}
    = 2 \cosh \div{\ell_1} \cosh \div{\ell_2} + \cosh \div{\ell_3}.
  \end{equation}
\end{itemize}

By \cref{lem:claim_exp_phi_T}, the function $\phi_g^\mathbf{P}(\ell_1, \ell_2, \ell_3)$ has an
expansion in powers of $1/g$, of which the dependency with respect to $(\ell_i)_{1 \leq i \leq 3}$
is detailed.  When substituting $\phi_g^\mathbf{P}$ by its expansion in
\eqref{eq:eight_pop_integral}, we obtain that any term of the asymptotic expansion of the average
$\av[{\type}]{F}$ in powers of $1/g$ is a linear combination of integrals of the form
\begin{align}
  \label{eq:term_F_exp}
  & \iiint_{\R_{>0}^3} F(h(\ell_1, \ell_2, \ell_3))
  \, \prod_{i=1}^3 \phi_i(\ell_i) \, \d \ell_1 \d \ell_2 \d \ell_3 \\
  \label{eq:term_F_exp_tor}
  & \iiint_{\R_{>0}^3} F(h(\ell_1, \ell_2, \ell_3))
    \,  \phi_i(\ell_i) \, \ell_j \delta(\ell_j-\ell_k) \d \ell_1 \d \ell_2 \d \ell_3
    \quad \text{for } \{i,j,k\}=\{1,2,3\}
\end{align}
where for all $i$, the functions $\ell_i \mapsto \phi_i(\ell_i)$ is of the form
\begin{equation}
  \label{eq:cases_h_i}
  \ell_i^{2k+1} \quad \text{or} \quad
  \ell_i^{2k+1} \cosh \div{\ell_i} \quad \text{or} \quad
  \ell_i^{2k} \sinh \div{\ell_i}.
\end{equation}

\subsection{Level-set decomposition}

Because we want to view the average $\av[{\type}]{F}$ as the integral of $F$ against a
density, we rewrite the integral \eqref{eq:term_F_exp} as
\begin{equation*}
  \int_0^{+ \infty} F(\ell) \left( \iint_{h(\ell_1, \ell_2, \ell_3) = \ell}
    \, \prod_{i=1}^3 \phi_i(\ell_i) \,
    \frac{\d \ell_1 \d \ell_2 \d \ell_3}{\d \ell} \right) \d \ell.
\end{equation*}
Let us give precise meaning to this writing. We fix $\ell_1, \ell_2 > 0$. Then, the application
\begin{equation*}
  h_{\ell_1, \ell_2} :
  \begin{cases}
    \R_{>0}   \rightarrow I(\ell_1, \ell_2) \\
    \ell_3 \hspace{0.5em}
    \mapsto h(\ell_1, \ell_2, \ell_3) = 2 \argcosh \left( 2 \cosh \div{\ell_1}
      \cosh \div{\ell_2} + \cosh \div{\ell_3} \right)
  \end{cases}
\end{equation*}
is a diffeomorphism from $\R_{>0}$ to an interval
$I(\ell_1, \ell_2)$ of $\R_{>0}$.  We introduce the following notation.

\begin{nota}
  Let $\ell > 0$. For any integrable function $G : \R_{>0}^3 \rightarrow \C$, we
  define
  \begin{align*}
      \iint_{h(\ell_1, \ell_2, \ell_3) = \ell} G(\ell_1, \ell_2, \ell_3) \,
      \frac{\d \ell_1 \d \ell_2 \d \ell_3}{\d \ell}  
     :=
      \iint _{\R_{>0}^2} 
      \frac{ \1{I(\ell_1, \ell_2)}(\ell) \, G(\ell_1, \ell_2, h_{\ell_1, \ell_2}^{-1}(\ell))}
    {h'_{\ell_1, \ell_2}(h_{\ell_1, \ell_2}^{-1}(\ell))}
    \d \ell_1 \d \ell_2.
\end{align*}
\end{nota}

This formula corresponds to considering the variable $\ell_3$ as a function of the variables
$\ell_1, \ell_2, \ell$, once restricted to the $2$-dimensional level-set
$\{(\ell_1, \ell_2, \ell_3) \; : \; h(\ell_1, \ell_2, \ell_3) = \ell\}$. The derivative
that appears corresponds, formally, to writing
\begin{equation*}
  \frac{\d \ell_1 \d \ell_2 \d \ell_3}{\d \ell} =
  \frac{\partial \ell_3}{\partial \ell}(\ell_1, \ell_2, \ell) \d \ell_1 \d \ell_2.
\end{equation*}
It is easy to check using the local inversion theorem that one obtains the same quantity by
performing this operation on the variables $\ell_2$ and $\ell_3$, or $\ell_1$ and $\ell_3$.

We do the same for the integrals of the form \eqref{eq:term_F_exp_tor}. For instance, for
$(i,j,k)=(1,2,3)$, we obtain the integral
\begin{equation}
  \label{eq:term_F_exp_tor_level_set}
  \int_{h(\ell_1,\ell_2,\ell_2)=\ell} \phi_1(\ell_1) \, \ell_2 \,
  \frac{\d \ell_1 \d \ell_2}{\d \ell} \cdot
\end{equation}

\subsection{Reformulation of the question and proof in simple cases}
\label{sec:reform-quest}

It will be handy to observe that, for any of the three cases in~\eqref{eq:cases_h_i}, we can write
\begin{equation*}
  f_i(\ell_i) := \frac{\phi_i(\ell_i)}{\sinh \div{\ell_i}}
  = p_i(\ell_i) + \O{(\ell_i+1)^{c_i} \exp \left(-\frac{\ell_i}{2}\right)}
\end{equation*}
where $p_i$ is a polynomial function (possibly equal to zero) and $c_i \geq 0$,
because 
\begin{equation*}
  \frac{\ell_i}{\sinh \div{\ell_i}} = \O{(\ell_i+1) \, e^{-\frac{\ell_i}{2}} }
  \quad \text{and} \quad
  \cosh \div{\ell_i} = \sinh \div{\ell_i} + e^{-\frac{\ell_i}{2}}.
\end{equation*}
All of the steps taken so far lead us to the following lemma.

\begin{lem}
  \label{lem:expression_before_change_var}
  For any integer $k \geq 0$, the function $\ell \mapsto f_k^{{\type}}(\ell)$ is a linear combination of
  functions of the form
  \begin{equation}
    \label{eq:integral_to_compute_FR}
    \mathrm{Int}[f_1, f_2, f_3] : \ell \mapsto \iint_{h(\ell_1, \ell_2, \ell_3) = \ell} \,
    \prod_{i=1}^{3} f_i(\ell_i) \sinh \div{\ell_i} \, \frac{\d \ell_1 \d \ell_2 \d \ell_3}{\d \ell}
  \end{equation}
  and
  \begin{equation}
    \label{eq:integral_to_compute_FR_tor}
    \begin{cases}
      & \mathrm{Int}^{\mathrm{tor}}_1[f_1] : \ell \mapsto \int_{h(\ell_1, \ell_2, \ell_2)
        = \ell} \, f_1(\ell_1) \sinh \div{\ell_1} \, \ell_2 \, \frac{\d \ell_1 \d
        \ell_2}{\d \ell} \\
      & \mathrm{Int}^{\mathrm{tor}}_2[f_2] : \ell \mapsto \int_{h(\ell_3, \ell_2, \ell_3)
        = \ell} \, f_2(\ell_2) \sinh \div{\ell_2} \, \ell_3 \, \frac{\d \ell_2 \d
        \ell_3}{\d \ell} \\
      & \mathrm{Int}^{\mathrm{tor}}_3[f_3] : \ell \mapsto \int_{h(\ell_1, \ell_1, \ell_3)
        = \ell} \, f_3(\ell_3) \sinh \div{\ell_3} \, \ell_1 \, \frac{\d \ell_1 \d
        \ell_3}{\d \ell} 
    \end{cases}
  \end{equation}
  where for any $i \in \{ 1, 2, 3 \}$, $f_i$ satisfies:
  \begin{equation}
    \label{eq:expansion_assumption_fi}
    \abso{f_i(\ell_i) - p_i(\ell_i)} \leq c_i(\ell_i+1)^{c_i} \exp \left(-\frac{\ell_i}{2}\right)
  \end{equation}
  for a polynomial function $p_i$ and a constant $c_i>0$.
\end{lem}

In particular, if we prove that the integrals in \eqref{eq:integral_to_compute_FR} and
\eqref{eq:integral_to_compute_FR_tor} are Friedman--Ramanujan for any $(f_i)_{1 \leq i \leq 3}$
satisfying \eqref{eq:expansion_assumption_fi}, then we can conclude that $f_k^{{\type}}$ is too, by
linearity, hence proving \cref{thm:eight}.  Let us prove this in some simple cases.

\begin{lem}
  \label{lem:easy_cases_remainder}
  Let $(f_i)_{1 \leq i \leq 3}$ be measurable functions, each satisfying
  \eqref{eq:expansion_assumption_fi}.
  \begin{enumerate}
  \item \label{it:easy_cases_remainder_1} If $p_1=p_2=0$, then
    $\mathrm{Int}[f_1, f_2, f_3] \in \FRrem \subset \FR$.
  \item \label{it:easy_cases_remainder_3} If $p_3=0$, then
    $\mathrm{Int}[f_1, f_2, f_3] \in \FRrem \subset \FR$.
  \item \label{it:easy_cases_remainder_tor} Without any further hypothesis, the
    integrals in \eqref{eq:integral_to_compute_FR_tor} belong in
    $\FRrem \subset \FR$.
  \end{enumerate}
\end{lem}

\begin{proof}
  Before we proceed to the proof, let us observe that
  \begin{equation}
    \label{eq:bounds_ell_ell_i}
    \ell_1 + \ell_2 \leq \ell \quad \text{and} \quad
    \ell_3 \leq \ell.
  \end{equation}
  This can be seen on the expressions, or directly on \cref{fig:eight_pop}, by
  minimality of the length of a geodesic in a free homotopy class.
  
  The proof of \cref{lem:easy_cases_remainder} for all integrals of the form
  \eqref{eq:integral_to_compute_FR_tor} are similar, and we therefore detail the computation for
  $\mathrm{Int}_1^{\mathrm{tor}}$ only.  First, we use the level-set condition
  \begin{equation*}
    \cosh \div{\ell} = \cosh \div{\ell_2} \paren*{2 \cosh \div{\ell_1} + 1}
  \end{equation*}
  to compute that
  \begin{equation*}
    \frac{\partial \ell_1}{\partial \ell}
    = \frac{\sinh \div{\ell}}{2 \sinh \div{\ell_1}
        \cosh \div{\ell_2}}
    \leq \frac{\sinh \div \ell}{\sinh \div{\ell_1}} \cdot
  \end{equation*}
  The hypothesis \eqref{eq:expansion_assumption_fi} on $f_1$ implies that there exists a
  constant $c>0$ such that
  \begin{equation*}
    \abso{f_1(\ell_1)} \leq c (\ell_1+1)^{c}.
  \end{equation*}
  It follows, using \eqref{eq:bounds_ell_ell_i}, that
  \begin{equation*}
    \left|f_1(\ell_1) \sinh \div{\ell_1} \ell_2 \frac{\partial \ell_1}{\partial \ell}\right|
    \leq c(\ell+1)^{c+1} \sinh \div{\ell}
  \end{equation*}
  and hence
  \begin{equation*}
    \abso{\mathrm{Int}^{\mathrm{tor}}_1[f_1](\ell)}
    \leq c (\ell+1)^{c+1} \sinh \div{\ell} \int_{0}^{\ell} \d \ell_2
    \leq c (\ell+1)^{c+2} \exp \div{\ell}
  \end{equation*}
  which means that $\ell \mapsto \mathrm{Int}^{\mathrm{tor}}_1[f_1](\ell)$ belongs in $\FRrem$.
  
  Let us now treat case when $p_3=0$. In that case, now using
  \begin{equation*}
    \cosh \div{\ell}
    = 2 \cosh \div{\ell_1} \cosh \div{\ell_2} + \cosh \div{\ell_3},
  \end{equation*}
  we obtain
  \begin{equation}
    \label{e:partial_deri_l2}
    \frac{\partial \ell_2}{\partial \ell}
    = \frac{\sinh \div{\ell}}{2 \cosh \div{\ell_1} \sinh \div{\ell_2}} \cdot
  \end{equation}
  We then observe that our assumptions on $(f_i)_{1 \leq i \leq 3}$ implies that
  there exists a constant $c \geq 0$ such that, for any $\ell_1, \ell_3, \ell$ for
  which $\ell_2$ is well-defined,
  \begin{equation}
    \label{e:bound_noterm_3}
    \left|\prod_{i=1}^{3} f_i(\ell_i) \sinh \div{\ell_i} \, \frac{\partial \ell_2}{\partial \ell}\right|
    \leq c (\ell_1 + \ell_2 + \ell_3+1)^c \sinh \div{\ell}.
  \end{equation}
  which, by \eqref{eq:bounds_ell_ell_i}, implies
  \begin{equation*}
    \mathrm{Int}[f_1, f_2, f_3](\ell)
    = \O{(\ell+1)^c \sinh \div{\ell} \iint_{[0,\ell]^2} \d \ell_1 \d \ell_3}
    = \O{(\ell+1)^{c+2} \exp \div{\ell}},
  \end{equation*}
  which is our claim.

  The proof the remaining case is the same, now expressing $\ell_3$ in terms of $\ell_1, \ell_2$ and
  $\ell$. Indeed,
  \begin{equation*}
    \frac{\partial \ell_3}{\partial \ell}
    = \frac{\sinh \div{\ell}}{\sinh \div{\ell_3}}
  \end{equation*}
  and hence 
  $\prod_{i=1}^{3} f_i(\ell_i) \sinh \div{\ell_i} \, \frac{\partial
    \ell_3}{\partial \ell}$ satisfies the bound \eqref{e:bound_noterm_3} when $p_1=p_2=0$.
\end{proof}

\begin{rem}
  The fact that we obtain a function in $\FRrem$ in cases \eqref{it:easy_cases_remainder_1} and
  \eqref{it:easy_cases_remainder_3} corresponds to the fact that, if
  $\tilde{f}_1, \tilde{f}_2 \in \FRrem$, then $\tilde{f}_1 \star \tilde{f}_2 \in \FRrem$, as
  observed in the proof of the stability of $\FR$ by convolution (\cref{p:conv}), which we invite
  the reader to read at this stage. Another insight that one can gather from proof is that, if
  $\tilde{f}_1, \tilde{f}_2 \in \FR$ are of respective main terms $\tilde{p}_1, \tilde{p}_2$, then
  the main term of $\tilde{f}_1 \star \tilde{f}_2$ is not only $\tilde{p}_1 \star \tilde{p}_2$, but
  actually contains contributions coming from the remainders terms of $\tilde{f}_1, \tilde{f}_2$.
  This is the reason why \cref{lem:easy_cases_remainder} only holds if \emph{both} $f_1$ and $f_2$
  have no polynomial term, and not when only one of them does. We do expect the contributions where
  $p_1 \neq 0$ whilst $p_2=0$ to participate to the main term of $\mathrm{Int}[f_1, f_2, f_3]$.
\end{rem}

\subsection{Change of variables}

In order to study the integrals $\mathrm{Int}[f_1, f_2, f_3]$ more precisely, it will be helpful to
introduce new variables, which transform the level-set integral \eqref{eq:integral_to_compute_FR}
into a convolution-like integral. We shall use the following new variables, which are represented in
\cref{fig:change_var}:
\begin{itemize}
\item $L_1$ and $L_2$ denote the lengths of the geodesic arcs based at the self-intersection
  of~$\alpha$, going around the first and second boundary components of the pair of pants
  respectively;
\item $u := \cos^2 \div{\theta}$, where $\theta$ denotes the outer angle of the self-intersection of
  $\alpha$.
\end{itemize}
We observe that, in these new coordinates, we always have
$h(\ell_1, \ell_2, \ell_3) = L_1 + L_2$, and hence the level-set
integral is an integral on
$\{(L_1, L_2,u) \in \R_{>0}^2\times (0,1) \, : \, L_1+L_2=\ell\}$, similar to a
convolution. This is evocative of the case of graphs, described in
\cref{sec:stab-conv}. A significant difference with the case of graphs is the
presence of an additional parameter $u \in (0,1)$ that is required to describe
the geometry of the pair of pants; in the following, this quantity will mostly
behave like a free parameter in $(0,1)$ that we will integrate out.

\label{sec:change-variables}
\begin{figure}[h]
  \centering
  \includegraphics[scale=0.5]{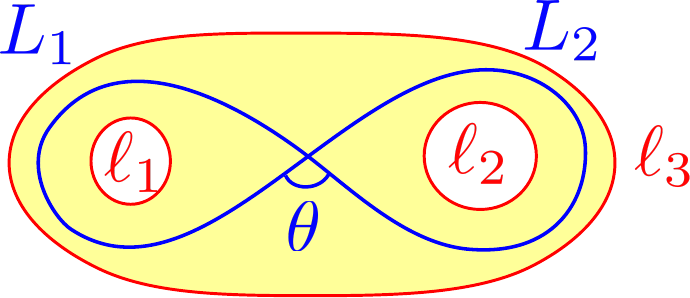}
  \caption{The new variables $(L_1, L_2, u = \cos^2(\theta/2))$ used to describe
    the geometry of the pair of pants.}
  \label{fig:change_var}
\end{figure} 

The following lemma provides an expression for the lengths
$\ell_1, \ell_2, \ell_3$ of the three boundary components of the pair of pants,
in terms of the new variables $(L_1,L_2,u)$.

\begin{lem}
  \label{lem:l123_calc}
  For any $\ell_1, \ell_2, \ell_3 > 0$,
  \begin{numcases}{}
    &  $\cosh \div{\ell_i} = \sqrt{u} \, \cosh \div{L_i}$ \qquad
    for $i \in \{1, 2\}$
    \label{e:l1_calc} \\
    & 
    $\cosh \div{\ell_3} = (1 - u) \cosh \div{L_1+L_2} - u \, \cosh \div{L_1 -
      L_2}.$ \label{e:l3_calc}
  \end{numcases}
\end{lem}

\begin{proof}
  Let $z$ denote the intersection point of the figure eight, and $\alpha_1$
  denote the portion of $\alpha$ going around the first boundary component,
  $b_1$, of $\mathbf{P}$. We draw the common perpendicular of $\alpha_1$ and $b_1$, as
  well as the perpendicular of $b_1$ passing through $z$. When cutting along
  $\alpha_1$ and those two perpendiculars, we obtain a trirectangle, represented
  in \cref{fig:proof_change_coord}. By \cite[Theorem 2.3.1(iii)]{buser1992},
  \begin{equation*}
    \cosh \div{\ell_1} = \sin \div{\pi - \theta} \cosh \div{L_1} = \sqrt{u} \cosh \div{L_1},
  \end{equation*}
  and same goes for $\ell_2$.
  \begin{figure}[h]
    \centering
    \includegraphics[scale=0.5]{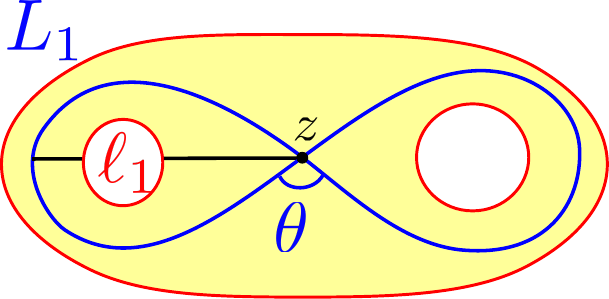}
    \hspace{2cm}
    \includegraphics[scale=0.7]{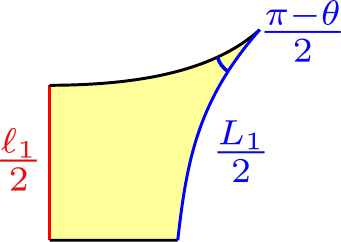}
    \caption{Illustration of the proof of \cref{e:l1_calc}.}
    \label{fig:proof_change_coord}
  \end{figure}

  The formula for $\ell_3$ is obtained directly from the formula for the length
  of the figure eight,
  \begin{equation*}
    \cosh \div{L_1+L_2}
    = 2 \cosh \div{\ell_1} \cosh \div{\ell_2} + \cosh \div{\ell_3}
  \end{equation*}
  together with the fact that, by \eqref{e:l1_calc},
  \begin{equation*}
    2\cosh \div{\ell_1} \cosh \div{\ell_2}
    = 2u \cosh \div{L_1} \cosh \div{L_2} 
  \end{equation*}
  and we have the usual trigonometric formula
  \begin{equation*}
     2 \cosh \div{L_1} \cosh \div{L_2}  =  \cosh \div{L_1+L_2} + \cosh \div{L_1-L_2}.
  \end{equation*}
\end{proof}

This is enough to entirely describe the change of variable
$(\ell_1, \ell_2, \ell_3) \rightarrow (L_1, L_2, u)$, which we do in the following lemma.

\begin{lem}
  \label{lem:ch_var}
  The application $(\ell_1, \ell_2, \ell_3) \mapsto (L_1, L_2, u)$ is a diffeomorphism from the set
  $\R_{>0}^3$ onto the set $\mathfrak{D} \subset \R_{>0}^2 \times (0,1)$ defined by 
  \begin{equation}
    \label{eq:def_E}
    \mathfrak{D} := \left\{ (L_1, L_2, u) \, : \,
      \begin{cases}
        \sqrt u \cosh \div{L_i} > 1  \qquad \text{for } i \in \{1, 2\} \\
        (1-u) \cosh \div{L_1+L_2} > u \cosh \div{L_1-L_2} +1
      \end{cases}
    \right\},
  \end{equation}
  and we have that
  \begin{equation*}
    \d \ell_1 \d \ell_2 \d \ell_3
    = - \dfrac{\sinh \div{L_1+L_2}^2}
    {\sinh \div{\ell_1} \sinh \div{\ell_2} \sinh \div{\ell_3}}
    \d L_1 \d L_2 \d u.
  \end{equation*}
\end{lem}

\begin{proof}
  The first part of the statement is a straightforward consequence of \cref{lem:l123_calc} together
  with the fact that $\cosh : (0, + \infty) \rightarrow (1, + \infty)$ is a diffeomorphism. The only
  thing we need to check is the expression for the Jacobian of the change of variable. Using
  \cref{lem:l123_calc}, we compute the partial derivatives of $\ell_1, \ell_2, \ell_3$ expressed as
  functions of $L_1, L_2, u$, using the slight variation of \eqref{e:l3_calc}
  \begin{equation*}
    \cosh \div{\ell_3} = \cosh \div{L_1+L_2} - 2u \cosh \div{L_1}\cosh \div{L_2}.
  \end{equation*}
  We obtain that we need to compute the determinant
  \begin{align*}
    & \frac{1}{\prod_{i=1}^3\sinh \div{\ell_i}}
    \begin{vmatrix}
      \sqrt{u}\sinh \div{L_1} &
      0 &
      \sinh \div{L_1+L_2} - 2 u \sinh \div{L_1} \cosh \div{L_2} \\
      0 &
      \sqrt{u}\sinh \div{L_2} &
      \sinh \div{L_1+L_2} - 2 u \cosh \div{L_1} \sinh \div{L_2} \\
      \frac{1}{\sqrt{u}} \cosh \div{L_1} &
      \frac{1}{\sqrt{u}} \cosh \div{L_2} &
      - 4 \cosh \div{L_1} \cosh \div{L_2}
    \end{vmatrix}.
  \end{align*}
  We add $2 \sqrt{u} \cosh \div{L_2}$ times the first column and $2 \sqrt{u} \cosh \div{L_1}$ times
  the second column to the third, which removes its negative terms. We then factor out the
  $\sinh \div{L_1+L_2}$ in the third column, and obtain the determinant
  \begin{align*}
    & \frac{\sinh \div{L_1+L_2}}{\prod_{i=1}^3\sinh \div{\ell_i}}
    \begin{vmatrix}
      \sqrt{u}\sinh \div{L_1} &
      0 &
      1  \\
      0 &
      \sqrt{u}\sinh \div{L_2} &
      1  \\
      \frac{1}{\sqrt{u}} \cosh \div{L_1} &
      \frac{1}{\sqrt{u}} \cosh \div{L_2} &
      0
    \end{vmatrix}
    = - \frac{\sinh^2 \div{L_1+L_2}}{\prod_{i=1}^3\sinh \div{\ell_i}} \cdot
  \end{align*}
\end{proof}

\subsection{Technical lemmas about the change of variables}
\label{sec:two-technical-lemmas}

In order to prove \cref{thm:eight}, which we shall do in the next section, we need a few technical
lemmas describing the change of variable $(\ell_1,\ell_2,\ell_3) \rightarrow (L_1, L_2,u)$ more
precisely. For that purpose, we introduce the following notations.

\begin{nota}
  \label{nota:upm_Linf}
  We set $\ell_0 := 4 \argch (\sqrt{2})$.  For $\ell > 0$, we define
  \begin{equation}
    \label{e:condition_u_pm_ell}
    u_-(\ell) := \frac{1}{\cosh^2 \paren*{\frac \ell 4}}
    \quad \text{and} \quad
    u_+(\ell) := 1 - \frac{1}{\cosh^2 \paren*{\frac \ell 4}} 
  \end{equation}
  and, for any $u \in (0,1)$, we define
  \begin{equation*}
    L_-^\infty(u) :=
    \max \paren*{2\argcosh \paren*{\frac 1 {\sqrt{u}}}, \log \paren*{\frac{u}{1-u}}}.
  \end{equation*}
  We shall refer to the following two subsets of $\R_{>0} \times (0,1)$:
  \begin{align*}
    \mathfrak{D}_{1}
    & := \{ (\ell, u) \in (\ell_0, + \infty) \times (0,1) \, : \,
    u_-(\ell) < u < u_+(\ell) \} \\
    \mathfrak{D}_{2}^\infty
    & := \{ (L, u) \in \R_{>0} \times (0,1) \, : \,
    L > L_-^\infty(u) \}.
  \end{align*}
\end{nota}

Let us list a few elementary properties of the quantities defined above.

\begin{lem}
  \label{lem:elem_upm_Linf} \quad
  \begin{enumerate}
  \item Let $\ell >0$. Then, $u_-(\ell) < u_+(\ell)$ if and only if $\ell > \ell_0$.
  \item \label{i:upm}
    We have for any $\ell >0$, $\abs{\log(u_-(\ell))} = \abs{\log(1-u_+(\ell))} \leq \ell/2$.
    Furthermore, for any integer $m \geq 0$, there exists $C_m>0$ such that, for any $\ell > 0$,
    \begin{equation}
      \label{eq:bound_int_log_upm}
      \int_0^{u_-(\ell)} |\log u|^m \d u \leq C_m (\ell+1)^m e^{- \frac \ell 2}.
    \end{equation}
  \item \label{i:Lminus} The function $u \mapsto L_-^{\infty}(u)$ is equivalent to $-\log(u)$ for
    $u$ close to $0$ and $-\log(1-u)$ for $u$ close to $1$.
  \end{enumerate}
\end{lem}

\begin{proof}
  The first and last points are trivial, and so is the first part of (\ref{i:upm}). The rest of the
  second point can be obtained by proving that for any $\epsilon >0$,
    \begin{equation*}
      \int_0^{\epsilon} |\log(u)|^{m} \d u
      = \int_{-\log(\epsilon)}^{+ \infty} v^m e^{-v}  \d v
      = \O{(\abso{\log(\epsilon)}+1)^m \epsilon}
    \end{equation*}
    by iterated integration by parts, and then taking $\epsilon = u_-(\ell)$.
\end{proof}

Let us now provide a finer description of the set $\mathfrak{D}$ defined in \cref{lem:ch_var}.

\begin{lem}
  \label{lem:E}
  For any $\ell > 0$, any $u \in (0, 1)$, any $L \in (0, \frac \ell 2)$,
  \begin{equation*}
    (L, \ell - L, u) \in \mathfrak{D}
    \quad \Leftrightarrow \quad
    \begin{cases}
      (\ell, u) \in \mathfrak{D}_1 \\
      L > L_-(\ell, u)
    \end{cases}
  \end{equation*}
  where the function $L_- : \mathfrak{D}_1 \rightarrow (0, \frac \ell 2)$ satisfies, for a
  constant $C>0$,
  \begin{equation}
    \label{eq:delta_Lmin}
    \forall (\ell, u) \in \mathfrak{D}_1, \quad
    0 \leq L_-(\ell,u) - L_-^\infty(u) \leq C \ell \frac{e^{- \frac \ell 2}}{1-u} \cdot
  \end{equation}
\end{lem}

Note that, here, we are only describing on the part of $\mathfrak{D}$ for which $L_1 < L_2$ (which is
equivalent to $L_1 < \ell/2$ since $L_1+L_2=\ell$). Since $L_1$ and $L_2$ play symmetric roles, this
allows us to describe the entirety of $\mathfrak{D}$.

\begin{proof}
  Let us consider a $\ell > 0$, $0 < u < 1$, and $0 < L < \frac{\ell}{2}$.

  The first condition defining $\mathfrak{D}$ is $\sqrt u \cosh \div{L} > 1$. The existence of number
  $L$ in $(0, \frac \ell 2)$ satisfying this inequality is equivalent to the following condition on $u$
  and $\ell$:
  \begin{equation*}
    \sqrt u \cosh \paren*{\frac \ell 4} > 1 \quad \Leftrightarrow \quad
    u > u_-(\ell).
  \end{equation*}
  Then, $L$ satisfies the condition if and only if 
  \begin{equation*}
    L > L^1_{-}(u) := 2\argcosh \left( \frac{1}{\sqrt u} \right).
  \end{equation*}
  
  Now, because we assume that $L < \frac \ell 2$, we have that $\ell - L > L$, and hence the second
  condition of $\mathfrak{D}$ is automatically satisfied once the first is.

  Let us now move on the last condition defining $\mathfrak{D}$,
  \begin{equation*}
    (1-u) \cosh \div{\ell} - u \cosh \paren*{\frac \ell 2 - L} > 1.
  \end{equation*}
  The left hand side of this equation is an increasing function of $L \in (0, \frac \ell 2)$, and is
  therefore maximal at $\ell/2$. As a consequence, there exists a value $L$ in $(0, \frac \ell 2)$
  so that it exceeds $1$ if and only if
  \begin{equation*}
    (1-u) \cosh \div{\ell} - u  > 1 \quad \Leftrightarrow \quad
    u < u_+(\ell).
  \end{equation*}
  In that case, $L$ satisfies the third condition of $\mathfrak{D}$ if and only if
  \begin{equation*}
    L > L^3_{-}(\ell, u)
    := \frac \ell 2 - \argcosh \paren*{\frac{1-u}{u} \cosh \div \ell - \frac{1}{u}}.
  \end{equation*}

  Finally, we observe that we need to have $u_-(\ell) < u_+(\ell)$ for all three conditions of
  $\mathfrak{D}$ to be satisfied together, which we saw in \cref{lem:elem_upm_Linf} is equivalent to
  $\ell > \ell_0$. Then, we have the claimed result with
  \begin{equation*}
    L_-(\ell,u)
    := \max \paren*{L^1_{-}(u), L^3_{-}(\ell, u)}.
  \end{equation*}
  
  To prove the bound on $L_- - L_-^\infty$, we simply need to observe that for $x>1$,
  \begin{equation*}
    \argcosh (x) = \log \paren*{x + \sqrt{x^2-1}} = \log \paren*{2x + \O{1}}
  \end{equation*}
  and hence
  \begin{align*}
    L^3_{-}(\ell, u) 
     = \frac{\ell}{2} - \log \paren*{\frac{1-u}{u} \, e^{\frac \ell 2} + \O{\frac{1}{u}}} 
     = \log \paren*{\frac{u}{1-u}} - \log(1+A(\ell,u))
  \end{align*}
  where the quantity $A(\ell,u)>-1$ satisfies $A(\ell,u) = \O{e^{- \frac \ell2} /(1-u)}$.  There
  exists a constant $c>0$ such that, for any real number $A \geq -1/2$, $|\log(1+A)| \leq c
  |A|$. This is enough to conclude whenever $A(\ell,u) \geq -1/2$, since $\ell > \ell_0$. Otherwise,
  we rather use the trivial bound
  \begin{equation*}
    \abso*{L^3_{-}(\ell, u)  - \log \left (\frac{u}{1-u} \right)}
    \leq \frac{\ell}{2} + |\log(u_-(\ell))| + |\log(1-u_+(\ell))| \leq 2 \ell.
  \end{equation*}
  This leads to the claimed estimate, because if $A(\ell,u) < -1/2$, then $e^{- \frac \ell2} /(1-u)$
  is bounded away from~$0$.
\end{proof}

The next lemma is an expansion of the functions $\ell_i(L_1,L_2,u)$ for $1 \leq i \leq 3$.

\begin{lem}
  \label{lem:expansion_ell}
  There exists functions $r_0, r : \mathfrak{D}_2^\infty \rightarrow \R$ and a constant $C>0$ such
  that, for any $(L_1, L_2, u) \in \mathfrak{D}$,
  \begin{equation*}
    \begin{cases}
      \ell_1(L_1, L_2, u)  = L_1 + \log(u) + r_0(L_1,u) \\
      \ell_2(L_1, L_2, u)  = L_2 + \log(u) + r_0(L_2,u) \\
      \ell_3(L_1, L_2, u)  = L_1 + L_2 + 2 \log(1-u) 
                            + r(L_1,u) + r(L_2,u) + \mathrm{err}(L_1,L_2,u)
    \end{cases}
  \end{equation*}
  where:
  \begin{enumerate}
  \item for any $(L_1, L_2, u) \in \mathfrak{D}$,
    \begin{equation*}
      |\mathrm{err}(L_1,L_2,u)| \leq C (L_1+L_2) \, e^{- \frac{\ell_3(L_1, L_2, u)}{2}};
    \end{equation*}
  \item for any $(\ell,u) \in \mathfrak{D}_1$, any $L > L_-^\infty(\ell,u)$,
    \begin{equation}
      \label{eq:expansion_ell_bound}
      |r_0(L,u)| + |r(L,u)| \leq C \, \ell \, \frac{e^{-L}}{u(1-u)};
    \end{equation}
  \item $r_0$ is bounded on $\mathfrak{D}_2^\infty$, and for any integer $k \geq 0$, any
    $(\ell, u) \in \mathfrak{D}_1$,
    \begin{equation}
      \label{eq:expansion_ell_bound_tail}
      \int_{L_-^\infty(u)}^{L_-(\ell, u)} |r(L,u)|^k \d L  = \O[k]{\ell^{k+1}
        \frac{e^{-\frac \ell 2}}{1-u}}
    \end{equation}
    and for any $k,j \geq 0$, any $u \in (0,1)$,
    \begin{equation}
      \label{eq:bound_r_int_all}
      \int_{L_-^\infty(u)}^{+ \infty} L^{j} |r(L,u)|^k \d L = \O[k]{(|\log(u)| + |\log(1-u)|)^j}.
    \end{equation}
  \end{enumerate}
\end{lem}

\begin{proof}[Proof of \cref{lem:expansion_ell}]
  Let us first prove the expansion for $\ell_1(L_1, L_2,u)$. First, we observe that the expression
  \eqref{e:l1_calc} allows to write $\ell_1$ as a function of $L_1$ and $u$ (and hence independent
  of $L_2$), which is well-defined as soon as $L_1 > L_-^\infty(u)$. Then, we write
  $\exp \div{\ell_1} = 2 \cosh \div{\ell_1} / (1+e^{-\ell_1})$, which yields
  \begin{align*}
    \ell_1(L_1, u)
    & = 2 \log \paren*{2 \cosh \div{\ell_1}} + \O{e^{- \ell_1}}  
      = 2 \log \paren*{\sqrt{u} (e^{\frac{L_1}{2}} + e^{-\frac{L_1}{2}})} + \O{e^{- \ell_1}} \\
    & = L_1 + \log(u) + \O{e^{-L_1} + e^{- \ell_1}}.
  \end{align*}
  This implies the claimed estimate taking $r_0(L,u) := \ell_1(L,u) - L - \log(u)$. Indeed, this
  function is defined for any $L > L_-^\infty(u)$, and is $\O{e^{-L} + e^{-\ell_1(L,u)}} = \O{1}$ by
  the bound above. We conclude thanks to the fact that $e^{- \ell_1(L,u)} = \O{e^{-L}/u}$ as soon as
  $L > L_-^\infty(u)$.

  The formula for $\ell_2$ is exactly the same, so we now move on to $\ell_3$. We shall pay
  particular attention to the fact that $\ell_3$ depends on both $L_1$ and $L_2$. First, we write
  similarly as before
  \begin{align*}
     \ell_3(L_1, L_2, u) 
    & = 2 \log \paren*{2 \cosh \div{\ell_3}} + \O{e^{- \ell_3}} \\
    & \hspace{-4pt} \overset{\eqref{e:l3_calc}}{=}
      2 \log \paren*{(1-u) (e^{\frac{L_1+L_2}{2}} + e^{-\frac{L_1+L_2}{2}})
      - u (e^{\frac{L_2-L_1}{2}} + e^{\frac{L_1-L_2}{2}})} + \O{e^{- \ell_3}} \\
    & = L_1 + L_2 + 2 \log(1-u)
      + R_u(L_1, L_2, u) 
      + \O{e^{- \ell_3}}
  \end{align*}
  where
  \begin{equation*}
    R_u(L_1, L_2, u) := 2 \log \paren*{1 + e^{- L_1 - L_2} - \frac{u}{1-u} (e^{-L_1} + e^{-L_2})}.
  \end{equation*}
  In order to write this term in the desired form, we rewrite the argument of the $\log$ as
  \begin{equation*}
     \paren*{1 - \frac{u}{1-u} e^{-L_1}}\paren*{1 - \frac{u}{1-u} e^{-L_2}} +
    \O{\frac{e^{-L_1-L_2}}{(1-u)^2}}. 
  \end{equation*}
  We observe that,
  by the third condition of the definition of $\mathfrak{D}$,
  \begin{equation}
    \label{eq:def_E_3_use}
    (1-u) \, e^{\frac{L_1+L_2}{2}} - u (e^{\frac{L_2-L_1}{2}} + e^{\frac{L_1-L_2}{2}}) > 2 - (1-u) > 1
  \end{equation}
  and hence 
  \begin{equation}
    \label{eq:def_E_3_use_2}
    \paren*{1- \frac{u}{1-u} e^{-L_1}} \paren*{1- \frac{u}{1-u} e^{-L_2}}
    > 1 - \frac{u}{1-u} (e^{-L_1} + e^{-L_2})
    \overset{\eqref{eq:def_E_3_use}}{>} \frac{e^{- \frac{L_1+L_2}{2}}}{1-u} \cdot
  \end{equation}
  This implies that, as soon as $(L_1, L_2, u) \in \mathfrak{D}$,
  \begin{equation}
    \label{eq:write_Ru_r}
    R_u(L_1, L_2, u)
    = r(L_1,u) + r(L_2,u)
    + \log \paren*{1 + \O{\frac{e^{- \frac{L_1+L_2}{2}}}{1-u}}}
  \end{equation}
  for the function $r(L,u) := \log \paren*{1 - \frac{u}{1-u} e^{- L}}$, well-defined for any
  $L>L_-^\infty(u)$.

  We proceed as in the proof of \eqref{eq:delta_Lmin} to prove the bound
  \eqref{eq:expansion_ell_bound} on $r$. Indeed,
  \begin{itemize}
  \item if $ue^{-L}/(1-u) \leq 1/2$, then $r(L,u) = \O{ue^{-L}/(1-u)}$ by the asymptotic behaviour
    of $\log$ near $1$;
  \item otherwise, we rather observe that, by \eqref{eq:def_E_3_use_2}, we have
    $|r(L,u)| \leq \frac{\ell}{2} + |\log(1-u)|$, which is $\leq \ell$ by
    \cref{lem:elem_upm_Linf}.(\ref{i:upm}).
  \end{itemize}
  The integral bound~\eqref{eq:expansion_ell_bound_tail} is obtained by studying for $\epsilon > 0$
  the integral
  \begin{equation*}
    \int_0^\epsilon |\log(1-e^{-x})|^k \d x = \O[k]{\epsilon \, (|\log(1-e^{-\epsilon})| +1)^k}
  \end{equation*}
  by integration by parts, taking $\epsilon := L_-(\ell,u) - \log \paren*{\frac{u}{1-u}}$, and using
  the bound \eqref{eq:delta_Lmin}. The other integral is straightforward.

  Similarly, we can prove that the logarithmic term in \eqref{eq:write_Ru_r} is $\O{\ell \,
    e^{- \frac{L_1+L_2}{2}}/(1-u)}$ for $(L_1,L_2,u) \in \mathfrak{D}$. 
  This allows us to conclude the expansion of $\ell_3$, because
  \begin{equation*}
    \cosh \div{\ell_3(L_1,L_2,u)} \leq (1-u) \cosh \div{L_1+L_2}
  \end{equation*}
  and hence $e^{-\frac{L_1+L_2}{2}} / (1-u)= \O{e^{- \frac{\ell_3(L_1,L_2,u)}{2}}}$.  
\end{proof}

We will actually use the following corollary of \cref{lem:expansion_ell}, which is obtained by
carefully taking powers and linear combinations of the statement.

\begin{cor}
  \label{cor:expansion_ell}
  Let $(f_i)_{1 \leq i \leq 3}$ be functions satisfying \eqref{eq:expansion_assumption_fi}. Let
  \begin{equation*}
    Q_u(L_1,L_2) := p_1(L_1+\log(u)) \, p_2(L_2+\log(u)) \, p_3(L_1+L_2+2\log(1-u))
  \end{equation*}
  where $(p_i)_{1 \leq i \leq 3}$ are the respective polynomials of $(f_i)_{1 \leq i \leq 3}$.
  There exists an integer~$K$ and a family of functions
  $\tilde{r}_j : \mathfrak{D}_2^\infty \rightarrow \R$, for $0 \leq j \leq K$, such that for any
  $(L_1,L_2,u) \in \mathfrak{D}$,
  \begin{align*}
     \prod_{i=1}^{3} f_i(\ell_i(L_1,L_2,u)) 
    & =  \, Q_u(L_1,L_2) + \sum_{0 \leq j \leq K}
      \big(L_1^j \tilde{r}_j(L_2,u) + L_2^j \tilde{r}_j(L_1,u)\big) \\
    & \hphantom{= \,}  + \O{(L_1+L_2)^K \left(e^{- \frac{\ell_3(L_1,L_2,u)}{2}}
      + \frac{e^{-\frac{L_1+L_2}{2}}}{u(1-u)}\right)}
  \end{align*}
  and for any integer $0 \leq j \leq K$, any $(\ell,u) \in \mathfrak{D}_1$, any $L > L_-^\infty(\ell,u)$,
  \begin{equation}
    \label{eq:bound_cor_rem_one_var}
    \tilde{r}_j(L,u) = \O{\ell^K \frac{e^{-L}}{u(1-u)}}. 
  \end{equation}
  Furthermore, for all $j$, there exists an integer $0 \leq k \leq K$ such that
  $\tilde{r}_j = \O{|r|^k}$ on $\mathfrak{D}_2^\infty$, where $r$ is the function from
  \cref{lem:expansion_ell}.
\end{cor}

\begin{proof}
  We write
  \begin{equation*}
    \prod_{i=1}^{3} f_i(\ell_i(L_1,L_2,u))
    = \prod_{i=1}^{3} \big( p_i(\ell_i(L_1,L_2,u)) + r_i(\ell_i(L_1,L_2,u)) \big),
  \end{equation*}
  replace $\ell_i$ by its expansion from \cref{lem:expansion_ell} in the polynomial terms, and
  expand everything. We obtain four types of terms.
  \begin{itemize}
  \item On the one hand, some terms are purely polynomial in $L_1$ and $L_2$. We group these terms
    together; they form the term $Q_u$.
  \item Then, some terms are a product of $L_2^j$ for a $j \geq 0$ and a function
    $\tilde{r}_j(L_1,u)$ containing no polynomial term in $L_1$. More precisely, contributions to
    this term are products of powers of contributions
    $r_1(\ell_1(L_1,u)) = \O{e^{-\ell_1}} = \O{e^{-L_1}/u}$, as well as the remainder terms
    $r_0(L_1,u)$ and $r(L_1,u)$ from \cref{lem:expansion_ell}. We prove that $\tilde{r}_j$ satisfies
    the claimed bounds using \eqref{eq:expansion_ell_bound}, the boundedness of $r_0$ and $r_1$, and
    the fact that $e^{-L}/(u(1-u)) = \O{1}$ for $L > L_-^\infty(u)$.
  \item There is also the term containing other crossed contributions, which is equal to
    $L_1^j \tilde{r}_j(L_2,u)$ by symmetry of the roles of $L_1$ and $L_2$.
  \item And then, all other terms contain at least a factor $\O{\ell^Ke^{-\ell_3/2}}$, or two factors
    decaying in $L_1$ and $L_2$ respectively, which gives us the two bounds on the error term.
  \end{itemize}
\end{proof}

\subsection{Proof of \cref{thm:eight}}

We are now finally ready to prove \cref{thm:eight}. This will be achieved using a reasoning very
similar to the proof of \cref{p:conv}, where we showed that the class of functions $\FR$ is stable
by convolution.

\begin{proof}[Proof of \cref{thm:eight}.]
  We note that the property of being Friedman--Ramanujan is an asymptotic property, and hence we can
  restrict ourselves to studying the case that $\ell > \ell_0$.  As shown in Lemmas
  \ref{lem:expression_before_change_var} and \ref{lem:easy_cases_remainder}, the problem reduces to
  showing that, for any $f_1, f_2, f_3$ satisfying \eqref{eq:expansion_assumption_fi},
  the function
  \begin{equation*}
    \mathrm{Int}[f_1, f_2, f_3] : \ell \mapsto \iint_{h(\ell_1, \ell_2, \ell_3) = \ell}
    \prod_{i=1}^3 f_i(\ell_i) \sinh \div{\ell_i}
    \frac{\d \ell_1 \d \ell_2 \d \ell_3}{\d \ell}
  \end{equation*}
  is a Friedman--Ramanujan function. We will now omit the mention of $(f_i)_i$ and denote the
  integral above as $\mathrm{Int}$, to simplify notations.
  
  Let us decompose this integral in two integrals $\mathrm{Int}^{1,2}$, depending on whether
  $\ell_1 \leq \ell_2$ or $\ell_1 \geq \ell_2$. By symmetry of the roles of $\ell_1$ and $\ell_2$,
  if we prove the result for any $(f_i)_i$ for $\mathrm{Int}^1$, the result follows for
  $\mathrm{Int}^2$, and hence for $\mathrm{Int}$. We can therefore assume without loss of generality
  that the integral runs over the set of parameters such that $\ell_1 \leq \ell_2$, and in
  particular $\ell_1 \leq \frac \ell 2$ and $\ell_2 \geq \frac \ell 2$.
  
  Let us perform the change of variable $(\ell_1, \ell_2, \ell_3) \mapsto (L_1, L_2, u)$, and use
  Lemmas \ref{lem:ch_var} and \ref{lem:E} to write
  \begin{equation*}
    \mathrm{Int}^1(\ell)
    = \sinh^2 \div{\ell}
    \int_{u_-(\ell)}^{u_+(\ell)} \int_{L_{-}(\ell, u)}^{\frac \ell 2}
    \prod_{i=1}^{3} f_i(\ell_i(L,\ell -L,u)) \d L \d u.
  \end{equation*}
  We now use \cref{cor:expansion_ell} to express the integrand, and examine the various
  contributions appearing successively.
  \begin{itemize}
  \item We first examine the term coming from the remainder decaying in $e^{-
      \frac{\ell_3}{2}}$, which is bounded by a multiple constant of 
    \begin{equation*}
      \ell^K \sinh^2 \div{\ell} \int_{u_-(\ell)}^{u_+(\ell)}
      \int_{L_-(\ell,u)}^{\frac \ell 2} e^{- \frac{\ell_3(L,\ell-L,u)}{2}} \d L \d u.
    \end{equation*}
    For this term, actually, we return to the old variables $\ell_1, \ell_2, \ell_3$, and obtain
    that our contribution is bounded by
    $\ell^K \, \mathrm{Int}[1,1,e^{- \frac{\cdot}{2}}](\ell)$. By
    \cref{lem:easy_cases_remainder}, $\mathrm{Int}[1,1,e^{- \frac{\cdot}{2}}]$ is a function in
    $\FRrem$ because the functions $(1, 1, e^{- \frac{\cdot}{2}})$ satisfy
    \eqref{eq:expansion_assumption_fi} with $p_3=0$. Hence, this contribution to $\mathrm{Int}^1$ is
    an element of $\mathcal{R} \subset \mathcal{F}$.
  \item The other part of the remainder of \cref{cor:expansion_ell} is bounded by a multiple
    constant of
    \begin{align*}
      &  \ell^{K} \sinh^2 \div{\ell}
        \int_{u_-(\ell)}^{u_+(\ell)} \int_{L_-(\ell,u)}^{\frac \ell 2}
        \frac{e^{- \frac \ell 2}}{u(1-u)} \d L \d u \\
      & \leq \ell^{K+1} e^{\frac \ell 2} \,
        \left[ |\log(u)| + |\log(1-u)| \vphantom{\int} \right]_{u_-(\ell)}^{u_+(\ell)}
        = \O{\ell^{K+2} e^{\frac \ell 2}}.
    \end{align*}
    As a consequence, this contribution is also an element of~$\mathcal{R}$.
  \item We now observe that all of the terms $L_1^j \tilde{r}_j(L_2,u)$ for $0 \leq j \leq K$ yield
    contributions in $\FRrem$, because as soon as $L_1 \leq L_2$,
    \begin{equation*}
       L_1^j\tilde{r}_j(L_2,u) = \O{\ell^{j+K} \frac{e^{- L_2}}{u(1-u)}}
       = \O{\ell^{j+K} \frac{e^{- \frac{L_1+L_2}{2}}}{u(1-u)}}
    \end{equation*}
    which is covered by the previous case.
  \item Let us now examine what we expect to be the ``leading contribution'', coming from $Q_u$. We
    need to estimate
    \begin{equation*}
      \sinh^2 \div{\ell}
      \int_{u_-(\ell)}^{u_+(\ell)} \int_{L_-(\ell, u)}^{\frac \ell 2} Q_u(L, \ell - L) \d L \d u.
    \end{equation*}
    We integrate the polynomial, and expand the powers, to express this quantity as a linear
    combination of integrals of the form
    \begin{equation}
      \label{e:case_Qu_1}
      \ell^{j_0} \sinh^2 \div{\ell} 
      \int_{u_-(\ell)}^{u_+(\ell)} \log(u)^{j_1} \log(1-u)^{j_2} L_-(\ell, u)^{j_3} \d u.
    \end{equation}
    Let us replace $L_-(\ell,u)$ by its approximation $L_-^\infty(u)$ in this expression. We
    obtain the quantity
    \begin{equation}
      \label{e:case_Qu_2}
      \ell^{j_0} \sinh^2 \div{\ell} 
      \int_{u_-(\ell)}^{u_+(\ell)} \log(u)^{j_1} \log(1-u)^{j_2} L_-^{\infty}(u)^{j_3} \d u.
    \end{equation}
    By \cref{lem:elem_upm_Linf}(\ref{i:Lminus}), the function
    $(0,1) \ni u \mapsto \log(u)^{j_1} \log(1-u)^{j_2} L_-^{\infty}(u)^{j_3}$ is integrable. Let
    $C_{\vec{\j}}$ denotes this integral. The tails of $C_{\vec{\j}}$ for $u < u_-(\ell)$ and
    $u > 1 - u_+(\ell)$ are $\O{\ell^{j_1+j_2+j_3} e^{- \frac{\ell}{2}}}$ by
    \cref{lem:elem_upm_Linf}(\ref{i:upm}).  As a consequence, for $|\vec{\j}| := j_0+j_1+j_2+j_3$,
    \cref{e:case_Qu_2} is equal to
    \begin{equation*}
      C_{\vec{\j}} \, \ell^{j_0} \, \sinh^2 \div{\ell} + \O{\ell^{|\vec{\j}|} e^{\frac \ell 2}}
    \end{equation*}
    which means it is an element of $\FR$.  We are therefore left to study the error made when
    replacing $L_-(\ell,u)$ by $L_-^\infty(u)$ in \cref{e:case_Qu_1}. There is nothing to do if
    $j_3=0$, and if $j_3 \geq 1$, by \eqref{eq:delta_Lmin},
    \begin{equation*}
      |L_-(\ell,u)^{j_3} - L_-^\infty(u)^{j_3}|
      = \O{\ell^{j_3-1}|L_-(\ell,u) - L_-^\infty(u)|}
      = \O{\ell^{j_3} \frac{e^{-\frac{\ell}{2}}}{1-u}}
    \end{equation*}
    and hence the error term is bounded by a multiple constant of
    \begin{equation*}
      \ell^{|\vec{\j}|} \, e^{\frac \ell 2} \, \abso*{\int_{u_-(\ell)}^{u_+(\ell)} 
        \frac{\d u}{1-u}} = \O{\ell^{|\vec{\j}|+1} e^{\frac \ell 2}}.
    \end{equation*}
    This means that the difference between \cref{e:case_Qu_1} and \cref{e:case_Qu_2} lies in
    $\FRrem \subset \FR$, which allows us to conclude.
  \item Last but not least, we examine the contributions of the form
    \begin{equation*}
      \sinh^2 \div{\ell}
      \int_{u_-(\ell)}^{u_+(\ell)} \int_{L_-(\ell, u)}^{\frac \ell 2}
      (\ell - L)^j \tilde{r}_j(L,u) \d L \d u
    \end{equation*}
    for $j \in \{ 0, \ldots, K \}$.
    We expand the exponents, so that we are left with a linear combination of integrals of the form
    \begin{equation}
      \label{eq:term_p10_expr}
      \ell^{j_0} \sinh^2 \div{\ell}
      \int_{u_-(\ell)}^{u_+(\ell)} \int_{L_-(\ell, u)}^{\frac \ell 2}
      L^{j_1} \tilde{r}_j(L,u) \d L \d u.
    \end{equation}
    Let us replace the integration over $L \in (L_-(\ell, u), \frac \ell 2)$ by an integration on a
    set independent of $\ell$, the interval $(L_-(u), + \infty)$. 
    \begin{itemize}
    \item First, using \eqref{eq:bound_cor_rem_one_var}, we prove that the tail $L > \frac \ell 2$
      is an element of $\FRrem$;
    \item For the other tail, where $L_-^\infty(u) < L < L_-(\ell,u)$, we use
      \cref{cor:expansion_ell} to bound $\tilde{r}_j$ by a power $r^{k}$ of the function $r$ from
      \cref{lem:expansion_ell}. Then,~\eqref{eq:expansion_ell_bound_tail} allows us to prove that this
      contribution is an element of $\FRrem$. 
    \end{itemize}
    As a consequence, \eqref{eq:term_p10_expr} is equal, modulo elements of $\FRrem$, to
    \begin{align*}
      \ell^{j_0} \sinh^2 \div{\ell}
      \int_{u_-(\ell)}^{u_+(\ell)} \int_{L_-^\infty(u)}^{+ \infty}
      L^{j_1} \tilde{r}_j(L,u) \d L \d u.
    \end{align*}
    Finally, we replace the integration on $(u_-(\ell), u_+(\ell))$ by an integration on $(0,1)$. We
    use \eqref{eq:bound_r_int_all} and \eqref{eq:bound_int_log_upm} to deduce that the cost of this
    substitution is an element of~$\FRrem$.  As a consequence, \eqref{eq:term_p10_expr} is equal
    modulo $\FRrem$ to
     \begin{equation*}
       \ell^{j_0} \sinh^2 \div{\ell}
      \int_{0}^{1} \int_{L_-^\infty(u)}^{+ \infty}
      L^{j_1} \tilde{r}_j(L,u) \d L \d u 
    \end{equation*}
    which is an element of $\FR$, and we have therefore proved our claim. 
  \end{itemize}  
\end{proof}


\section{Extension to any geodesic filling a surface of Euler characteristic $-1$}
\label{sec:other-geodesics}

We shall now extend the result of the previous section to any loop filling a
pair of pants or a once-holed torus. More precisely, we prove the following.

{
\begin{thm}
  \label{thm:all_chi_-1}
  For any local topological type $\type$ of absolute Euler characteristic $1$, the functions
  $(f_k^\type)_{k \geq 1}$ are Friedman--Ramanujan in the weak sense. More precisely, for any
  $k \geq 1$, there exists a constant $c_k \geq 0$ such that $f_k^\type \in \cF_w^{c_k,c_k}$ and
  $\|f_k^\type\|_{\cF^{c_k,c_k}}^w \leq c_k$.
\end{thm}}

\begin{rem}
  The statement we prove in \cref{thm:all_chi_-1} is stronger than our statement for
  \ref{chal:FR_type}, because we prove an additional uniformity with respect to the local
    type. This uniformity is essential for applications to spectral gap problems, as illustrated in
    the applications in \cite{anantharaman_29} and the companion article.
\end{rem}

The plan of the proof is the following.
\begin{itemize}
\item In \cref{sec:double-fill-curv}, we prove the result for local topological types satisfying an
  additional hypothesis, ``double-filling'' loops.
\item In \cref{sec:class-fill-curv}, we prove that all local topological types $\type$ with
  $\chi(\type)=1$ are either double-filling, a figure-eight, or elements of a class of local
  topological types called ``one-sided iterated eights''.
\item It then follows that the only remaining case to examine is the case of one-sided iterated
  eights, which is done in \cref{sec:one-sided-iterated}.
\end{itemize}

\subsection{The case of double-filling loops}
\label{sec:double-fill-curv}

We make the following definition.

{
\begin{defa}\label{d:double_fill} Let $\mathbf{T} = \eqc{\Sf,\mathbf{c}}$ be a local topological type.
  \begin{enumerate}
  \item A \emph{simple portion} of $\mathbf{c}$ is a maximal open sub-segment of $\mathbf{c}$ which
    does not contain any self-intersection point of $\mathbf{c}$.
  \item A simple portion is said to be \emph{shielded} if it belongs to the boundary of a contractible
    component of $\mathbf{S}\setminus \mathbf{c}$, and \emph{unshielded} otherwise.
  \item We say $\mathbf{c}$ is  \emph{double-filling} if all simple portions are
    shielded.
  \end{enumerate}
\end{defa}
}

In order to determine if a loop is double-filling, we highlight the boundary components that are
contractible in its complement, as done in \cref{fig:double_fill_exp}. The loop is double-filling if
and only if the whole loop is bordered by highlighted boundary components.
\begin{figure}[h]
  \includegraphics[scale=0.15]{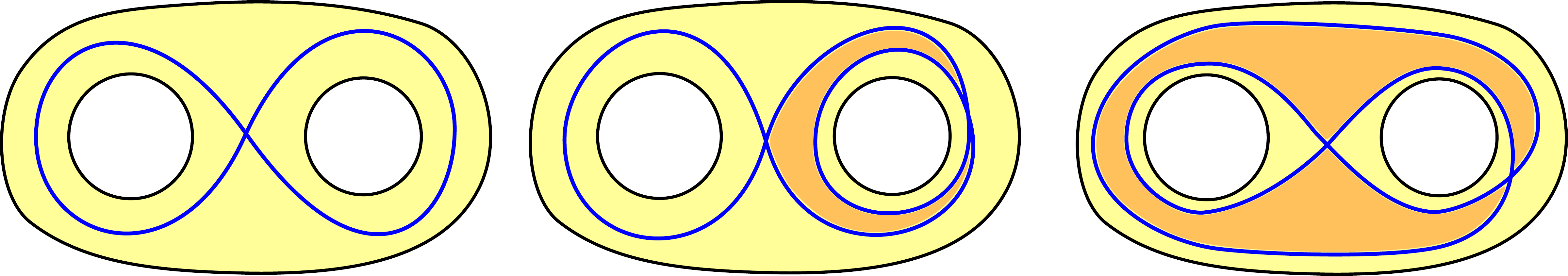}
  \caption{Three loops (in blue) filling a pair of pants. We highlighted the disks in the complement
    of the loop.  The rightmost loop is double-filling while the two left ones are not.}
    \label{fig:double_fill_exp}
\end{figure}

\begin{rem}
  This notion does not depend on the choice of the loop $\curve$ in a free-homotopy class. Indeed,
  we saw in the proof of \cref{lem:equivalent} that two loops homotopic in $\Sf$ can be obtained from
  one another by a sequence of third Reidemeister moves. Examining \cref{fig:reid_move_curves}, we
  can see that if a loop is double-filling according to our definition, then it remains so after a
  third Reidemeister move.
\end{rem}

The double-filling hypothesis is a setting in which \cref{thm:all_chi_-1} is easy to prove.

\begin{prp}
  \label{prp:con_double_filling}
  For any $k \geq 1$, there exists a constant $c_k \geq 0$ satisfying the following. For any local
  topological type $\type = \eqc{\Sf, \curve}$ with $\chi(\type) = 1$, if $\curve$ is a
  double-filling loop, then for any $L > 0$, 
  \begin{equation*}
    \int_{0}^{L} \abso{f_k^\type(\ell)} \d \ell
    \leq c_k (L+1)^{c_k-1} e^{L/2}.
  \end{equation*}
  In other words, $f_k^\type$ belongs in $\cR_w^{c_k} = \cF_w^{0,c_k}$, and its norm satisfies
  $\|f_k^\type\|_{\cF^{0,c_k}}^w \leq c_k$.
\end{prp}

The key element that allows to prove \cref{prp:con_double_filling} is the following lemma.

\begin{lem}
  \label{lem:double_fill_ineq}
  Let $Y$ be a bordered hyperbolic surface, and $\curve$ be a double-filling
  geodesic on~$Y$.  Then, $\ell_Y(\partial Y) \leq \ell_Y(\curve)$.
\end{lem}

This is an improvement of \cref{lem:length_fill}, which states that
$\ell_Y(\partial Y) \leq 2 \ell_Y(\curve)$ whenever~$\curve$ fills $Y$.  The double-filling
hypothesis allows us to remove the factor $2$ in this length inequality, which comes from the fact
that the regular neighbourhood of $\curve$ has length approximately~$2 \ell_Y(\curve)$. We first
prove the lemma.

\begin{proof}[Proof of \cref{lem:double_fill_ineq}]
  Let $\epsilon, \delta > 0$, and $\mathcal{N} = \mathcal{N}_\delta(\curve)$ be a $\delta$-regular
  neighbourhood of~$\curve$. By definition of a filling loop, each boundary component $b$ of $Y$ is
  homotopic to a non-contractible boundary components $b({\mathcal{N}})$ of $\mathcal{N}$. By
  minimality of the length of the geodesic representative in a homotopy class,
  $\ell_Y(b) \leq \ell_Y(b({\mathcal{N}}))$ for all $b$.  Thanks to the double-filling hypothesis,
  each $b({\mathcal{N}})$ borders a disjoint simple portion of the loop $\curve$, and therefore we
  can pick $\delta$ to be small enough so that
  \begin{equation*}
    \ell_Y(\partial Y) \leq  \sum_{b} \ell_Y(b({\mathcal{N}}))
    \leq \ell_Y(\curve) + \epsilon.
  \end{equation*}
  This implies the claim taking $\epsilon \rightarrow 0$.
\end{proof}

We now proceed to the proof of \cref{prp:con_double_filling}.

\begin{proof}
  Let us first consider the case of a double-filling loop $\curve$ on the pair of pants
  $\mathbf{P}$, and $\type := \eqc{\mathbf{P},\curve}$. Let $k \geq 1$ be an integer. {For
    $L>0$, we apply \cref{rem:proof_asymptotic} to the test function
    $\1{[0,L]}(\ell) \, \mathrm{sign}(f_k^\type(\ell))$ and obtain, by the triangle inequality,
  \begin{equation}
    \label{eq:apply_asymp_proof_doublefill}
    \int_{0}^{L} |f_k^\type(\ell)| \d \ell
    \leq \int_{\R_{>0}^3}
    \1{[0,L]}(h_\curve(\ell_1,\ell_2,\ell_3)) \,
    |\psi_k^{\mathbf{P}}(\ell_1, \ell_2, \ell_3)| \,
    \d \ell_1 \d \ell_2 \d \ell_3
  \end{equation}
where $\psi_k^{\mathbf{P}}$ is the $k$-th term of the asymptotic expansion of the
function $\phi_g^{\mathbf{P}}$, expressed in \cref{exa:int_pop}, and
$h_\curve : \R_{>0}^3 \rightarrow \R$ is the function which, to the
lengths of the three boundary components of $\mathbf{P}$, associates the length of the
geodesic $\curve$.
We now use the  double-filling hypothesis, and \cref{lem:double_fill_ineq}, which imply
\begin{equation*}
    \int_{0}^{L} |f_k^\type(\ell)| \d \ell
    \leq \int_{\R_{>0}^3}
    \1{[0,L]}(\ell_1+\ell_2+\ell_3)) \,
    |\psi_k^{\mathbf{P}}(\ell_1, \ell_2, \ell_3)| \,
    \d \ell_1 \d \ell_2 \d \ell_3.
\end{equation*}
Note that this quantity is now independent of the loop $\curve$. The conclusion directly follows
from the bound in \cref{lem:claim_exp_phi_T} at the order $\ord=\chi(\type)$.}

  If now $\curve$ is a double-filling loop on the once-holed torus $\mathbf{S}_{1,1}$, then
  \eqref{eq:apply_asymp_proof_doublefill} is replaced by
  \begin{equation*}
    \int_{0}^{L} |f_k^\type(\ell)| \d \ell
    \leq \int_{0}^\infty \int_{\mathcal{T}_{1,1}(x)}
    \,  \1{[0,L]}(\ell_Y(\curve)) \, |\psi_k^{(1,1)}(x)|
    \d \Volwp[1,1,x](Y) \d x.
  \end{equation*}
  We shall prove that, for any $L, x >0$,
  \begin{equation}
    \label{eq:bound_vol_dom_fund_11}
    \Volwp[1,1,x] (\{ Y \in \mathcal{T}_{1,1}(x) \, : \, \ell_Y(\curve) \leq
    L\}) \leq 4 L^2. 
  \end{equation}
  This allows to bound the integral over $\mathcal{T}_{1,1}(x)$, and then the rest of the proof is
  exactly identical to the pair of pants case above.

  Let us equip $\Sf_{1,1}$ with an arbitrary hyperbolic structure, and assume that $\curve$ is the
  geodesic representative for this metric. We consider a parametrization
  $\curve : [0,1] \rightarrow \Sf_{1,1}$. The loop $\curve$ is filling and therefore not simple, and
  we can therefore define
  \begin{equation*}
    t_+ := \inf \{ t \geq 0 \, : \, \exists s \in [0,t) \, : \, \curve(s)=\curve(t)\}
  \end{equation*}
  and $t_- < t_+$ so that $z := \curve(t_-) = \curve(t_+)$. Then, the sub-segment
  $\curve_{|[t_-,t_+]}$ is a simple loop~$\alpha$ on $\Sf_{1,1}$ based at $z$.

  The surface $\Sf_{1,1} \setminus \alpha$ is a topological pair of pants $\mathbf{P}$. We denote as $\alpha_1$
  and $\alpha_2$ the two boundary components of $\mathbf{P}$ corresponding to $\alpha$, so that the path
  $\curve$ enters $\mathbf{P}$ through $\alpha_1$ at the time $t_+$. We denote as $z_1$ and $z_2$ the copies
  of $z$ on $\mathbf{P}$, following this numbering convention.

  Let us consider a small enough $\epsilon > 0$ such that
  $\mathcal{C}_{\epsilon} := \{w \in \mathbf{P} \, : \, 0 < \mathrm{dist}(w,\alpha_2) < \epsilon\}$ is a
  cylinder. The loop $\curve$ fills the surface $\Sf_{1,1}$, and $\mathcal{C}_\epsilon$ is an annulus
  that is not homotopic to $\partial \Sf_{1,1}$, and hence the set $\mathcal{C}_\epsilon$ cannot lie
  entirely in the complement of $\curve$. We deduce that there exists a sub-segment $\beta_0$ of
  $\curve$ (or $\curve^{-1}$) going from a point of $\alpha_2$ to a point of $\partial \mathbf{P}$ which
  cannot be homotoped to a path on $\alpha_2$.  Because $\curve$ does not intersect
  $\partial \Sf_{1,1}$, there are two possibilities for the topology of $\beta_0$, represented in
  \cref{fig:proof_double_fill_R_cases}.
  \begin{figure}
    \centering
      \begin{subfigure}[t]{0.35\textwidth}
    \centering
    \includegraphics[scale=0.3]{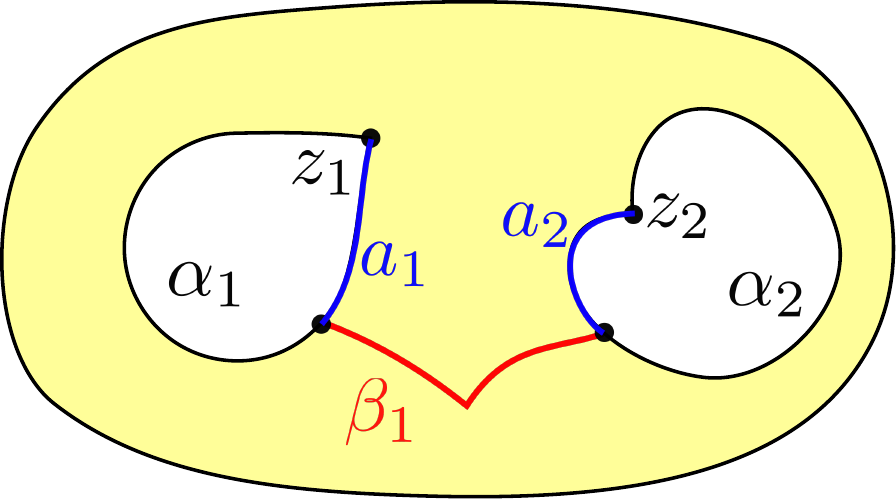}
    \caption{The case \eqref{en:proof_double_fill_R_cases_1}.}
  \end{subfigure}%
  \begin{subfigure}[t]{0.65\textwidth}
    \centering
    \includegraphics[scale=0.3]{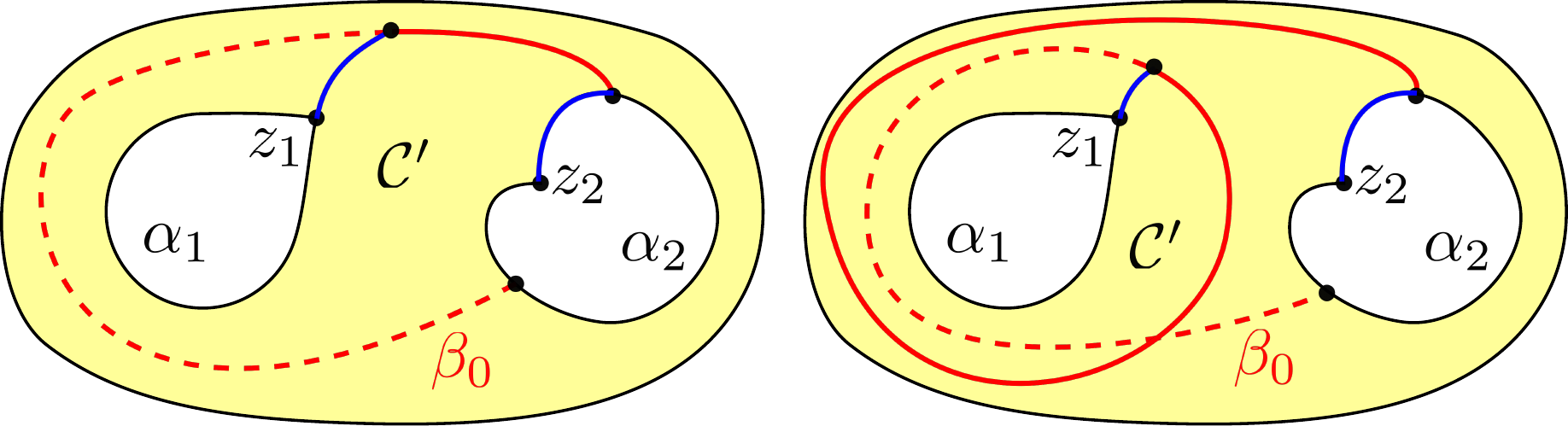}
    \caption{Two possible scenarios in case \eqref{en:proof_double_fill_R_cases_2}.}
  \end{subfigure}%
    \caption{Illustration of the proof of \cref{eq:bound_vol_dom_fund_11}.}
    \label{fig:proof_double_fill_R_cases}
  \end{figure}
  In both those cases, let us prove that can use pieces of $\curve$ to create a simple closed loop
  $\beta$ on $\Sf_{1,1}$ that intersects $\alpha$ exactly once.
  \begin{enumerate}
  \item \label{en:proof_double_fill_R_cases_1} If $\beta_0$ goes from $\alpha_2$ to $\alpha_1$, we
    extract a simple sub-path $\beta_1$ from it by iteratively removing the portion of $\beta_0$
    restricted to $[u_-,u_+]$ for any times $u_- < u_+$ such that $\beta_0(u_-) = \beta_0(u_+)$,
    until no intersections are left. We then take a sub-path $a_2$ of $\alpha_2$ from $z_2$ to the
    beginning of $\beta_1$, and a sub-path $a_1$ of $\alpha_1$ from the endpoint of $\beta_1$ to
    $z_1$. The concatenation $\beta := a_2 \smallbullet \beta_1 \smallbullet a_1$ satisfies the
    hypotheses.
  \item \label{en:proof_double_fill_R_cases_2} If $\beta_0$ goes from $\alpha_2$ to itself, then by
    construction it is not homotopic with fixed endpoints to a portion of $\alpha_2$. We use this
    information to find a topological cylinder $\mathcal{C}'$ which has $\alpha_1$ as one boundary
    component, and portions of $\beta_0$ and perhaps $\alpha_2$ as the other. The path $\curve$
    enters $\mathcal{C}'$ at time $t_+$, and cannot be homotoped to a sub-path of $\alpha_1$
    (otherwise $\Sf_{1,1}$ would contain a geodesic bigon). As a consequence, it has to escape the
    cylinder $\mathcal{C}'$ by its other boundary component. We concatenate the portion of $\curve$
    from $t_+$ to its escape of $\mathcal{C}'$, as well as portions of $\beta_0$, and finally a
    portion of $\alpha_2$ from the starting point of $\beta_0$ to $z_2$, to construct~$\beta$.
  \end{enumerate}

  Then, by construction,
  \begin{equation*}
    \forall Y \in \mathcal{T}_{1,1}(x), \quad
    \begin{cases}
      \ell_Y(\alpha) \leq \ell_Y(\curve) \\
      \ell_Y(\beta) \leq 2 \ell_Y(\curve).
    \end{cases}
  \end{equation*}
  For $x >0$, we consider the Fenchel--Nielsen coordinates $(\ell_1,\tau_1)$ on
  $\mathcal{T}_{1,1}(x)$ so that $\ell_1$ is equal to $\ell_Y(\alpha)$, and the geodesics homotopic
  to $\alpha$ and $\beta$ are orthogonal when $\tau_1=0$. Then, an elementary computation yields
  \begin{equation*}
    \cosh (\ell_Y(\curve))
    \geq \cosh \div{\ell_Y(\beta)}
    = \cosh \div{\ell_1} \cosh \div{\tau_1}
    \geq \cosh \div{\tau_1}
  \end{equation*}
  and therefore, for any $L >0$,
  \begin{equation*}
    \{(\ell_1, \tau_1) \in \R_{>0} \times \R \, : \, \ell_Y(\curve) \leq L\}
    \subset [0, L] \times [-2L, 2L]
  \end{equation*}
  which is enough to conclude to \eqref{eq:bound_vol_dom_fund_11}.
\end{proof}

\subsection{Classification of filling loops that do not double-fill}
\label{sec:class-fill-curv}

Thanks to Sections~\ref{sec:generalised_convolutions} and
\ref{sec:double-fill-curv}, we now know that \cref{thm:all_chi_-1} holds for
figure-eights and double-filling loops.  We have seen, at the centre of
\cref{fig:double_fill_exp}, an example of a loop filling a pair of pants which
is neither a double-filling loop nor a figure-eight.  We make the following
definition to refer to this loop, and other similar examples.

\begin{figure}[h]
  \includegraphics[scale=0.7]{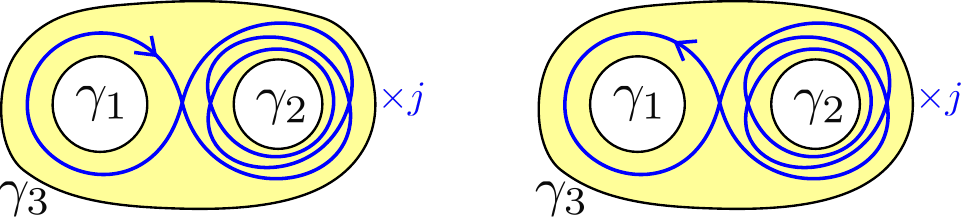}
  \caption{Two distinct one-sided figure eights with $j=3$ iterations.}
  \label{fig:one_sided_it_def}
\end{figure}

\begin{nota}
  We call the local topological types represented in \cref{fig:one_sided_it_def} \emph{one-sided
    iterated eights}, and refer to the integer $j$ as their \emph{number of iterations}. For any
  $j \geq 2$, we denote as $\curve_j$ the loop on the left-hand side of \cref{fig:one_sided_it_def}.
\end{nota}

We prove the following.

\begin{prp}
  \label{prp:non-double-fill}
  Let $\Sf$ be a filling type of absolute Euler characteristic $1$. Then, for any loop $\curve$
  filling $\Sf$, one and only of the two following occurs:
  \begin{itemize}
  \item either $\curve$ is double-filling;
  \item or $\Sf$ is a pair of pants, and $\curve$ is a figure-eight or a one-sided iterated eight.
  \end{itemize}
\end{prp}

As a consequence, one-sided iterated eights are the only remaining case that we need to consider to
prove \ref{chal:FR_type} when $\chi(\Sf)=1$. We shall do so in \cref{sec:one-sided-iterated}.

\begin{proof}
  First, let $\Sf = \mathbf{P}$ be a pair of pants, and $\curve$ be a filling loop that is
  not double-filling. Let us prove that $\curve$ is a figure-eight or a
  one-sided iterated eight.

  \begin{figure}[h]
    \includegraphics[scale=0.7]{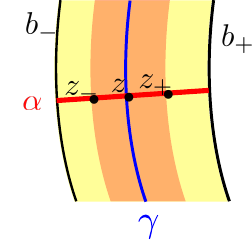} \hspace{2em}
    \includegraphics[scale=0.7]{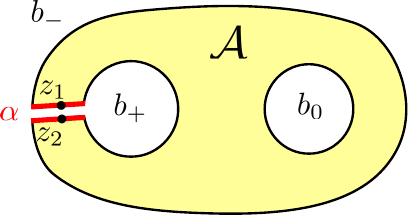} \hspace{2em}
    \includegraphics[scale=0.7]{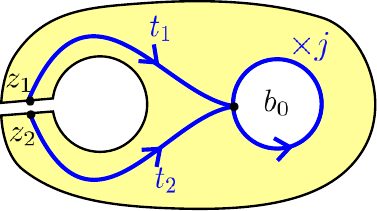}
    \caption{Illustration of the proof in the pair of pants case.}
    \label{fig:non-double-fill}
  \end{figure}

  Let $\mathcal{N}$ be a regular neighbourhood of $\curve$. By definition of double-filling loops,
  there exists a simple portion of $\curve$, denoted as $\cI$, such that none of the two boundary
  components $b_\pm(\mathcal{N})$ of $\cN$ bordering $\curve$ are contractible. They therefore are
  homotopic to two (non-necessarily distinct) boundary components~$b_\pm$ of $\mathbf{P}$. Let us
  pick a point $z$ on $\cI$, two neighbours $z_\pm$ of $z$ on $b_\pm({\mathcal{N}})$, and a simple
  path $c_0$ going from $z_+$ to $z_-$ by crossing the regular neighbourhood $\mathcal{N}$
  transversally, which only intersects $\curve$ at the point $z$.  For $\sigma \in \{+, -\}$, since
  $b_\sigma({\mathcal{N}})$ and $b_\sigma$ delimit an annulus, there exists a simple path $c_\sigma$
  entirely contained in this annulus going from $b_\sigma$ to $z_\sigma$. Then, we let
  $\alpha := c_+ \smallbullet c_0 \smallbullet c_-^{-1}$, as represented in the left part of
  \cref{fig:non-double-fill}. By construction, the path $\alpha$ is simple, and only intersects
  $\curve$ once, at $z$.  Let us cut the pair of pants $\mathbf{P}$ along the simple path $\alpha$
  going from $b_+$ to $b_-$.

  If $b_+ = b_-$, then the path $\alpha$ would separate the pair of pants $\mathbf{P}$ it into two connected
  components, while it intersects $\curve$ only once. This is impossible because the loop~$\curve$
  is closed.  As a consequence, $b_+ \neq b_-$, and the result we obtain by cutting $\mathbf{P}$ along
  $\alpha$ is therefore an annulus $\mathcal{A}$, as represented on the middle part of
  \cref{fig:non-double-fill}.

  The point $z$ corresponds to two marked points $z_{1}$ and $z_2$ on one boundary component of the
  annulus. Because $\curve$ only intersects $\alpha$ once, at $z$, the path $\curve$ corresponds to
  a path $\curve'$ on $\mathcal{A}$, going from $z_1$ to $z_2$. The fundamental group of an annulus
  is $\Z$, and therefore $\curve'$ is homotopic with fixed endpoint to
  $t_1 \smallbullet b_0^j \smallbullet t_2^{-1}$ for a $j \in \Z$, where~$b_0$ is the boundary
  component of $\mathcal{A}$ on which $z_1$ and $z_2$ do not lie, corresponding to the third
  boundary component of the pair of pants, and $t_{1}$, $t_2$ are paths from $z_{1}$, $z_2$ to a
  shared fixed point of $b_0$ (see the right-hand side of \cref{fig:non-double-fill}). It follows
  that $\curve$ is a figure-eight or a one-sided iterated eight.

  \begin{figure}[h]
    \includegraphics[scale=1]{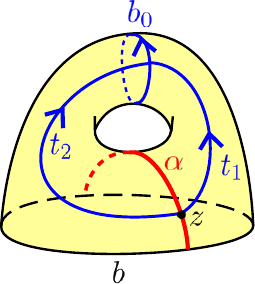}
    \caption{Illustration of the proof in the once-holed torus case.}
    \label{fig:non-double-fill-torus}
  \end{figure}  
  
  Now, let us assume that $\Sf$ is the once-holed torus $\Sf_{1,1}$, and let $\curve$ be a loop on
  $\Sf_{1,1}$ that is not double-filling. We prove that $\curve$ does not fill $\Sf_{1,1}$. Indeed, let
  us construct $b_-$, $b_+$, $\alpha$ be as above.  The once-holed torus $\Sf_{1,1}$ only has one
  boundary component, i.e.  $b_- = b_+ =: b$. The surface $\Sf_{1,1} \setminus \alpha$ has Euler
  characteristic $0$ and cannot be disconnected, since $\curve$ is a closed loop intersecting
  $\alpha$ exactly once. As a consequence, it is an annulus.  Let~$z_1$ and $z_2$ denote the two
  copies of $z$, that lie on the two distinct boundary components of the annulus. The fundamental
  group of an annulus is $\Z$, so the path $\curve$ on the annulus can be written as
  $t_1 \smallbullet b_0^j \smallbullet t_2^{-1}$ for a $j \in \Z$, where $b_0$ is the core of
  the annulus and $t_1$, $t_2$ are two paths connecting $z_1$, $z_2$ to a shared point of $b_0$,
  as represented on \cref{fig:non-double-fill-torus}. This path admits a simple representative in
  the homotopy class with fixed endpoints, which corresponds to a simple loop on $\Sf_{1,1}$ homotopic
  to~$\curve$. By minimality of the intersection number for the geodesic representation, this
  implies that~$\curve$ is simple. In particular, $\curve$ does not fill $\Sf_{1,1}$, which was our
  claim.

  Because figure-eights and one-sided iterated eights are not double-filling, it
  is clear the two cases are mutually exclusive.
\end{proof}

\subsection{Proof  for one-sided iterated eights}
\label{sec:one-sided-iterated}

Let us now proceed to the last step to the a proof of \cref{thm:all_chi_-1}, which consists in
proving it for any one-sided iterated eight.  For any $j \geq 2$, there are exactly two distinct
local types of one-sided iterated eights with $j$ iterations, represented on
\cref{fig:one_sided_it_def}.  Because these two types only differ by orientation, we shall only
study the local type $\type_j := \eqc{\mathbf{P},\curve_j}$. Let us prove the following statement.

\begin{prp}
  \label{thm:it_eight}
  For any $k \geq 1$, there exists a constant $c_k \geq 0$ satisfying the
  following. For any $j \geq 2$, the function $f_k^{\type_j}$ belongs in
  $\cF^{c_k,c_k}$ and its norm satisfies $\|f_k^{\type_j}\|_{\cF^{c_k,c_k}} \leq c_k$.
\end{prp}

The proof is extremely similar to the case of the figure eight, done in detail
in \cref{sec:generalised_convolutions}, and we shall therefore only sketch it.
The first step we need to take is to compute the length-formula relating the
length of the geodesic $\curve_j$ to the three boundary lengths of the pair of pants,
generalising \cref{eq:formula_length_eight} for $j \geq 1$. This leads to the
following.

\begin{lem}
  \label{lem:expression_before_change_var_j}
  The expression for $(f_k^{\type_1})_{k \geq 1}$ from \cref{lem:expression_before_change_var}, true
  for the figure-eight, can be extended to the local type $\type_j$ for $j \geq 2$ by replacing the
  length-function $h$ in the level-set integrals by the new length-function $h_j$ satisfying
  \begin{equation}
    \label{lem:formula_length_it}
    \cosh \div{h_j(\ell_1, \ell_2, \ell_3)}
    = 
    \frac{\cosh \div{\ell_1} \sinh \div{(j+1) \ell_2} 
      + \cosh \div{\ell_3} \sinh \div{j \ell_2}}{\sinh \div{\ell_2}} \cdot
  \end{equation}
  The family of functions $(f_i)_{i}$ and the coefficients that appear in this
  expansion does not depend on the parameter~$j$.
\end{lem}

\begin{proof}
  We simply follow the exact same steps as in Sections
  \ref{sec:expression-as-an} to \ref{eq:term_F_exp_tor_level_set}. The only
  difference here is that the expression of the length,
  \eqref{eq:formula_length_eight} in the case of the figure eight,
  and~\eqref{lem:formula_length_it} here.  We prove that, for any $j \geq 2$,
  \begin{equation}
    \label{eq:rec_formula_it}
    \cosh \div{h_j(\ell_1, \ell_2, \ell_3)}
    = 2 \cosh \div{\ell_1} \cosh \div{j \ell_2} + \cosh
    \div{h_{j-1}(\ell_3, \ell_2, \ell_1)}.
  \end{equation}
  We then obtain \eqref{lem:formula_length_it} by a straightforward induction.

  The proof of \cref{eq:rec_formula_it} is a geometric argument, that uses the interpretation of the
  unit tangent bundle of the pair of pants $\mathbf{P}$ as a quotient of $\mathrm{PSL}_2(\R)$, and
  the formula $\cosh \div{\ell(\curve)} = \frac 12 \abso{\tr (\curve)}$ true for any hyperbolic
  element $\curve$ of $\mathrm{PSL}_2(\R)$. Because these tools only appear in this proof, we will
  not provide more introductory detail, and invite the reader to refer to \cite{katok1992} for more
  details. 

  \begin{figure}[h]
    \centering
    \includegraphics[scale=0.75]{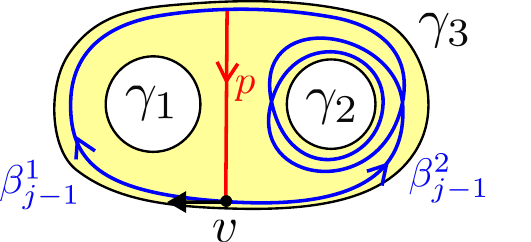}
    \caption{The geometric construction for the proof of \cref{eq:rec_formula_it}.}
    \label{fig:rec_formula_it}
  \end{figure}
  
  The following construction is illustrated on \cref{fig:rec_formula_it}. Let $\beta_{j-1}$ denote
  the one-sided iterated eight obtained by going once around~$\gamma_3$ and $j-1$ times
  around~$\curve_2$, which has length $h_{j-1}(\ell_3, \ell_2, \ell_1)$.  Let $p$ denote
  the common perpendicular of~$\beta_{j-1}$ with itself represented in \cref{fig:rec_formula_it} --
  we shall also use the notation $p$ for its length. The endpoints of $p$ delimits two
  paths~$\beta_{j-1}^1$ and~$\beta_{j-1}^2$ obtained from $\beta_{j-1}$, of respective lengths
  denoted by $t_1$ and $t_2$. Then, the one-sided iterated eight~$\curve_j$, of length
  $h_j(\ell_1,\ell_2,\ell_3)$, is freely homotopic to
  $\beta_{j-1}^1 \smallbullet p \smallbullet \beta_{j-1}^2 \smallbullet p$. Transporting the vector
  $v$ represented on \cref{fig:rec_formula_it} along this trajectory using elements of
  $\mathrm{SL}_2(\R)$, we obtain that
  \begin{equation}
    \label{eq:trace_expr}
    \cosh \div{h_j(\ell_1, \ell_2, \ell_3)}
    = \frac{1}{2} \, \abso{\tr (a_{t_1} w_p a_{t_2} w_{-p})}
  \end{equation}
  where 
  \begin{equation*}
    a_t :=
    \begin{pmatrix}
      e^{\frac t2} & 0 \\ 0 & e^{- \frac t2}
    \end{pmatrix}
    \quad \text{and} \quad
    w_p :=
    \begin{pmatrix}
      \cosh \div{p} & \sinh \div{p} \\ \sinh \div{p} & \cosh \div{p}
    \end{pmatrix}
  \end{equation*}
  respectively correspond to moving the vector $v$ in a straight direction for a
  length $t$, or along the orthogonal geodesic in the right direction for the
  (algebraic) length $p$.  We rewrite \eqref{eq:trace_expr} as
  \begin{equation}
    \label{eq:trace_expr_reexp}
    \cosh \div{h_j(\ell_1, \ell_2, \ell_3)}
    = \frac{1}{2} \, \abso{\tr (A_1 A_2)}   
  \end{equation}
  where
  \begin{equation}
    \label{eq:2}
    A_1 := a_{t_1} w_p R, \quad
    A_2 := R^{-1} a_{t_2} w_{-p}, \quad \text{with} \quad
    R :=
    \begin{pmatrix}
      0 & 1 \\ -1 & 0
    \end{pmatrix}.
  \end{equation}
  In $\mathrm{SL}_2(\R)$, we have that
  \begin{equation}
    \label{eq:trace_prod}
    \tr (A_1A_2)
    = \tr(A_1) \, \tr(A_2) + \tr(- A_1 A_2^{-1}).
  \end{equation}
  We observe by a direct computation that for $i=1, 2$,
  $ \tr(A_i) = 2 \sinh \div{p} \sinh \div{t_i} > 0$, and the path described by
  $A_i$ is freely homotopic to the closed geodesic $\gamma_1$ for $i=1$
  and~$\gamma_2^j$ for $i=2$. Hence,
  \begin{equation}
    \label{eq:trace_one}
    \tr (A_1) = 2 \cosh \div{\ell_1}
    \quad \text{and} \quad
    \tr (A_2) = 2 \cosh \div{j \ell_2}.
  \end{equation}
  We use that $R^{-1}=-R$, $a_{t_2}^{-1} = a_{-t_2}$, $w_p^{-1}=w_{-p}$ and
  $R \, w_p a_{-t_2} \, R^{-1} = w_{-p} a_{t_2}$ to obtain
  \begin{equation}
    \label{eq:trace_inv}
    \tr(- A_1 A_2^{-1}) = \tr(a_{t_1} w_p  R \, w_p a_{-t_2} R^{-1}) =
    \tr(a_{t_1} a_{t_2})
    = 2 \cosh \div{h_{j-1}(\ell_3, \ell_2, \ell_1)}.
  \end{equation}
  The conclusion then follows from
  \cref{eq:trace_expr,eq:trace_prod,eq:trace_one,eq:trace_inv}.
\end{proof}

In the case of the pair of pants, \cref{lem:easy_cases_remainder} allowed us to
deal with a few elementary cases. We easily adapt it to this new setting, and
prove the following.

\begin{lem}
  \label{lem:easy_cases_remainder_j}
  For any $j \geq 2$, the integrals of the form \eqref{eq:integral_to_compute_FR_tor} are always
  elements of $\FRrem \subset \FR$, and so is the integral \eqref{eq:integral_to_compute_FR} as soon
  as $p_1=0$ or $p_3=0$. The estimate on the remainders can be made uniformly in $j$.
\end{lem}

\begin{proof}
  We only sketch the proof for $p_1=0$, because this case is actually new
  compared to before (in the case of the eight, we needed to assume
  $p_1=p_2=0$). The difference is the fact that, now,
  \begin{equation*}
    \frac{\partial \ell_3}{\partial \ell}
    = 
    \frac{\sinh \div{\ell_2}}{\sinh \div{j \ell_2}}
    \frac{\sinh \div{\ell}}{\sinh \div{\ell_3}}
  \end{equation*}
  and hence the integral we need to bound is
  \begin{equation*}
    \sinh \div{\ell} \iint \left(  \prod_{i=1}^3 f_i(\ell_i) \right)
    \sinh \div{\ell_1} \frac{\sinh^2 \div{\ell_2}}{\sinh \div{j \ell_2}}
    \1{I(\ell_1,\ell_2)}(\ell) \d \ell_1 \d \ell_2
  \end{equation*}
  which, if $p_1=0$, is clearly an element of $\FRrem$, uniformly in $j$,
  because for any $j \geq 2$,
  \begin{equation*}
    \frac{\sinh^2 \div{\ell_2}}{\sinh \div{j \ell_2}}
    \leq \frac{\sinh^2 \div{\ell_2}}{\sinh(\ell_2)} = \O{1}.
  \end{equation*}
\end{proof}

\begin{figure}[h]
  \includegraphics[scale=0.75]{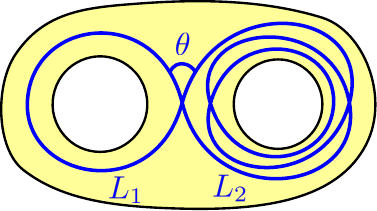}
  \caption{The coordinates for the one-sided iterated eight.}
  \label{fig:one_sided_it_coord}
\end{figure}

For the rest of the proof, we need to pick new variables, suited to the
integration on the new level-set
$\{(\ell_1,\ell_2,\ell_3) \in \R_{>0}^3 \, : \,
h_j(\ell_1,\ell_2,\ell_3)=\ell\}$.  We choose the natural adaptation of
the case of the figure-eight, illustrated in \cref{fig:one_sided_it_coord}. More
precisely, $L_1$ and $L_2$ denote the lengths of the two portions of $\curve_j$
going around $\gamma_1$ (once) and $\gamma_2$ ($j$ times) respectively, so that
$\ell_Y(\curve_j)=L_1+L_2$. We once again let $u := \cos^2 \div{\theta}$,
where~$\theta$ is the angle represented on \cref{fig:one_sided_it_coord}. We
then prove the following formula for the change of variables, generalising
\cref{lem:l123_calc} for $j \geq 1$.

\begin{lem}
  \label{lem:l123_calc_it}
  For any $j \geq 1$, any $\ell_1, \ell_2, \ell_3 > 0$,
  \begin{equation*}{}
    \begin{cases}
    &  \cosh \div{\ell_1} =  \sqrt{u} \, \cosh \div{L_1} \\
    &  \cosh \div{j \ell_2} =  \sqrt{u} \, \cosh \div{L_2} \\   
    & \cosh \div{\ell_3} = a_j(\ell_2) \brac*{(1 - b_j(\ell_2) \, u)\cosh
       \div{L_1+L_2} - b_j(\ell_2) \, u \cosh \div{L_1-L_2}} \\
    \end{cases}
  \end{equation*}
  where
  \begin{equation*}
    a_j(\ell_2) := \frac{\sinh \div{\ell_2}}{\sinh \div{j \ell_2}}
    \quad \text{and} \quad
    b_j(\ell_2) := \frac{\sinh \div{(j+1)\ell_2}}{2 \sinh \div{\ell_2} \cosh
      \div{j \ell_2}} \cdot
  \end{equation*}
\end{lem}

We then proceed to compute the expression of the Jacobian of the change of
variables.

\begin{lem}
  For any $j \geq 1$, the change of variable
  $(\ell_1, \ell_2, \ell_3) \rightarrow (L_1, L_2, u)$ is a diffeomorphism from
  $\R_{>0}^3$ to an open set $\mathfrak{D}_j$ of $\R_{>0}^2 \times (0,1)$, and
  its Jacobian can be written as
  \begin{align*}
    \prod_{i=1}^3 \sinh \div{\ell_i} \d \ell_i 
     = - \sinh^2 \div{L_1+L_2} \frac{\sinh^2 \div{\ell_2}}{j \sinh^2 \div{j \ell_2}} \d L_1 \d L_2 \d u.
  \end{align*}
\end{lem}

\begin{proof}
  In order to compute the Jacobian, we shall use the expression of $\ell_3$
  derived in \cref{lem:easy_cases_remainder_j}:
  \begin{align*}
    \cosh \div{\ell_3}
    & = \frac{\sinh \div{\ell_2} \cosh \div{L_1+L_2}
      - \cosh \div{\ell_1} \sinh\div{(j+1)\ell_2}}{\sinh \div{j \ell_2}} \\
    & = \mathrm{expr}(L_1,L_2,\ell_1(L_1,L_2,u),\ell_2(L_1,L_2,u)).
  \end{align*}
  We observe that, by the chain rule, we can write the derivative of the
  previous equation according to $L_1$ as:
  \begin{equation*}
    \frac{1}{2} \sinh \div{\ell_3} \frac{\partial \ell_3}{\partial L_1}
    = \frac{\sinh \div{\ell_2}}{2 \sinh \div{j \ell_2}} \sinh \div{L_1+L_2}
    + \frac{\partial \ell_1}{\partial L_1} \partial_3 \mathrm{expr}
    + \frac{\partial \ell_2}{\partial L_1} \partial_4 \mathrm{expr}
  \end{equation*}
  and similar expressions are valid for the derivation with respect to $L_2$ and $u$. The advantage
  of these expressions is that it makes it clear that, when we compute the Jacobian, we can remove
  from the column of the variable $\ell_3$ the crossed terms in which the derivatives
  $\partial_3 \mathrm{expr}$ and $\partial_4 \mathrm{expr}$ appear, by subtracting multiples of the
  columns from the variables $\ell_1$ and $\ell_2$.  We can therefore reduce ourselves to computing
  the determinant
  \begin{equation*}
    \frac{\sinh \div{\ell_2}}{\sinh \div{\ell_1} j \sinh^2 \div{j \ell_2}\sinh \div{\ell_3}}
    \begin{vmatrix}
      \sqrt{u}\sinh \div{L_1} &
      0 &
      \sinh \div{L_1+L_2} \\
      0 &
      \sqrt{u}\sinh \div{L_2} &
      \sinh \div{L_1+L_2}  \\
      \frac{1}{\sqrt{u}} \cosh \div{L_1} &
      \frac{1}{\sqrt{u}} \cosh \div{L_2} &
      0
    \end{vmatrix}
  \end{equation*}
  which yields the claimed expression.
\end{proof}

We will then be able to conclude using the following.

\begin{prp}
  \label{prp:final_lem_it}
  Let $(f_i)_{1 \leq i \leq 3}$ be functions satisfying
  \eqref{eq:expansion_assumption_fi}. For any $j \geq 2$, the function
  \begin{equation}
    \label{eq:final_lem_it}
    \mathrm{Int}_j[f_1,f_2,f_3] : 
    \ell \mapsto \iint_{L_1+L_2=\ell}
    \left( \prod_{i=1}^3 f_i(\ell_i) \right)
    \frac{\sinh^2 \div{\ell_2}}{j \sinh^2 \div{j \ell_2}}
    \1{\mathfrak{D}_j}(L_1,L_2,u) \frac{\d L_1 \d L_2 \d u}{\d \ell}
  \end{equation}
  satisfies an estimate of the form
  \begin{equation*}
    \mathrm{Int}_j[f_1,f_2,f_3](\ell)
    = P_j[f_1,f_2,f_3](\ell) + \O{(\ell+1)^ce^{-\ell/2}}
  \end{equation*}
  uniformly in $j$, 
  where $P_j[f_1,f_2,f_3]$ is a polynomial function of degree and coefficients bounded uniformly in
  $j$. 
\end{prp}

The proof of \cref{prp:final_lem_it} is a straightforward adaptation of the
proof in the case of the figure-eight, thanks to the similarity of the formulas
for the changes of variables. Actually, it is slightly more elementary thanks to
the additional decay in $\ell_2$ in \cref{eq:final_lem_it}, which removes the
polynomial behaviour of $\ell_2$ and hence allows to consider the dependency in
$\ell_2$ as a ``remainder term'' directly.


\section{The second-order term is not a Friedman--Ramanujan function}
\label{sec:second-order-term}

In this section, we shall prove the following result, the hyperbolic surface
counterpart of Theorem 2.12 from \cite{friedman2003} for random regular graphs.

\begin{thm}
  \label{prp:non_FR}
  The function $\ell \mapsto \ell f_1^{\mathrm{all}}(\ell)$ is not a Friedman--Ramanujan function in
  the weak sense.
\end{thm}

We shall prove this by contradiction, grouping the following two observations.
\begin{itemize}
\item On the one hand, in \cref{sec:proof-first_contradiction}, we prove that
  $\Pwp{\lambda_1 \leq a}$ is not very small as $g \rightarrow + \infty$ (for a fixed $a$). Indeed,
  we know that this probability goes to $0$ as soon as $a < \frac{3}{16}$ by
  \cite{wu2022,lipnowski2021}, but we show that it does so at the speed $1/g$ as
  $g \rightarrow + \infty$ only, and no faster.
\item On the other hand, in \cref{sec:proof-crefprp:sm}, we show that, if
  $\ell \mapsto \ell f_1^{\mathrm{all}}(\ell)$ was Friedman--Ramanujan in the weak sense, then we would be able
  to prove that $\Pwp{\lambda_1 \leq a}$ goes to $0$ at a rate $1/g^{1+\delta}$ for a $\delta>0$.
\end{itemize}

\Cref{prp:non_FR} could appear to be a contradiction, because we have shown that
$f_1^{\mathrm{all}}$ is the sum of $f_1^\type$ for all local types $\type$ that are simple or filling a
surface of Euler characteristic~$-1$, and $\ell \mapsto \ell f_1^\type(\ell)$ is Friedman--Ramanujan in
the weak sense for all those types by \cref{prop:simple} and \cref{thm:all_chi_-1}. However, the
class $\FRw$ is not stable by countable summation, so there is no contradiction.

The reader is invited to read this section thinking of the fact that the issue causing that
$\ell f_1^{\mathrm{all}}(\ell) \notin \FRw$ is the possible existence of \emph{tangles}, as
introduced in \cite{monk2021a,lipnowski2021} for hyperbolic surfaces, and
\cite{friedman2003,bordenave2020} for random regular graphs. We shall see that tangles are small
pairs of pants or once-holed tori, which contain ``too many'' closed geodesics, and cause the first
non-zero eigenvalue $\lambda_1$ to be small.

\subsection{Estimate of the probability of having a small eigenvalue}
\label{sec:proof-first_contradiction}

Let us prove the following result on the speed of convergence of
$\Pwp{\lambda_1 \leq a}$ as $g \rightarrow + \infty$.

\begin{thm}
  \label{lem:prob_small_eig}
  There exists $a_0, c_1, c_2 >0$ such that, for any $a \leq a_0$, any large enough~$g$,
  \begin{equation}
    \label{eq:prob_small_eig}
    c_1 \frac{a^2}{g} \leq \Pwp{\lambda_1 \leq a} \leq c_2 \, \frac{a}{g} \cdot
  \end{equation}
\end{thm}

The proof of the upper and lower bounds are different, and will be treated separately.

\subsubsection{The lower bound}

In order to prove the lower bound, we shall prove that, if a surface contains a
once-holed torus with a short boundary, then it has a small eigenvalue.

\begin{lem}
  \label{lem:minmax}
  Let $X$ be a compact hyperbolic surface of genus $g \geq 2$. We assume
  that~$X$ contains an embedded once-holed torus $Y$ with geodesic boundary, and
  that $\ell_X(\partial Y) \leq 1$. Then, $\lambda_1(X) \leq \ell_X(\partial Y)$.
\end{lem}

\begin{proof}
  The min-max principle allows to bound eigenvalues in terms of well-chosen
  Rayleigh quotients. A classic application, that can be found in \cite[Theorem
  8.2.1]{buser1992} for instance, states that, if $\phi_1, \phi_2 \in H^1(X)$
  are $L^2$-orthogonal, then
  \begin{equation}
    \label{e:minmax}
    \lambda_1(X) \leq \max \left\{ \frac{\int_X \| \grad \phi_1 \|^2 \d \mu}{\int_X |\phi_1|^2 \d \mu},
      \frac{\int_X \| \grad \phi_2 \|^2 \d \mu}{\int_X |\phi_2|^2 \d \mu}\right\},
  \end{equation}
  where $\grad$ and $\d \mu$ are the gradient and volume form associated to the metric on $X$.

  By the collar lemma \cite[Theorem 4.1.1]{buser1992}, the neighbourhood of width
  \begin{equation*}
    \arcsinh \paren*{\sinh \div{\ell_X(\partial Y)}^{-1}}
  \end{equation*}
  of $\partial Y$ is isometric to a cylinder. As soon
  as $\ell_X(\partial Y) \leq 1$, this width is larger than $\arcsinh(1/\sinh(1/2)) > 1$. We can
  therefore define two elements of $H^1(X)$ by
  \begin{align*}
    \phi_1 (z) := \min (1, \mathrm{dist}(z,X \setminus Y))
    \quad \text{and} \quad
    \phi_2 (z) := \min (1, \mathrm{dist}(z,Y)).
  \end{align*}
  These functions are $L^2$-orthogonal because they have disjoint support, and we
  can therefore apply the min-max principle to them. Let us estimate the Rayleigh
  quotients appearing in \cref{e:minmax}.

  First, the norm of the gradient of $\phi_1$ is equal to $1$ on the set
  \begin{equation*}
    Y^- := \{z \in Y \, : \, \mathrm{dist}(z,X \setminus Y) < 1\},
  \end{equation*}
  and $0$ outside of it. Therefore,
  \begin{equation*}
    \int_X \| \grad \phi_1 \|^2 \d \mu
    \leq \area(Y^-) = \int_0^{\ell_X(\partial Y)} \int_0^1 \cosh \rho  \d \rho \d t
    = \sinh (1) \; \ell_X(\partial Y)
  \end{equation*}
  where the area is computed using Fermi coordinates (see \cite[Section 1.1]{buser1992}).
  We then observe that $\phi_1$ is identically equal to $1$ in $Y \setminus Y^-$, and hence
    \begin{equation*}
      \int_X |\phi_1|^2 \d \mu \geq \area(Y) - \area(Y^-) \geq 2 \pi - \sinh(1) \, \ell_X(\partial Y)
      \geq \pi
    \end{equation*}
    provided that $\ell_X(\partial Y) \leq \pi / \sinh 1$, which is the case by hypothesis.
  
  The function $\phi_2$ satisfies the same bounds.
  Therefore, \cref{e:minmax} implies that
  $\lambda_1(X) \leq \sinh(1) \, \ell_X(\partial Y)/\pi \leq \ell_X(\partial Y)$, which is our
  claim.
\end{proof}

We then estimate the probability for a surface to contain a small once-holed torus.

\begin{lem}
  \label{lem:once_holed}
  There exists $a_0, c_1, c_2 >0$ such that, for any $a \leq a_0$, any large enough~$g$,
  \begin{equation}
    \label{eq:once_holed}
     c_1 \, \frac{a^ 2}{g}
    \leq \Pwp{ \exists \text{ once-holed torus } Y \subset X \, : \, \ell_X(\partial Y) \leq
      a}
    \leq c_2 \, \frac{a^2}{g} \cdot
  \end{equation}  
\end{lem}

The proof of the lower bound of \cref{lem:prob_small_eig} then directly follows,
because for any small enough $a$ and large enough $g$,
\begin{equation*}
  \Pwp{\lambda_1 \leq a}
  \underset{\text{Lemma } \ref{lem:minmax}}{\geq}
  \Pwp{\exists \text{ once-holed torus } Y \subset X \, : \ell_X(\partial Y) \leq a}
  \underset{\text{Lemma } \ref{lem:once_holed}}{\geq}
  c_1 \frac{a^2}{g} \cdot
\end{equation*}
We actually only need the lower bound part of \cref{lem:once_holed} to conclude to the lower bound
of \cref{lem:prob_small_eig}, but it is the hard part of the statement.

The proof of \cref{lem:once_holed} is inspired by Mirzakhani's proof of the fact that
\begin{equation*}
  c_1 \, a^2
  \leq \Pwp{\exists \text{ closed geodesic } \gamma \text{ on } X \, : \, \ell_X(\gamma) \leq a}
  \leq c_2 \, a^2,
\end{equation*}
done in \cite[Theorem 4.2]{mirzakhani2013}. We adapt it to count small once-holed tori rather than small
closed geodesics.

\begin{proof}
  For $a >0$, let $\ntor(X)$ denote the number of once-holed tori $Y$ with geodesic boundary
  embedded in $X$ such that $\ell_X(\partial Y) \leq a$. Let us estimate $\Pwp{\ntor \geq 1}$.
  
  \emph{Step 1: expectation estimate.} It is easy to compute the expectation of
  the random variable $\ntor$, because it is a geometric function, and we can
  therefore apply Mirzakhani's integration formula:
  \begin{equation*}
    \Ewpo \big[\ntor \big]
    =  \frac{1}{V_g} \int_0^a V_{1,1}(x) \, V_{g-1,1}(x) \, x \d x.
  \end{equation*}
  We know by \cite{naatanen1998} that $V_{1,1}(x) = \pi^2/12 + x^2/48$. We then replace $V_{g-1,1}(x)$
  by its first-order approximation, see \eqref{e:Vgn_first_order}, and obtain that
  \begin{equation}
    \label{e:tor_exp_before}
    \Ewpo \big[\ntor \big]
    =  \frac{V_{g-1,1}}{V_g} \int_0^a  \paren*{\frac{\pi^2}{6} + \frac{x^2}{12}} \sinh \div{x} \d x
    + \O{\frac{V_{g-1,1}}{g V_g} \int_0^a (1+x)^4 \,e^{\frac x 2}  \d x}.
  \end{equation}
  Since $1/2 < \pi^2/12 < 1$, we can take $a_0$ to be small enough so that
  \begin{equation}
    \forall x \in [0, a_0], \quad
    \frac x 2 \leq \paren*{\frac{\pi^2}{6} + \frac{x^2}{12}} \sinh \div{x} \leq x.
  \end{equation}
  By equations \eqref{eq:asymp_dev_ratio_same_euler} and \eqref{eq:asymp_dev_ratio_add_cusp}, the
  volume ratio is
  \begin{equation}
    \label{eq:proof_lower_non_FR_ratio_vol}
    \frac{V_{g-1,1}}{V_g}
    = \frac{V_{g-1,1}}{V_{g-1,2}} \, \frac{V_{g-1,2}}{V_g}
    = \frac{1}{8 \pi^2 g} + \O{\frac{1}{g^2}}.
  \end{equation}
  Equations \eqref{e:tor_exp_before} to \eqref{eq:proof_lower_non_FR_ratio_vol} together imply the
  existence of constants $c_1, c_2$ (depending on $a_0$) such that, for any $a < a_0$,
  \begin{equation}
    \label{e:tor_exp}
    c_1 \frac{a^2}{g} \leq \Ewpo \big[\ntor \big] \leq c_2 \frac{a^2}{g} \cdot
  \end{equation}

  \emph{Step 2: from the expectation to the probability.}
  By Markov's inequality, we can directly obtain the upper bound part of the claim:
  \begin{equation*}
    \Pwp{\ntor \geq 1} \leq \Ewpo \big[\ntor \big] \leq c_2 \, \frac{a^2}{g} \cdot
  \end{equation*}
  The other side of the inequality is harder to obtain because there is no lower bound on
  $\mathbb{P}(\ntor \geq 1)$ in term of $\mathbb{E}[\ntor]$ in general. We shall express the
  expectation of $\ntor$ as
  \begin{equation}
        \label{e:Pwp_vs_Ewp_tor}
    \Ewpo \big[\ntor \big]
    = \sum_{k=1}^{+ \infty} \Pwp{\ntor \geq k}
    = \Pwp{\ntor \geq 1} + \sum_{k=2}^{+ \infty} \Pwp{\ntor \geq k}.
  \end{equation}
  Thanks to this expression, we observe that in order to prove that
  $\Pwp{\ntor \geq 1}$ is bounded below by $c_1' a^2/g$, it is sufficient to prove
  that $\sum_{k=2}^{+ \infty} \Pwp{\ntor \geq k}$ is negligible compared to
  $a^2/g$. This reduces the problem to studying the probability for a typical
  surface to contain multiple small once-hole tori.

  Let $X$ be a compact hyperbolic surface of genus $g$. Let us describe the topology of families of
  embedded once-holed tori in the surface~$X$.  In order to make the discussion simpler, let us
  assume that $a_0 \leq 2 \arcsinh 1$. Then, by the collar lemma \cite[Theorem 4.1.6]{buser1992},
  all closed geodesics shorter than $a_0$ on $X$ are pairwise disjoint.  As a consequence, if $X$
  contains $k$ once-holed tori $Y_1, \ldots, Y_k$ of boundary lengths $\leq a \leq a_0$, then
  they are all disjoint (here, we assume that $g > 2$, so that no two embedded once-holed tori can
  share a boundary component). Let $C$ denote the surface obtained by removing all of these
  once-holed tori from $X$. Since the once-holed tori only have one boundary component and the
  surface $X$ is connected, $C$ is also connected. It has $k$ boundary components (corresponding to
  the $k$ once-holed tori). By additivity of the Euler characteristic, the genus of $C$ is $g-k$,
  and in particular $k \leq g$.

  \begin{figure}[h]
    \centering
    \includegraphics[scale=0.8]{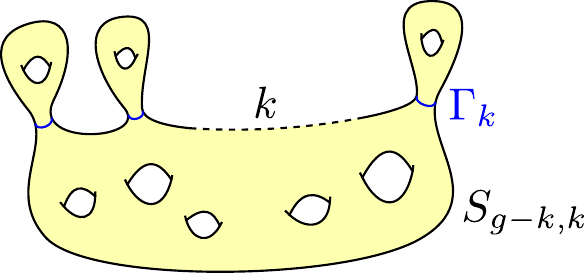}
    \caption{The multi-curve $\Gamma_k$ that separates $k$ once-holed tori off a
      surface of genus $g$.}
    \label{fig:tori}
  \end{figure}

  Therefore, for any $1 \leq k \leq g$, there is exactly one mapping-class-group orbit of family of
  $k$ once-holed tori in a surface of genus~$g$ (and none for $k > g$).  Let us take a
  representative $\Gamma_k$ of this orbit on the base surface $S_g$, as represented in
  \cref{fig:tori}.  We consider the function
  \begin{equation*}
    F_{a,k} :
    \begin{cases}
      \R_{\geq 0}^k & \rightarrow \R \\
      (x_1, \ldots,x_k) & \displaystyle \mapsto \frac{1}{k!} \prod_{i=1}^k \1{[0,a]}(x_i).
    \end{cases}
  \end{equation*}
  Then, the geometric function $X \mapsto F_{a,k}^{\Gamma_k}(X)$ defined by \eqref{eq:def_F_gamma}
  exactly counts the number of unordered families of $k$ embedded once-holed tori with boundary
  length $\leq a$. Hence,
  \begin{equation*}
    \Pwp{\ntor \geq k} = \Pwp{F_{a,k}^{\Gamma_k} \geq 1} \leq \Ewpo \big[ F_{a,k}^{\Gamma_k}\big]
  \end{equation*}
  by Markov's inequality. We compute the expectation of this geometric function
  using Mirzakhani's integration formula:
  \begin{equation}
    \label{e:int_p_tor_k}
    \Pwp{\ntor \geq k} 
    \leq \frac{1}{k!} \frac{1}{V_g}
    \int_{[0, a]^k} 
      V_{g-k,k}(x_1, \ldots, x_k) 
      \prod_{i=1}^k V_{1,1}(x_i) \, x_i \d x_i.
  \end{equation}
  We use once again \cite{naatanen1998} to express the volume $V_{1,1}(x_i)$. By
  \cref{e:increase_Vgn_bound} and \cite[Lemma 3.2]{mirzakhani2013}, there exists a constant $C>0$
  independent of $k$ such that
  \begin{equation*}
    \frac{V_{g-k,k}(x_1, \ldots, x_k)}{V_g}
    \leq \frac{C^k}{g^k} \exp \div{x_1 + \ldots + x_k}.
  \end{equation*}
  Then, \cref{e:int_p_tor_k} becomes
  \begin{equation*}
    \Pwp{\ntor \geq k}
    = \O{ \frac{C^k}{k!} \frac{1}{g^k}
      \brac*{\int_0^a \left(\frac{\pi^2}{12}+\frac{x^2}{48}\right) x \exp \div{x} \d x}^k}
    = \O{\frac{C^k}{k!}  \frac{a^{2k}}{g^k}},
  \end{equation*}
  provided we picked $a_0$ to be small enough.
  We deduce from the previous inequality that
  \begin{equation}
    \label{e:sum_k_tor}
    \sum_{k=2}^{+ \infty} \Pwp{\ntor \geq k}
    = \O{\sum_{k=2}^{+ \infty} \frac{C^k}{k!} \frac{a^{2k}}{g^k}}
    = \O{\frac{a^4}{g^2} \, e^{\frac{C a^{2}}{g}}} = \O{\frac{a^4}{g^2}}
  \end{equation}
  which is negligible compared to $\frac{a^2}{g}$ for small enough $a$ and large
  enough $g$.  Our claim then follows directly by putting together
  \cref{e:tor_exp,e:Pwp_vs_Ewp_tor,e:sum_k_tor}.
\end{proof}

\subsubsection{The upper bound}
\label{sec:upper-bound}

Let us now prove the upper bound part of \cref{lem:prob_small_eig}. It cannot be
done the same way as the lower bound. Indeed, we have used the fact that, by the
min-max principle, the presence of a small once-holed torus implies the presence
of a small eigenvalue. But the converse is not true, and the existence of a
small eigenvalue does not imply that the surface contains a small once-holed
torus. 
We shall therefore rather rely on another geometric quantity, the Cheeger
constant $h(X)$, defined by
\begin{equation*}
  h(X) := \inf_{A \sqcup B = X} \left\{ \frac{\ell(\partial A)}{\min (\area (A), \area(B))} \right\}
\end{equation*}
where the infimum is taken over all partitions $A \sqcup B$ of $X$ into two
smooth connected components. Cheeger's inequality \cite{cheeger1970} states that
\begin{equation}
  \label{eq:cheeger_ineq}
  \lambda_1(X) \geq \frac{h(X)^2}{4}
\end{equation}
and in particular, if $\lambda_1(X)$ is small, then $h(X)$ is too (the converse
is also true by Buser's inequality \cite{buser1982}).

The Cheeger constant $h(X)$ is a priori difficult to estimate for random Weil--Petersson surfaces,
because it does not depend only on geometric functions. In \cite[Section 4.5]{mirzakhani2013},
Mirzakhani defined a \emph{modified Cheeger constant} $H(X)$. It is defined the same way as~$h(X)$
is, except the infimum ranges over all partitions $A \sqcup B$ of $X$ into two connected components
such that $\partial A$ is a \emph{union of disjoint simple geodesics}. Mirzakhani proved in
\cite[Proposition 4.6]{mirzakhani2013} that the two Cheeger constants satisfy the inequality
\begin{equation}
  \label{eq:mirz_cheeger}
  \frac{H(X)}{1+H(X)} \leq h(X) \leq H(X)
\end{equation}
where the upper bound is trivial. In particular, $h(X)$ is small if and only if
$H(X)$ is.

We shall use the following probabilistic estimate on $H(X)$.

\begin{lem}
  \label{lem:mirz_prob_cheeger}
  There exists $c, a_0 > 0$ such that, for any $a \leq a_0$ and any large enough $g$, 
  \begin{equation}
    \label{eq:mirz_prob_cheeger}
    \Pwp{H(X) \leq a} \leq c \frac{a^2}{g} \cdot
  \end{equation}
\end{lem}

The upper bound of \cref{lem:prob_small_eig} then follows, because for any small
enough $a$ and large enough $g$,
\begin{equation*}
  \Pwp{\lambda_1 \leq a}
  \underset{\eqref{eq:cheeger_ineq}}{\leq} \Pwp{h(X) \leq 2 \sqrt{a}}
  \underset{\eqref{eq:mirz_cheeger}}{\leq} \Pwp{H(X) \leq \frac{2 \sqrt{a}}{1 - 2 \sqrt{a}}}
  \underset{\eqref{eq:mirz_prob_cheeger}}{\leq} c_2 \, \frac{a}{g} \cdot
\end{equation*}
\Cref{lem:mirz_prob_cheeger} comes as a consequence of the proof of \cite[Theorem
4.8]{mirzakhani2013}, although it is not stated as such. For the sake of self-containment, we repeat the
argument here.

\begin{proof}
  For a surface $X$, let $\Gamma$ be a multi-curve candidate to realizing the
  modified Cheeger constant $H(X)$. The mapping-class-group orbit of $\Gamma$ is entirely
  determined by its number of components, denoted by $k$, and the genera $g_1$
  and $g_2$ of the two connected components of $X \setminus \Gamma$. We pick the
  numbering so that $g_1 \leq g_2$, and note that by additivity of the Euler
  characteristic, we always have $g_1 + g_2 + k = g + 1$. For such a $g_1$ and
  $k$, we fix one multi-curve $\Gamma_{g_1,k}$ on the base surface $S_g$ of this
  topology. Then,
  \begin{equation}
    \label{e:decomp_proba_cheeger}
    \Pwp{H(X) \leq a}
    \leq \sum_{\substack{g_1+g_2+k = g+1 \\ 0 \leq g_1 \leq g_2 \\ g_1 + k > 1}} \Pwp{H_{g_1,k}(X) \leq a},
  \end{equation}
  where for any $g_1, k$, the function $H_{g_1, k}$ is defined by 
  \begin{equation}
    \label{e:Hkm_cheeger}
    H_{g_1,k}(X)
    :=
    \frac{\min \left\{ \ell_X(\Gamma) \, : \, \Gamma \in \orb_{g}(\Gamma_{g_1,k})\right\}}
    {2\pi(2g_1-2+k)} \cdot
  \end{equation}
  By Markov's inequality and Mirzakhani's integration formula,
  \begin{equation*}
    \Pwp{H_{g_1,k}(X) \leq a}
    \leq \frac{1}{V_g} \frac{1}{k!} \int_{\R_{>0}^k} 
    V_{g_1,k}(\x) V_{g_2,k}(\x) \,
    \1{[0,L_{g_1,k}(a)]} (x_1 + \ldots + x_k) \,
    \prod_{i=1}^k x_i \d x_i
  \end{equation*}
  for $L_{g_1,k}(a) := 2 \pi a (2g_1-2+k)$. This can be bounded using
  \cref{e:increase_Vgn_bound} and the fact that
  \begin{equation*}
    \forall L > 0, \quad \int_{\R_{>0}^k} \1{[0,L]}(x_1 + \ldots + x_k)
    \prod_{i=1}^k x_i \d x_i \leq \frac{L^{2k}}{(2k)!} 
  \end{equation*}
  and we obtain that
  \begin{equation*}
    \Pwp{H_{g_1,k}(X) \leq a}
    \leq \frac{1}{V_g}
    \frac{L_{g_1,k}(a)^{2k} e^{L_{g_1,k}(a)}}{k!(2k)!} V_{g_1,k} V_{g_2,k} \cdot
  \end{equation*}
  In order to sum over all values of $k$, we now set
  $g_i' = g_i + \floor{\frac k 2}$ and
  $n = k - 2\floor{\frac k 2} \in \{0, 1\}$. We also assume that
  $2\pi a_0 \leq 1$. By \cite[Lemma 3.2 (3)]{mirzakhani2013}, $V_{g_1,k}
  V_{g_2,k} = \O{V_{g_1',n} V_{g_2',n}}$ and hence
  \begin{equation*}
    \Pwp{H_{g_1,k}(X) \leq a}
    = \O{\frac{a^2}{V_g} \frac{(2g_1'-2+n)^{2k}}{k!(2k)!} \, e^{2\pi a (2g_1'-2+n)} \,V_{g_1',n} V_{g_2',n}}
  \end{equation*}
  with a constant independent of $k$ (note that we have kept a term $a^{2}$ in this estimate,
  granted by the fact that $k \geq 1$; this is the only difference with Mirzakhani's proof).  We
  have that, for any $x$,
  \begin{equation*}
    \sum_{k=1}^{+ \infty} \frac{x^{2k}}{k! (2k)!} = \O{\exp \paren*{\frac 32 x^{\frac 23}}}
  \end{equation*}
  and hence, \cref{e:decomp_proba_cheeger} yields
  \begin{equation*}
    \Pwp{H(X) \leq a}
    = \O{\frac{a^2}{V_g} \sum_{n =0}^1 \sum_{\substack{g_1' \leq g_2' \\ g'_1 + g'_2 = g+1-n}}
    e^{2\pi a(2 g_1'-2+n) + \frac 32 (2g_1'-2+n)^{\frac 23}}
    V_{g_1',n} V_{g_2',n}}.
  \end{equation*}
  In \cite[Corollary 3.7]{mirzakhani2013}, Mirzakhani proved that, provided that
  $4 \pi a_0 < 2 \log(2)$, the sum over all values of $g_1', g_2'$ is $\O{V_g/g}$
  for $n=0$ and $n=1$, which allows us to conclude.
\end{proof}

\subsection{Contradiction if $\ell f_1^{\mathrm{all}}(\ell)$ is Friedman--Ramanujan}
\label{sec:proof-crefprp:sm}

Let us now prove that, if the function $\ell \mapsto \ell f_1^{\mathrm{all}}(\ell)$ was
Friedman--Ramanujan in the weak sense, then we would be able to prove spectral estimates that are
too good to be true.

\begin{lem}
  \label{lem:FR_contradiction}
  If $\ell \mapsto \ell f_1^{\mathrm{all}}(\ell)$ is Friedman--Ramanujan in the weak sense, then for
  any small enough $\delta > 0$, any large enough $g$,
  \begin{equation*}
    \Pwp{\delta \leq \lambda_1 \leq \frac{5}{72}} = \O[\delta]{g^{-\frac 54}}.
  \end{equation*}
\end{lem}

\begin{proof}
  Let us assume that $\ell \mapsto \ell f_1^{\mathrm{all}}(\ell)$ is a Friedman--Ramanujan function
  in the weak sense, i.e. that this function belongs in $\cF_w^{m,c}$ for some integers $m$, $c$. We
  may assume w.l.o.g. that $m \geq 1$.  We shall use the trace method developed in
  \cref{sec:canc-selb-trace}, with the parameters:
  \begin{equation*}
    A = 6, \quad
    \alpha = \frac{5}{12} < \frac 12, \quad
    \epsilon = \frac{1}{144}
  \end{equation*}
    and $m$ given by the Friedman--Ramanujan hypothesis for $\ell f_1^{\mathrm{all}}(\ell)$.
  We observe that $1/4 - \alpha^2 - \epsilon = 5/72$ and
  $(\alpha+\epsilon)A > \alpha A = 5/2$, and therefore
  \cref{lem:selberg_reformulated} implies that
  \begin{equation}
    \label{eq:proba_estimate_contradiction}
    \Pwp{\delta \leq \lambda_1 \leq \frac{5}{72}} 
    = \O[\delta]{g^{-\frac 52} \avb{\ell e^{- \frac \ell 2} \, \mathcal{D}^m h_L(\ell)}
      + g^{-\frac 3 2}}. 
  \end{equation}
  We then use the density-writing of the average $\av{\cdot}$, to write
  \begin{equation}
    \label{eq:proof_contr_dens_write}
    \avb{\ell \, e^{- \frac \ell 2} \, \mathcal{D}^m h_L(\ell)}
    = \int_{0}^{+ \infty} e^{- \frac \ell 2} \, \mathcal{D}^m h_L(\ell)
    \frac{\ell V_{g}^{\mathrm{all}}(\ell)}{V_g}  \1{[0,L]}(\ell) \d \ell
  \end{equation}
  because $\D^m h_L$ is identically equal to zero outside of $[0,L]$.  By our asymptotic expansion
  result for the sum over all geodesics, \cref{thm:existence_asym_all}, at the second order,
  using a fixed value $\epsilon' < 1/24$,
  \begin{equation}
    \label{eq:proof_contr_exp_all}
    \frac{\ell \, V_{g}^{\mathrm{all}}(\ell)}{V_g} \1{[0,L]}(\ell)
    = F_{g,1}^{\mathrm{all}}(\ell) + \Ow{\frac{\ell \, e^{(1+\epsilon') \ell}}{g^2}}
    = F_{g,1}^{\mathrm{all}}(\ell) + \Ow{\frac{e^{\frac{25\ell}{24}}}{g^2}}
  \end{equation}
  for
  $F_{g,1}^{\mathrm{all}}(\ell) := \ell f_0^{\mathrm{all}}(\ell) + g^{-1} \ell
  f_1^{\mathrm{all}}(\ell)$. When we inject \eqref{eq:proof_contr_exp_all} into
  \eqref{eq:proof_contr_dens_write}, we obtain
  \begin{equation}
    \label{eq:apply_est_order_2}
    \avb{\ell \, e^{- \frac \ell 2} \, \mathcal{D}^m h_L(\ell)}
    = \int_0^{L} F_{g,1}^{\mathrm{all}}(\ell) \, e^{- \frac \ell 2} \, \D^m h_L(\ell)  \d \ell
    + \O{g^{\frac 54}}
  \end{equation}
  because the remainder is bounded by
  \begin{equation*}
    \O{\frac{\norminf{\D^m h_L}}{g^2} e^{\frac{25 L}{24} - \frac{L}{2}}} =
    \O{\frac{e^{\frac{13L}{24}}}{g^2}} = \O{g^{\frac{13}{4}-2}} = \O{g^{\frac 54}}.
  \end{equation*}
  We know that $\ell f_0^{\mathrm{all}} (\ell) = 4 \sinh^2 \div{\ell} \in \cF^{1,0}$, and therefore
  our hypothesis implies that $F_{g,1}^{\mathrm{all}}$ is an element of $\FR_w^{m,c}$. We notice
  that $\|F_{g,1}^{\mathrm{all}}\|_{\cF^{m,c}}^w = \O{1}$ uniformly in $g$. As a consequence, by
  \cref{prp:int_FR}, 
  \begin{equation*}
    \int_0^{L} F_{g,1}^{\mathrm{all}}(\ell) \, e^{- \frac \ell 2} \, \D^m h_L(\ell)  \d \ell
    = \O{L^{c+1}}.
  \end{equation*}
  Together with \cref{eq:apply_est_order_2}, this implies that
  $\avb{\ell \, e^{- \frac \ell 2} \, \mathcal{D}^m h_L(\ell)} = \O{g^{\frac
      54}}$.  The conclusion then follows directly from
  \cref{eq:proba_estimate_contradiction}.
\end{proof}

We conclude with the proof of \cref{prp:non_FR}.

\begin{proof}
  Let us assume by contradiction that $\ell \mapsto \ell f_1^{\mathrm{all}}(\ell)$ is a
  Friedman--Ramanujan function in the weak sense.  By \cref{lem:FR_contradiction}, for any small
  enough $\delta > 0$,
  \begin{equation*}
    \Pwp{\delta \leq \lambda_1 \leq \frac{5}{72}} \leq C_\delta g^{- \frac{5}{4}}.
  \end{equation*}
  This is in particular also true if we replace $5/72$ by $a := \min(a_0, 5/72)$, where 
  $a_0$ is the constant from \cref{lem:prob_small_eig}. Applying this result yields
  \begin{equation*}
    \Pwp{\delta \leq \lambda_1 \leq a}
    = \Pwp{\lambda_1 \leq a} - \Pwp{\lambda_1 < \delta}
    \geq c_1 \frac{a^2}{g} - c_2 \frac{\delta}{g} 
  \end{equation*}
  for some constants $c_1, c_2 > 0$. We take $\delta$ to be a fixed number smaller than
  $a^2 c_1/(2c_2)$. Then, we obtain that, for any large enough $g$,
  \begin{equation*}
    c_1 \frac{a^2}{2} 
    \leq  g \, \Pwp{\delta \leq \lambda_1 \leq a}
    \leq C_{\delta}g^{-\frac 14}
  \end{equation*}
  which is a contradiction as $g \rightarrow + \infty$.
\end{proof}


\bibliographystyle{plain}
\bibliography{bibliography}

\begin{thebibliography}{10}

\bibitem{alon1986}
Noga Alon.
\newblock Eigenvalues and expanders.
\newblock {\em Combinatorica}, 6(2):83--96, 1986.

\bibitem{anantharaman2022}
Nalini Anantharaman and Laura Monk.
\newblock A high-genus asymptotic expansion of {{Weil}}\textendash{{Petersson}}
  volume polynomials.
\newblock {\em Journal of Mathematical Physics}, 63(4):043502, 2022.

\bibitem{Moebius}
Nalini Anantharaman and Laura Monk.
\newblock A {Moebius} inversion formula to discard tangled hyperbolic surfaces.
\newblock {\em arXiv:2401.01601}, 2023.

\bibitem{anantharaman_29}
Nalini Anantharaman and Laura Monk.
\newblock The spectral gap of a typical hyperbolic surface is greater than $2/9
  - \epsilon$.
\newblock {\em to appear on arxiv.}, 2023.

\bibitem{bergeron2016}
Nicolas Bergeron.
\newblock {\em The {{Spectrum}} of {{Hyperbolic Surfaces}}}.
\newblock Universitext. {Springer International Publishing}, {Cham}, 2016.

\bibitem{bordenave2020}
Charles Bordenave.
\newblock A new proof of {{Friedman}}'s second eigenvalue theorem and its
  extension to random lifts.
\newblock {\em Annales Scientifiques de l'\'Ecole Normale Sup\'erieure.
  Quatri\`eme S\'erie}, 53(6):1393--1439, 2020.

\bibitem{brooks2004}
Robert Brooks and Eran Makover.
\newblock Random construction of riemann surfaces.
\newblock {\em Journal of Differential Geometry}, 68(1):121--157, 2004.

\bibitem{buser1982}
Peter Buser.
\newblock A note on the isoperimetric constant.
\newblock {\em Annales Scientifiques de l'\'Ecole Normale Sup\'erieure.
  Quatri\`eme S\'erie}, 15(2):213--230, 1982.

\bibitem{buser1992}
Peter Buser.
\newblock {\em Geometry and {{Spectra}} of {{Compact Riemann Surfaces}}}.
\newblock {Birkh\"auser}, {Boston}, 1992.

\bibitem{buser1988}
Peter Buser, Marc Burger, and Jozef Dodziuk.
\newblock Riemann surfaces of large genus and large
  {$\Lambda$}{\textsubscript{1}}.
\newblock In {\em Geometry and Analysis on Manifolds}, volume 1339, pages
  54--63. {Springer, Berlin, Heidelberg}, 1988.

\bibitem{cheeger1970}
Jeff Cheeger.
\newblock A lower bound for the smallest eigenvalue of the {{Laplacian}}.
\newblock In {\em Problems in Analysis}, pages 195--199. {Princeton University
  Press}, 1970.

\bibitem{friedman2003}
Joel Friedman.
\newblock A proof of {{Alon}}'s second eigenvalue conjecture.
\newblock {\em Proceedings of the thirty-fifth annual ACM symposium on Theory
  of computing}, pages 720--724, 2003.

\bibitem{gilmore2021}
Clifford Gilmore, Etienne Le~Masson, Tuomas Sahlsten, and Joe Thomas.
\newblock Short geodesic loops and {$\mathrm{L}^p$} norms of eigenfunctions on
  large genus random surfaces.
\newblock {\em Geometric and Functional Analysis}, 31(1):62--110, 2021.

\bibitem{golubev2019}
Konstantin Golubev and Amitay Kamber.
\newblock Cutoff on hyperbolic surfaces.
\newblock {\em Geometriae Dedicata}, 203(1):225--255, 2019.

\bibitem{graaf1997}
Maurits Graaf and Alexander Schrijver.
\newblock Making {{Curves Minimally Crossing}} by {{Reidemeister Moves}}.
\newblock {\em Journal of Combinatorial Theory, Ser. B}, 70:134--156, 1997.

\bibitem{guth2011}
Larry Guth, Hugo Parlier, and Robert Young.
\newblock Pants decompositions of random surfaces.
\newblock {\em Geometric and Functional Analysis}, 21(5):1069--1090, 2011.

\bibitem{hide2022a}
Will Hide.
\newblock Spectral {{Gap}} for {{Weil}}\textendash{{Petersson Random Surfaces}}
  with {{Cusps}}.
\newblock {\em International Mathematics Research Notices}, page rnac293, 2022.

\bibitem{hide2021}
Will Hide and Michael Magee.
\newblock Near optimal spectral gaps for hyperbolic surfaces.
\newblock {\em Annals of Mathematics}, 198(2):791--824, 2023.

\bibitem{hide2022}
Will Hide and Joe Thomas.
\newblock Short geodesics and small eigenvalues on random hyperbolic punctured
  spheres.
\newblock {\em arXiv:2209.15568}, 2022.

\bibitem{huber1974}
Heinz Huber.
\newblock {\"Uber den ersten Eigenwert des Laplace-Operators auf kompakten
  Riemannschen Fl\"achen}.
\newblock {\em Commentarii mathematici Helvetici}, 49:251--259, 1974.

\bibitem{katok1992}
Svetlana Katok.
\newblock {\em Fuchsian {{Groups}}}.
\newblock {University of Chicago Press}, 1992.

\bibitem{kim2003}
Henry~H. Kim.
\newblock Functoriality for the exterior square of {$\mathrm{GL}_4$} and the
  symmetric fourth of {$\mathrm{GL}_2$}.
\newblock {\em Journal of the American Mathematical Society}, 16(1):139--183,
  2003.

\bibitem{lipnowski2021}
Michael Lipnowski and Alex Wright.
\newblock Towards optimal spectral gaps in large genus.
\newblock {\em Ann. Probab.}, 52(2):545--575, 2024.

\bibitem{louder2022}
Larsen Louder, Michael Magee with Appendix by~Will Hide, and Michael Magee.
\newblock Strongly convergent unitary representations of limit groups.
\newblock {\em Journal of Functional Analysis (to appear)}, 2025.

\bibitem{lubotzky1988}
A.~Lubotzky, R.~Phillips, and P.~Sarnak.
\newblock Ramanujan graphs.
\newblock {\em Combinatorica}, 8(3):261--277, 1988.

\bibitem{magee2020b}
Michael Magee.
\newblock Letter to {{Bram Petri}}, 2020.

\bibitem{magee2022}
Michael Magee, Fr{\'e}d{\'e}ric Naud, and Doron Puder.
\newblock A random cover of a compact hyperbolic surface has relative spectral
  gap {$\frac{3}{16} - \epsilon$}.
\newblock {\em Geometric and Functional Analysis}, 32(3):595--661, 2022.

\bibitem{mirzakhani2007}
Maryam Mirzakhani.
\newblock Simple geodesics and {{Weil--Petersson}} volumes of moduli spaces of
  bordered {{Riemann}} surfaces.
\newblock {\em Inventiones Mathematicae}, 167(1):179--222, 2007.

\bibitem{mirzakhani2013}
Maryam Mirzakhani.
\newblock Growth of {{Weil--Petersson}} volumes and random hyperbolic surfaces
  of large genus.
\newblock {\em Journal of Differential Geometry}, 94(2):267--300, 2013.

\bibitem{mirzakhani2019}
Maryam Mirzakhani and Bram Petri.
\newblock Lengths of closed geodesics on random surfaces of large genus.
\newblock {\em Commentarii Mathematici Helvetici}, 94(4):869--889, 2019.

\bibitem{mirzakhani2015}
Maryam Mirzakhani and Peter Zograf.
\newblock Towards large genus asymptotics of intersection numbers on moduli
  spaces of curves.
\newblock {\em Geometric and Functional Analysis}, 25(4):1258--1289, 2015.

\bibitem{monk2021}
Laura Monk.
\newblock {\em Geometry and Spectrum of Typical Hyperbolic Surfaces}.
\newblock Theses, Universit\'e de Strasbourg, 2021.

\bibitem{monk2022}
Laura Monk.
\newblock Benjamini\textendash{{Schramm}} convergence and spectra of random
  hyperbolic surfaces of high genus.
\newblock {\em Analysis \& PDE}, 15(3):727--752, 2022.

\bibitem{monk2021a}
Laura Monk and Joe Thomas.
\newblock The tangle-free hypothesis on random hyperbolic surfaces.
\newblock {\em International Mathematics Research Notices}, rnab160, 2021.

\bibitem{naatanen1998}
Marjatta N{\"a}{\"a}t{\"a}nen and Toshihiro Nakanishi.
\newblock Weil--{{Petersson}} areas of the moduli spaces of tori.
\newblock {\em Results in Mathematics}, 33(1-2):120--133, 1998.

\bibitem{nie2023}
Xin Nie, Yunhui Wu, and Yuhao Xue.
\newblock Large genus asymptotics for lengths of separating closed geodesics on
  random surfaces.
\newblock {\em Journal of Topology}, 16(1):106--175, 2023.

\bibitem{parlier2021}
Hugo Parlier, Yunhui Wu, and Yuhao Xue.
\newblock The simple separating systole for hyperbolic surfaces of large genus.
\newblock {\em Journal of the Institute of Mathematics of Jussieu}, pages
  1--10, 2021.

\bibitem{ratner1987}
Marina Ratner.
\newblock The rate of mixing for geodesic and horocycle flows.
\newblock {\em Ergodic Theory and Dynamical Systems}, 7(2):267--288, 1987.

\bibitem{rudnick2022}
Ze{\'e}v Rudnick.
\newblock Goe statistics on the moduli space of surfaces of large genus.
\newblock {\em Geometric and Functional Analysis}, 33:1581 -- 1607, 2022.

\bibitem{selberg1956}
Atle Selberg.
\newblock Harmonic analysis and discontinuous groups in weakly symmetric
  {{Riemannian}} spaces with applications to {{Dirichlet}} series.
\newblock {\em The Journal of the Indian Mathematical Society}, 20:47--87,
  1956.

\bibitem{selberg1965}
Atle Selberg.
\newblock On the estimation of {{Fourier}} coefficients of modular forms.
\newblock In {\em Proc. {{Sympos}}. {{Pure Math}}., {{Vol}}. {{VIII}}}, pages
  1--15. {Amer. Math. Soc., Providence, R.I.}, 1965.

\bibitem{weil1958}
Andr{\'e} Weil.
\newblock On the moduli of {{Riemann}} surfaces.
\newblock In {\em \OE uvres {{Scientifiques}} - {{Collected Papers
  II}}:1951-1964}, pages 379--389. {Springer-Verlag}, {Berlin, Heidelberg},
  1958.

\bibitem{wolpert1981}
Scott Wolpert.
\newblock An elementary formula for the {{Fenchel--Nielsen}} twist.
\newblock {\em Commentarii Mathematici Helvetici}, 56(1):132--135, 1981.

\bibitem{wright2020}
Alex Wright.
\newblock A tour through {{Mirzakhani}}'s work on moduli spaces of {{Riemann}}
  surfaces.
\newblock {\em Bulletin of the American Mathematical Society}, 57(3):359--408,
  2020.

\bibitem{wu2022}
Yunhui Wu and Yuhao Xue.
\newblock Random hyperbolic surfaces of large genus have first eigenvalues
  greater than {$\frac{3}{16}-\epsilon$}.
\newblock {\em Geometric and Functional Analysis}, 32(2):340--410, 2022.

\bibitem{wu2022a}
Yunhui Wu and Yuhao Xue.
\newblock Prime geodesic theorem and closed geodesics for large genus.
\newblock {\em J. Eur. Math. Soc. (to appear)}, 2025.

\end{thebibliography}

\end{document}